\documentclass[times,sort&compress,3p]{elsarticle}
\journal{Journal of Multivariate Analysis}
\usepackage[labelfont=bf]{caption}

\usepackage{amsmath,amsfonts,amssymb,amsthm,booktabs,color,epsfig,graphicx,hyperref,url}
\usepackage{float}
\usepackage{graphicx}
\usepackage{caption}
\usepackage{adjustbox}
\usepackage{pdfpages} 
\usepackage{pgffor} 

\makeatletter
\AtBeginDocument{\let\LS@rot\@undefined}
\makeatother

\def\supplementfilename{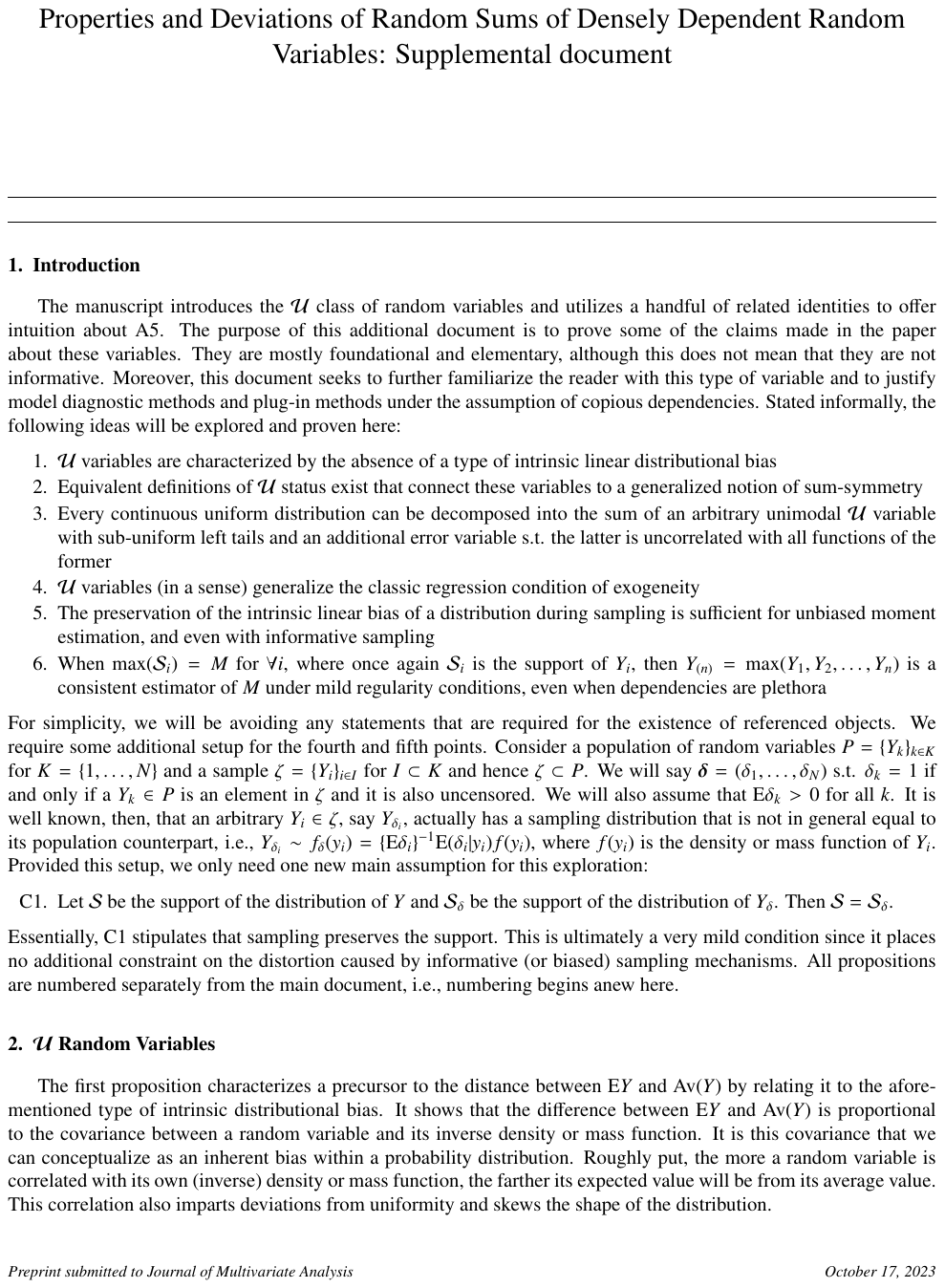}

\pdfximage{\supplementfilename}
\def\numbersupplementpages{\the\pdflastximagepages}

\newif\ifarXiv
\arXivtrue 

\theoremstyle{plain}
\newtheorem{theorem}{Theorem}

\newtheorem{proposition}{Proposition}
\newtheorem{lemma}{Lemma}
\newtheorem{corollary}{Corollary}

\theoremstyle{definition}
\newtheorem{definition}{Definition}

\newtheorem{example}{Example}

\begin{document}

\begin{frontmatter}

\title{Properties and Deviations of Random Sums of Densely Dependent Random Variables}

\author[1]{Shane Sparkes}
\author[1]{Lu Zhang}

\address[1]{University of Southern California, Department of Population and Public Health Sciences, Los Angeles}

\cortext[mycorrespondingauthor]{Shane Sparkes. Email address: \url{sgugliel@usc.edu}}

\begin{abstract}
A classical problem of statistical inference is the valid specification of a model that can account for the statistical dependencies between observations when the true structure is dense, intractable, or unknown. To address this problem, a new variance identity is presented, which is closely related to the Moulton factor. This identity does not require the specification of an entire covariance structure and instead relies on the choice of two summary constants. Using this result, a weak law of large numbers is also established for additive statistics and common variance estimators under very general conditions of statistical dependence. Furthermore, this paper proves a sharper version of Hoeffding's inequality for symmetric and bounded random variables under these same conditions of statistical dependence. Put otherwise, it is shown that, under relatively mild conditions, finite sample inference is possible in common settings such as linear regression, and even when every outcome variable is statistically dependent with all others. All results are extended to estimating equations. Simulation experiments and an application to climate data are also provided.   
\end{abstract}


\end{frontmatter}

\section{Introduction\label{sec:1}}

Popular methods for statistical inference model probabilistic dependencies as limited and schematic in nature. The latter notion is used to justify asymptotic normality, while the former reduces an unknowable picture to one that is tame and mathematically pliable. However, in the now everyday words of George Box, "... all models are wrong" \citep{box1987empirical}. In many research contexts, there is little reason to believe that the system of statistical dependencies governing the joint distribution of outcome variables behaves in accordance with a tractable sequence, or that it is sparse even conditionally. Sociological, climate, or clinical settings are but a few examples. Box of course went on to state that 'some are useful.' However, this final addendum, while by construction irrefutable, requires great caution when cited besides models of statistical uncertainty. Invalid models for the expected values of outcome variables still possess salient interpretations as approximations. On the other hand, an invalid model for statistical dependence will furnish untrue statements about error, or statements in general that are exceptionally vulnerable to doubt. This undercuts the cogency of knowledge construction.

This paper addresses this problem for additive statistics, which we now define. Let $I = \{1, \ldots, n \}$ be an indexing set for a sample $\zeta = \{ Y_i \}_{i \in I}$ of random variables. Furthermore, let $\{ w_i \}_{i \in I}$ be any set of constants. We define $S_n = \sum_{i=1}^n w_i Y_i$ as an additive statistic, although for simplicity, we will often set $w_i =1$ without loss of generality (WLOG) since we can simply say $Z_i = w_i Y_i$ and reason about $\{ Z_i\}_{i \in I}$. It is important to note that $Y_i$ $(Z_i)$ is also general and can be any measurable function of $k \in \mathbb{N}$ random variables. Ultimately, we are interested in establishing some basic properties and behaviors of $\mathbf{S_n}$, a vector of random sums, under very general but unfavorable conditions of dependence. We do so without any particular theory of how this looks, which makes these results widely applicable. 

Lots of work exists on the topic of dependence and additive statistics. The generalized least squares approach encapsulates common classes of estimators. Ultimately, this method replaces $\text{Var}(\mathbf{S_n})$ with a model $\tilde{\boldsymbol{\Sigma}}$ in an attempt to induce the conditions of the Gauss-Markov theorem \citep{aitken1936iv, amemiya1985generalized}. Generalized linear models and the generalized estimating equation approach of \citet{liang1986longitudinal} produce estimators related to this type, as do hierarchical linear models \citep{nelder1972generalized, gardiner2009fixed}. The approach of Liang and Zeger, however, also makes use of cluster-robust standard errors in conjunction with specified covariance models \citep{ziegler1998generalised, zorn2001generalized}. These strategies posit that a user-specified partition of the sample results in independent clusters, the variances of which can be identified. A thorough review of the theory behind cluster-robust variance estimation is available elsewhere \citep{mackinnon2023cluster}. Spatial and time series methods are also important in this universe. They adopt objects---such as variograms or auto-regressive weight matrices---to facilitate covariance estimation under the supposition of largely localized dependencies \citep{cressie2015statistics, kedem2005regression, anselin2009spatial}. The details of these methods are beyond this paper's scope. In essence, they share in one or two pivotal facts, however: they usually posit that the average number of dependencies in a sample is bounded by a constant and that the researcher has knowledge of a partition of $\zeta$ that produces $K(n) \to \infty$ independent clusters, or (2) they replace $\text{Var}(\mathbf{S_n})$ with a blueprint that is salient, but simple and mathematically convenient. Overall, these strategies set a grand majority of the $n \choose 2$ covariance parameters to zero \textit{a priori} in protestation of real conditions. Inference also depends on mixing conditions or concepts such as $m$-dependence, which justify a central limit theorem for dependent random variables \citep{withers1981central, berk1973central}.

Researchers have also utilized the concentration of measure phenomenon as a tool for inference \citep{ledoux2001concentration}. Hoeffding's inequality is a classical result in this domain, as are others \citep{hoeffding1994probability, bennett1962probability}. Most formulations rely on the assumption of mutual independence \citep{janson2011random, boucheron2003concentration, talagrand1996new}. Many, however, also allow for weak or local dependence conditions to exist \citep{daniel2014concentration, kontorovich2008concentration, gotze2019higher}. For instance, Hoeffding's inequality applies to sums of negatively associated random variables and to additive statistics that can be decomposed into smaller independent sums \citep{wajc2017negative, janson2004large}. These results are invaluable for theoretically bounding the tail probabilities of random sums under a wider range of circumstances. Unfortunately, however, most of these inequalities still rely upon restricted dependency pictures that do not fit the concerns of this paper.

The main contribution of this manuscript is to show that finite sample inference is possible for additive statistics of bounded random variables without a detailed dependency model, and even when \textit{all} measured outcome variables are statistically dependent and no central limit theorem applies. Essentially, this is done by showing that a sharper version of Hoeffding's inequality still applies to $\mathbf{S_n} - \text{E}\mathbf{S_n}  = \boldsymbol{\epsilon}$ when $\boldsymbol{\epsilon}$ is densely dependent, insofar as each $\epsilon_i$ behaves in accordance with some symmetric, unimodal, but not necessarily identical probability law. No other restrictions are necessarily placed on the marginal or joint probability functions. While still non-trivial, this contribution is useful since random errors of this type often surface in regression settings. Additionally, Hoeffding's inequality possesses a closed form that is accessible to working statisticians. The paper's secondary contribution is a novel variance identity that is useful for analyzing the properties of $\mathbf{S_n}$ under very general and terrible conditions. Extending these results to estimating equations is this paper's tertiary contribution.

Section~\ref{sec:2} introduces some key definitions and the aforementioned variance identity, while Section~\ref{sec:3} uses it to prove a weak law of large numbers (WLLN) for additive statistics under the circumstances identified, and to explore the behavior of cluster-robust variance estimators in these same settings. Section~\ref{sec:4} returns to our main theme of statistical inference in the face of dense, unknown, and intractable dependency structures. It defines a new type of random variable and uses this definition to prove our main results. Following this, Section~\ref{sec:5} makes use of the work of \citet{jennrich1969asymptotic}, \citet{yuan1998asymptotics}, and \citet{hall2005generalized} to extend the statements of the previous sections to estimating equations. The last part of this paper before the conclusion---Section~\ref{sec:6}---demonstrates the value of our approach with a set of simulation experiments that mimic some worst-case dependency scenarios with linear estimators. After this is accomplished, the association between global changes in temperature and carbon dioxide levels is estimated to demonstrate the utility of the approach.

\section{A Novel Variance-Covariance Identity for Additive Statistics \label{sec:2}}

We now introduce some important definitions. Say $\mathbf{Y}^{\top} = ( Y_1, \ldots, Y_n )$ is a $1 \times n$ random vector s.t. $\text{E}Y_i^2 < \infty$ for $\forall i$ and $\mathbf{w} \in \mathbb{R}^{p \times n}$ is a matrix of constants. Furthermore, as is tradition, say $\text{Cov}(Y_i, Y_j) = \sigma_{i, j} = \text{E}Y_i Y_j - \text{E}Y_i\text{E}Y_j$ and hence $\text{Var}(Y_i) = \sigma_{i, i} = \sigma_i^2$. Graphs are important for this exploration. A graph $\text{G} = (\text{V}, \text{L})$ is constituted by a set of nodes V and a set of lines, L, that connect them. It is undirected if $\text{e}_{i, j} \in \text{L}$ implies that $\text{e}_{j, i} \in \text{L}$; otherwise it is directed. Here, only undirected graphs will be of interest. Also, note that $I = \{1, \ldots, n \}$ can now be seen as a node set. The definition presented next is central and allows for the construction of the variance identity.

\begin{definition}[Linear Dependency Graph]
Let $\mathcal{L} = (I, \text{L})$ be a graph with a node set $I$   w.r.t. $\{ Y_i \}_{i \in I}$ and a set of edges $L$ between them. Then $\text{e}_{i, j} \in \text{L}$ if and only if $\sigma_{i, j} \neq 0$.
\end{definition}

It is also relevant to know that the degree of a node is defined as the sum of its existent links. Denote this function as $d(i) = \sum^{n-1}_{j=1} 1_{e_{i,j} \in L}$, where $1_{e_{i,j} \in L}$ is an indicator function, and similarly denote the mean degree as $\mu_n = n^{-1} \sum_{i=1}^{n} d(i)$. In the context of this paper, $\mu_n$ is equal to the mean number of random variables that a typical random variable is correlated with in the sample. Pertinently, each $1_{e_{i,j}}$ will be treated as a non-stochastic function, conditional on the realization of $\{ Y_i \}_{i \in I}$ as a sample of random variables, unless otherwise noted.

A few more definitions are necessary. The $\odot$ symbol will signify the Hadamard product, which is the component-by-component multiplication of two matrices. Letting $\mathbf{w}_s$ denote the $s$th row of $\mathbf{w}$, consider $\bar{\sigma}_{r, t} = |L|^{-1} \sum^{|L|}_{i < j} w_{r, i} w_{t, j} \sigma_{i, j} $ w.r.t. the statistics $\mathbf{w}_r \mathbf{Y}$ and $\mathbf{w}_t \mathbf{Y}$. This value will be called an average non-zero covariance term, while $\phi_{r, t} = \{n^{-1} \sum_{i=1}^n w_{r, i} w_{t, i} \sigma_i^2 \}^{-1} \cdot \bar{\sigma}_{r, t}$ will be called an average correlation. Ultimately, three matrices will also be required for the identity: $\mathbf{G}, \mathbf{C}$, and $\mathbf{V}$. The first two matrices are defined from the previously specified quantities: $\mathbf{C}^{p \times p} = (\bar{\sigma}_{i, j})$ and $\mathbf{G} = \boldsymbol{1} + \mu_n \boldsymbol{\phi}$, where $\boldsymbol{1}$ is a $p \times p$ matrix of ones and $\boldsymbol{\phi}^{p \times p} = (\phi_{i, j})$. Finally, $\mathbf{V}$ is a diagonal matrix s.t. its diagonal is equal to the diagonal of $\text{Var}(\mathbf{Y})$. This last matrix is recognizable as $\text{Var}(\mathbf{Y})$ under the counterfactual assumption of mutual independence.

\begin{proposition}
Let $\mathbf{Y} = ( Y_1, \ldots, Y_n )^{\top}$ be a $n \times 1$ random vector s.t. $\text{E}Y_i^2 < \infty$ for $\forall i$ and $\mathbf{w} \in \mathbb{R}^{p \times n}$ is a matrix of constants. Then $\text{Var}(\mathbf{w}\mathbf{Y}) = \mathbf{w} \mathbf{V} \mathbf{w}^{\top} \odot \mathbf{G} =  \mathbf{w} \mathbf{V} \mathbf{w}^{\top} + n \mu_n \mathbf{C}$.
\end{proposition}

\begin{proof}
Let $s, t$ be arbitrary indexes from $\{1, \ldots, p \}$. Then $\mathbf{w}_s \mathbf{Y} = \sum_{i=1}^n w_{s, i} Y_i$ and $\mathbf{w}_t \mathbf{Y} = \sum_{i=1}^n w_{t, i} Y_i$. It then follows that $\text{Cov}(\mathbf{w}_s \mathbf{Y}, \mathbf{w}_t \mathbf{Y}) = \sum_{i=1}^n w_{s, i} w_{t, i} \sigma_i^2 + \sum_{i \neq j} w_{s, i} w_{t, j} \sigma_{i, j}$. From here, consider the linear dependency graph $\mathcal{L} = (I, L)$ w.r.t. $\{ Y_i \}_{i \in I}$ as previously defined. Then:
\begin{align*}
 \sum_{i \neq j} w_{s, i} w_{t, i} \sigma_{i, j} = 2  \sum_{i < j} w_{s, i} w_{t, j} \sigma_{i, j} = 2 \cdot ( \sum^{|L|}_{i < j} w_{s, i} w_{t, j} \sigma_{i, j} + 0 ) = 2 |L| \cdot \bar{\sigma}_{s, t} = n \mu_n \cdot  \bar{\sigma}_{s, t}
\end{align*}
The first equality follows from the definition of $\mathcal{L}$, while the fourth follows from the Handshake lemma. Hence, $\text{Cov}(\mathbf{w}_s \mathbf{Y}, \mathbf{w}_t \mathbf{Y}) = (\mathbf{w} \mathbf{V} \mathbf{w}^{\top})_{s, t} +  n \mu_n \mathbf{C}_{s, t}$. It then follows that $\text{Var}(\mathbf{w}\mathbf{Y}) =  \mathbf{w} \mathbf{V} \mathbf{w}^{\top} + n \mu_n \mathbf{C}$ since $s, t$ were arbitrary.

Now, note that  $\text{Cov}(\mathbf{w}_s \mathbf{Y}, \mathbf{w}_t \mathbf{Y}) = \sum_{i=1}^n w_{s, i} w_{t, i} \sigma_i^2 + n\mu_n \cdot \bar{\sigma}_{s, t} = \{1 + n \mu_n \cdot (\sum_{i=1}^n w_{s, i} w_{t, i} \sigma_i^2)^{-1} \bar{\sigma}_{s, t} \} \sum_{i=1}^n w_{s, i} w_{t, i} \sigma_i^2 = \{1 + \mu_n \phi_{s, t} \} \sum_{i=1}^n w_{s, i} w_{t, i} \sigma_i^2$. Again, since $s, t$ were arbitrary, it then follows that  $\text{Var}(\mathbf{w}\mathbf{Y}) = \mathbf{w} \mathbf{V} \mathbf{w}^{\top} \odot \mathbf{G}$. 
\end{proof}
The utility of Proposition 1 is that it summarizes the impact of an unknowable and inestimable system of statistical dependencies on the variance of an additive statistic with two summary constants that are more defensibly specified or bounded. Although intuition might be lacking as to how some environment dynamically acts upon $\{ Y_i\}_{i \in I}$ to produce $\text{Var}(\mathbf{Y})$, this might not be the case for $\mu_n$ or the diagonal elements of $\boldsymbol{\phi}$. Prior beliefs or knowledge might exist pertaining to these values. Although this also suggests a possible Bayesian route where these values are conceptualized as draws from prior distributions to model uncertainty, this road is not pursued in this paper.

Moreover, recall that a Moulton factor is an expression for the variance inflation caused by intra-cluster correlation \citep{moulton1986random}. They typically have the form $\gamma_{s, s} = (\mathbf{w}\mathbf{V}\mathbf{w}^{\top})^{-1}_{s, s} \text{Var}(\mathbf{w} \mathbf{Y})_{s, s}$. From Proposition 1, we can see that $\gamma_{s, s} = \mathbf{G}_{s, s}$. However, Proposition 1 is more general since Moulton factors are often derived under specific modeling constraints. This correspondence can also be seen from rearranging an expression from Proposition 1. For instance, say $\boldsymbol{\Gamma} = (n^{-1} \cdot \mathbf{w}\mathbf{V}\mathbf{w}^{\top})^{-1} \mathbf{C}$. Then $\text{Var}(\mathbf{w} \mathbf{Y}) =  \mathbf{w}\mathbf{V}\mathbf{w}^{\top} \{\boldsymbol{1}^{p \times p} + \mu_n \boldsymbol{\Gamma} \}$, which is also recognizable as an even more general form of a Moulton style covariance factor \citep{moulton1990illustration}.

We will prove three bounds on the variances of simpler additive statistics, i.e., for $\mathbf{w}_s = \boldsymbol{1}^{1 \times n}$. These bounds can be be useful for proving asymptotic properties. The next propositions bound $\phi$ under the conditions of fully connected graphs or sampling designs that are relatively uninformative with respect to the covariance structure.
\begin{proposition}
Say $S_n = \sum_{i=1}^n Y_i$ and again consider $\mathcal{L}$. Then $-\mu^{-1}_n \leq \phi \leq \mu^{-1}_n (n-1)$ when $\mu_D > 0$.
\end{proposition}
\begin{proof}
Observe $\sum_{i \neq j} \sigma_{i, j}$.By Cauchy-Schwarz and the Geometric-Mean Inequality, $\sum_{i \neq j} \sigma_{i, j} \leq (n-1) \sum_{i=1} \sigma_i^2$. From here, we know that $2 \sum_{i < j} \sigma_{i, j} = n \mu_n \bar{\sigma}$. Then $n \mu_n \bar{\sigma} \leq (n-1) \sum_{i=1} \sigma_i^2$, which implies $\phi \leq \mu^{-1}_n(n-1)$.
Furthermore, since variance are non-negative, $1 + \mu_n \phi \geq 0 \implies \mu_n \phi \geq -1$.
\end{proof}
\begin{corollary}
For fully connected linear dependency graphs, $-(n-1)^{-1} \leq \phi \leq 1$.
\end{corollary}
\begin{proof}
Immediate from Proposition 2 since $\mu_n = n-1$.
\end{proof}
Provided that $\mu_n(n) \to \infty$ as $n \to \infty$, another informal corollary to Proposition 2 is that $\phi \geq 0$. Philosophically, this is useful to know since it implies that any system that is heavy with statistical dependencies in this fashion almost necessitates non-negative average correlation. From here, consider $\boldsymbol{\Delta}$, a vector of $N$ indicator variables $\delta_i$, such that $\delta_i = 1$ if and only if $Y_i$ is sampled from a larger population $P$. This next proposition also provides conditions s.t. $\phi \leq 1$ without strictly requiring a fully connected graph. It is useful when $P$ possesses a high number of dependent random variables and it is possible to execute a non-informative sampling mechanism s.t. $\mu_n$ (now treated as random in a temporary abuse of notation solely for Proposition 3) has a probability distribution that places most probability on lower values.
\begin{proposition}
Say $S_n=\sum_{i=1}^n Y_i$. Suppose $n^{-1} \sum^n_{i=1} \text{Var}(Y_i | \boldsymbol{\Delta}) = n^{-1}\sum^n_{i=1} \sigma_i^2$ and $|L|^{-1} \sum^{|L|}_{i < j} \text{Cov}(Y_i, Y_j | \boldsymbol{\Delta}) = |L|^{-1} \sum^{|L|}_{i < j} \sigma_{i, j}$, i.e., the sampling mechanism is uninformative on average with respect to the mean variance and individual covariances. Furthermore, denote $n_*$ as the number of correlated random variables in the population $P$ and say $n_* > n$ and thus $2^{-1} n(n-1)$ is the upper bound of the support of $|L|$. Then $\phi \leq 1$.
\end{proposition}
\begin{proof}
$\text{Var}(S_n | \boldsymbol{\Delta}) = \sum_{i=1}^n \text{Var}(Y_i | \boldsymbol{\Delta}) + 2 \sum_{i < j}  \text{Cov}(Y_i, Y_j | \boldsymbol{\Delta}) = \sum^n_{i=1} \sigma_i^2 + 2 \sum^{|L|}_{i < j} \sigma_{i, j}$, where only $|L|$ is random. The derivation is similar to Proposition 2. We skip these steps and observe that $n \mu_n |L|^{-1} \sum^{|L|}_{i < j} \sigma_{i, j} \leq (n-1) \sum^n_{i=1} \sigma_i^2$. Here, note that $\mu_n$ is also random and that $\mu_n = n-1$ when $|L| = 2^{-1} n(n-1)$. However, since $(n-1) \sum^n_{i=1} \sigma_i^2$ is a constant upper bound of the left-hand side and it is possible to sample $n < n_*$ correlated random variables, it must be the case that $n (n-1) |L|^{-1} \sum^{|L|}_{i < j} \sigma_{i, j} \leq (n-1) \sum^n_{i=1} \sigma_i^2$ and hence that $\phi \leq 1$.
\end{proof}
This next proposition is less demanding in its premises and hence applicable in more contexts. Furthermore, it will provide another route for establishing that $\phi \leq 1$ when $\sigma_i^2 = \sigma^2$ for $\forall i$, i.e., when the assumption of equal variances holds. Denote $\eta = (n^{-1} \sum_i^n \sigma_i^2)^{-1} \underset{i \in I}{\text{max}}(\sigma_i^2)$ as the ratio of maximum variance to mean variance to these ends. It won't be proven, but it is obvious that $\eta = O(1)$ when variances are finite. This fact also generalizes to the weighted case insofar as $\mathbf{w}_s$ is not a trivial vector of zeroes.
\begin{proposition}
Again observe $S_n = \sum_{i=1}^n Y_i$. Then $\text{Var}(S_n) \leq \{1 +\mu_n \eta \} \sum_{i=1}^n \sigma_i^2$. If $\sigma_i^2 = \sigma^2$ for $\forall i$, then $\phi \leq 1$ and $\text{Var}(S_n) \leq \{1 +\mu_n \} \sum_{i=1}^n \sigma_i^2$.
\end{proposition}
\begin{proof}
Recall that $\bar{\sigma} = |L|^{-1} \sum^{|L|}_{i < j} \sigma_{i, j}$. Hence, $\bar{\sigma} \leq \underset{i < j}{\text{max}}(\sigma_{i, j}) = \sigma_{r, s}$, say. By Cauchy-Schwarz, $\sigma_{r, s} \leq \sigma_r \sigma_s \leq  \underset{i \in I}{\text{max}}(\sigma_i^2)$. The general result then follows since $\phi \leq \eta$. Under the additional premise, $\eta = 1$, which implies the last stated bound. 
\end{proof}
Pertinently, $\text{Var}(S_n) \leq \{1 +\mu_n \} \sum_{i=1}^n \sigma_i^2$ is a tight inequality w.r.t. some dependency structures. When $\zeta$ is a sample of independent $Y_i$, $\mu_n = 0$ and $\text{Var}(S_n) = \sum_{i=1}^n \sigma^2_i \leq (1 + 0)  \sum_{i=1}^n \sigma^2_i = (1 + \mu_n)  \sum_{i=1}^n \sigma^2_i$. When $\zeta$ is a collection of the same random variable $n$ times, $S = n Y_1$ WLOG. Then $\text{Var}(S_n) = n^2 \sigma_1^2 \leq (1 + n - 1)n \sigma^2_1 = (1 + \mu_n) \sum_{i=1}^n \sigma^2_i$ since $\mu_n = n-1$.

Here, it is also useful to comment that we can also define a dependency graph $D = (I, L_D)$ s.t. a link exists in this graph if and only if $\text{Pr}(Y_i, Y_j) \neq \text{Pr}(Y_i)\text{Pr}(Y_j)$. All the results of this section also hold when the corresponding values are defined w.r.t. this graph. We will call these values $\mu_D$ and $\phi_D$ respectively. When this is done, it is easy to show that $\mu_n \phi_n = \mu_D \phi_D$. This can be established from the fact that $\phi_n = |L|^{-1}|L_D| \cdot \phi_D$.
\section{Asymptotic Properties of Additive Statistics\label{sec:3}}
We now show the utility of the previous set of results for the asymptotic analysis of additive statistics. Recall that one goal of this paper is to establish some basic but important statistical properties of additive estimators under 'apocalyptic' conditions, i.e., conditions such that independence is violated in such a way as to render the true dependency structure inconceivably checkered with non-nullified co-variation. To this end, we first establish some mild assumptions for the consistency of additive estimators and common variance estimators in this setting. Variance estimators are usually reserved for plug-in strategies with Wald statistics. Although we have no intention of justifying Wald-like hypothesis testing under our assumed conditions, establishing the consistency of variance estimators can still be useful for other purposes. Later, we see that they are still useful for investigating features of the unknown correlation structure, for example.
\begin{enumerate}
\item[A1.] There exists a $C_* \in \mathbb{R}^+$ s.t. $|Y_i| \leq C_*$ for $\forall i \in I$
\item[A2.] Denote a partition of $I$ s.t. $J_k \subseteq I$ for $k \in \{1, 2, \ldots K \} = \mathcal{K}$, $\bigcup_{k=1}^K J_k = I$ and therefore $J_r \cap J_t = \emptyset$ if $r \neq t$. In accordance with this, say $\zeta = \{ Y_i \}_{i \in I}$, $\zeta_k = \{ Y_{j} \}_{j \in J_k}$, and hence $\zeta = \bigcup_{k=1}^K \zeta_k$. Denote $n_k = |\zeta_k|$, which means $\sum_{k=1}^K n_k = n$. Then $n^{-1}_r n_t =O(1)$ for $\forall r,t \in \mathcal{K}$
\item[A3.] Let $\mathcal{L}$ be an arbitrary linear dependency graph. Then $\mu_n = o(n)$
\end{enumerate}
A1 establishes that we are working within a universe of bounded random variables. If this condition is not strictly necessary for mean and variance estimation, we can of course also do with the common assumptions that $\text{E}Y^2_i < \infty$ as A1' or the assumption that $\text{E}Y^4_i < \infty$ for all $i$ and that all referenced random variables are uniformly integrable. A2 simply stipulates that all cluster sizes have the same asymptotic order. This assumption is met in common research circumstances. We will not always need it. Note that A3 is very mild. For instance, it allows $\mu_n (n) \to \infty$ as $n \to \infty$. It only keeps $\mu_n$ from \textit{linearly} scaling with sample size. Put otherwise, A3 is mild because it allows for the typical random variable of a given sample to be correlated, on average, with a diverging number of others also collected---and in any imaginative way---as $n$ becomes arbitrarily large. There is no constraint on how this occurs.
\begin{proposition}
Suppose A1' and A4 for $\mathbf{w} \mathbf{Y}$ s.t. $\mathbf{w}_{i, j} = (w_{i, j}) = O(n^{-1})$. Then $\mathbf{w} \mathbf{Y} \overset{p}{\to} \text{E}(\mathbf{w} \mathbf{Y})$ as $n \to \infty$, where $\overset{p}{\to}$ denotes convergence in probability.
\end{proposition}
\begin{proof}
Let $s$ be arbitrary and observe $\mathbf{w}_s \mathbf{Y} = \sum_{i=1}^n w_{s, i} Y_i$. From Proposition 1, we know that $\text{Var}(\mathbf{w}_s \mathbf{Y}) = \sum_{i=1}^n w_{s, i}^2 \sigma_i^2 + n \mu_n \bar{\sigma}_{s, s}$, where $ \bar{\sigma}_{s, s} = |L|^{-1} \sum_{i < j} w_{s, i} w_{s, j} \sigma_{i, j}$. Since $\bar{\sigma}_{s, s} \leq \{ \underset{i \in I}{\text{max}}(|w_{s, i}|) \}^2 \underset{i \in I}{\text{max}}(\sigma_i^2)$ and $\underset{i \in I}{\text{max}}(\sigma_i^2) \leq C \in \mathbb{R}^+$ for $\forall n$, it follows that $n \mu_n \bar{\sigma}_{s, s} \leq n \mu_n  \{ \underset{i \in I}{\text{max}}(|w_{s, i}|) \}^2 \underset{i \in I}{\text{max}}(\sigma_i^2) \leq  n \mu_n  \{ \underset{i \in I}{\text{max}}(|w_{s, i}|) \}^2 C$. This implies that:
$$ \lim_{n \to \infty} n \mu_n \bar{\sigma}_{s, s} \leq \lim_{n \to \infty}   n \mu_n  \{ \underset{i \in I}{\text{max}}(|w_{s, i}|) \}^2 C = 0$$
The last step follows since $C$ is finite, while $\mu_n = o(n)$ and $\{ \underset{i \in I}{\text{max}}(|w_{s, i}|) \}^2 = O(n^{-2})$. Since $ \sum_{i=1}^n w_{s, i}^2 \sigma_i^2 \to 0$ as $n \to \infty$ as well, this implies that $\text{Var}(\mathbf{w}_s \mathbf{Y}) \to 0$. A use of Chebyshev's inequality then implies that $\mathbf{w}_s \mathbf{Y} \overset{p}{\to}  \text{E}(\mathbf{w}_s \mathbf{Y})$. Because $s$ was arbitrary, the conclusion is reached. Note that a quicker alternative proof could have used Proposition 4 since $\eta = O(1)$.
\end{proof}
Therefore, an adequate sub-linear behavior w.r.t. $\mu_n$ is sufficient for quadratic mean convergence. For instance, if $\mu_n = n^{1/c}, c  >1$, this is still sufficient to establish a WLLN, although the rate of convergence will be sub-optimal, especially for small $c$. It is also therefore apparent that weak dependence---at least as traditionally conceived---is not a necessary  condition for the consistency of a large class of commonly employed statistics, including average loss functions for learning algorithms.

Next, define $\hat{e}_i = Y_i - \hat{\text{E}}Y_i$. Analyzing variance estimators that make use of squared residuals is easier under the simplifying assumption that the linear dependency graph of $\{ \hat{e}_i^2 \}_{i \in I}$ is isomorphic to the linear dependency graph for $\{ \hat{e}_i \}_{i \in I}$. This is true for special cases, such as when the $\hat{e}_i$ are normally distributed. Even if this does not hold exactly, it is arguably mild to assert that it is at least true that the mean degrees of these graphs have the same asymptotic order and thus can be interchangeably described by one $\mu_n$. Moreover, we also use $\mu_n$ for the limiting average degree, but context will make this clear. For instance, for finite sample sizes, the linear dependency graph of $\{ \hat{e}_i^2 \}_{i \in I}$ is conceivably fully connected. However, provided that $\hat{\text{E}}Y_i \overset{p}{\to} \text{E}Y_i$ for $\forall i$ as $n \to \infty$ and thus $\hat{e}_i \overset{p}{\to} \epsilon_i$ for $\forall i$ as $n$ becomes arbitrarily large, the limiting linear dependency graph of $\{ \hat{e}_i^2 \}_{i \in I}$ is equal to the graph for $\{ \epsilon_i^2 \}_{i \in I}$. Hence, when discussing asymptotic orders, it is apparent that $\mu_n$ refers to the mean degree of the limiting object. Under the aforementioned simplifying assumption, it is also apropos to note that the limiting mean degree of the graph for $\{ Y_i \}_{i \in I}$ can also be represented by the same $\mu_n$ as the two previously mentioned. From here, we will always implicitly assume that an arbitrary weight $w$ is $O(n^{-1})$.
\begin{proposition}
Assume A1 and A3. Let $\mathbf{B} = \mathbf{w} \mathbf{Y} = \boldsymbol{\beta}+ \mathbf{w} \boldsymbol{\epsilon}$ s.t. $\text{E}\boldsymbol{\epsilon} = \boldsymbol{0}$ and again let $\mathbf{V} = \text{diag}\{ \text{Var}(\mathbf{Y}) \}$, where $\text{diag}(\mathbf{A})$ constructs a diagonal matrix from the diagonal of $\mathbf{A}$. Define $\hat{e} = Y_i - \hat{\text{E}Y_i}$. Then for an arbitrary pair $(s, t)$, it is true that $(\mathbf{w} \hat{\mathbf{V}} \mathbf{w}^{\top})_{s, t} = \sum_{i=1}^n w_{s, i} w_{t, i} \hat{e}_i^2 \overset{p}{\to} (\mathbf{w} \mathbf{V} \mathbf{w}^{\top})_{s, t}$ as $n \to \infty$.
\end{proposition}
\begin{proof}
Let $(s, t)$ be arbitrary under our premises. Then by Proposition 1, $\text{Var} \{ \sum_{i=1}^n w_{s, i} w_{t, i} \hat{e}_i^2 \} = \{1 + \mu_{n} \phi_{s, t: n} \} \sum_{i=1}^n w_{s, i}^2 w_{t, i}^2 \text{Var}(\hat{e}^2_i)$. Say $\phi_{s, t: n} \leq C_{\phi}$ and $\text{Var}(\hat{e}^2_i) \leq M$ for $\forall i$ since we know that they are $O(1)$. Then $\text{Var} \{ \sum_{i=1}^n w_{s, i} w_{t, i} \hat{e}_i^2 \} \leq \{1 + (n-1)C_{\phi}\} \cdot M \cdot \sum_{i=1}^n w_{s, i}^2 w_{t, i}^2$. Note that $\sum_{i=1}^n w_{s, i}^2 w_{t, i}^2 = O(n^{-3})$ while $(n-1)C_{\phi} \cdot M$ is obviously only $O(n)$ since $M$ is just some constant. Thus, it follows that $\text{Var} \{ \sum_{i=1}^n w_{s, i} w_{t, i} \hat{e}_i^2 \} \to 0$ as $n \to \infty$. Although this is sufficient, we observe that the variance converges to zero at the asymptotically equivalent rate of $\text{max}(n^{-3}\mu_n, n^{-3})$, which can be considerably faster than what was shown since we automatically set $\mu_n$ to its upper bound.

Under the same premises, $\hat{e}_i \overset{p}{\to} Y_i - \text{E}Y_i$ as $n \to \infty$. Hence, $\text{E}\hat{e}_i^2 \overset{p}{\to} \mathbf{V}_{i, i}$ as well as $n$ becomes arbitrarily large by the Portmanteau theorem since $\hat{e}_i$ is bounded with probability one for $\forall i$. This is sufficient for our conclusion. 
\end{proof}
An important note to make in relation to Proposition 6 is that $\text{Var} \{ \sum_{i=1}^n w_{s, i} w_{t, i} \hat{e}_i^2 \} $ is $O(n^{-2})$ at worst under our terrible, but manageable conditions. This is good to know since it must converge at a much faster rate than $\mathbf{B}$ if it is to be used as a plug-in estimator. To be exact, we require $n^2 \text{Var} \{ \sum_{i=1}^n w_{s, i} w_{t, i} \hat{e}_i^2 \} \to 0$ as a sufficient condition for Wald plug-in estimation. Here is an informal proof. Let $\hat{V} =  \sum_{i=1}^n w_{s, i} w_{t, i} \hat{e}_i^2$ and observe that $W_n = \{ \gamma_s(n)\hat{V} \}^{-1/2} (\mathbf{B}_s - \beta_s)$ is a Wald-like statistic for some function $\gamma_s(n) > 0$ that is intended to adjust for missed variability. Typical practices set $\gamma_s(n)=1$ \textit{a priori} in accordance with the assumption that the employed covariance model is well-specified. For simplicity, we set $\gamma_s(n)$ to $\mu_n$ since we care only about asymptotic orders in this exploration and temporarily assume that $\mu_{n}$ is non-zero and does not tend to zero. We know that $\text{Var}(\mathbf{B}_s) = O(n^{-1} \mu_n)$. Thus, $\{ n^{-1} \mu_n \}^{-1/2}$ stabilizes the variance of $\mathbf{B}_s$ in the sense that $n \mu^{-1}_n \text{Var}(\mathbf{B}_s)$ converges to a non-zero constant. Observe that $W_n = \{ n \hat{V} \}^{-1/2} \{ n^{-1} \mu_n \}^{-1/2}(\mathbf{B}_s - \beta_s)$. Since $n^2 \text{Var}(\hat{V}) = O(n^{-1} \mu_n)$, plug-in Wald statistics will always behave as intended under A3, at least insofar as asymptotic normality is not being considered.

\subsection{Some Consideration for Clustered Statistics}

Our next object is to explore the properties of cluster-robust variance estimation under thick dependency conditions. As setup, enact a partition of $\zeta$ into $K$ groups, as in A2, s.t. $\mathbf{B} = \mathbf{w} \mathbf{Y} = \sum_{k=1}^K \mathbf{w}_k \mathbf{Y}_k$. The basic cluster robust estimator has the following form: $\hat{\mathbf{C}} = \sum_{k=1}^K \mathbf{w}_k \mathbf{\hat{e}}_k \mathbf{\hat{e}}_k^{\top} \mathbf{w}_k^{\top}$, where each $\mathbf{\hat{e}}_k$ is a $n_k \times 1$ vector of residuals as previously defined. Classically, this approach has required that $K \to \infty$ as $n \to \infty$ and for $\underset{k \in \mathcal{K}}{\text{max}}(n_k) = n_M$ to be $O(1)$ in addition. If the outcome variables of different groups are independent---or at least uncorrelated---then $\hat{\mathbf{C}}$ is consistent for $\text{Var}(\mathbf{B})$ under these nice conditions \citep{mackinnon2023cluster}.

Although incredibly useful, the problem with this approach is simple, albeit undeniable. Outside of contrived examples, the specification of a partition that produces $K$ independent or even uncorrelated clusters is an arduous task. Recall that $\mathcal{L}_n$ is unknown and inestimable. Hence, provided a dynamic and inter-dependent world, the chances that a user specifies a valid partition with partial and imperfect knowledge about $\mathcal{L}_n$ are safely assumed to be small. This point has doubled poignancy when $n_M$ is small in addition. This is problematic because an invalid specification of the partition structure calls into question the consistency and utility of cluster-robust variance estimation.

 Here, we prove that---even if a partition is invalidly specified---cluster-robust variance estimation can still be consistent for its identified portion of the variance under a couple of common scenarios. Overall, three cases are considered: (1) the case s.t. $n_M \leq Q \in \mathbb{N}$, (2) the case s.t. $\underset{k \in \mathcal{K}}{\text{min}}(n_k) = n_m \to \infty$ and $\mathcal{L}_{k}$ is fully connected for $\forall k \in \mathcal{K}$, and (3) the case s.t. $n_m \to \infty, \mu_{n_k} =o(n_k)$ for an arbitrary $k \in \mathcal{K}$, and $K=O(1)$. Although establishing the consistency of a particular $\mathbf{\hat{C}}_s$ does not help us identify and estimate $\text{Var}(\mathbf{B}_s)$ in total, we show that the consistent estimation of portions of $\text{Var}(\mathbf{B}_s)$ can provide information on the magnitude of missed variability, which is why it is still important. Ultimately, this information can be used to inform a researcher's choice of a corrective factor in the style of Proposition 1. Gaining insight into the magnitude of $\mu_n \phi_n$ is also important for additional reasons that become clear in Section 6. The next proposition provides a version of Proposition 1 for the sum of random vectors. It is helpful because it gives an expression for the bias induced by an invalid partition and helps to see the problem with more clarity. It is not strictly necessary for our analysis, but we show it for the sake of completeness.

Say $\mathbf{T}_k = \mathbf{w}_k \mathbf{Y}_k = \boldsymbol{\mu}_k + \mathbf{w}_k \boldsymbol{\epsilon}_k$ and define a directed linear dependency graph $\mathcal{L}_{\boldsymbol{T}} = (\mathcal{K}, E_{\mathcal{L}_{\boldsymbol{T}}})$ s.t. a link exists from node $r$ to $t$ if and only if $\text{E}\{(\boldsymbol{T_r - \mu_r})(\boldsymbol{T_t - \mu_t})^{\top} \} \neq \boldsymbol{0}$. Define $\bar{\sigma}_{\mathbf{T}} = |E_{\mathcal{L}_{\boldsymbol{T}}}|^{-1} \sum_{r \neq t} \text{E}\{(\boldsymbol{T_r - \mu_r})(\boldsymbol{T_t - \mu_t})^{\top} \}$. Further, define $\boldsymbol{\phi}_{\boldsymbol{T}} = \{ \sum^K_{k=1} \text{Var}(\boldsymbol{T}_k) \}^{-1} K \bar{\sigma}_{\mathbf{T}}$.
\begin{proposition}
Observe $\mathbf{T}_k = \mathbf{w}_k \mathbf{Y}_k = \boldsymbol{\mu}_k + \mathbf{w}_k \boldsymbol{\epsilon}_k$ for $\mathbf{B} = \sum_{k=1}^K \mathbf{T}_k$ and define a directed linear dependency graph $\mathcal{L}_{\boldsymbol{T}} = (\mathcal{K}, E_{\mathcal{L}_{\boldsymbol{T}}})$ as previously constructed. Then $\text{Var}(\mathbf{B}) = \sum^K_{k=1} \text{Var}(\boldsymbol{T}_k) \{ \boldsymbol{I} + \mu_{\mathbf{T}} \boldsymbol{\phi}_{\boldsymbol{T}} \} = \sum^K_{k=1} \text{Var}(\boldsymbol{T}_k) + K \mu_{\mathbf{T}} \bar{\sigma}_{\mathbf{T}}$.
\end{proposition}
\begin{proof}
Observe that  $\text{Var}(\mathbf{B}) = \text{E}\{ \sum_{k=1}^K (\mathbf{T}_k - \boldsymbol{\mu}_k) \} \{ \sum_{k=1}^K (\mathbf{T}_k - \boldsymbol{\mu}_k) \}^{\top} = \sum_{k=1}^K \text{Var}(\mathbf{T}_k) + | E_{\mathcal{L}_{\boldsymbol{T}}}| \bar{\sigma}_{\mathbf{T}}$. By the degree-sum formula, $ |E_{\mathcal{L}_{\boldsymbol{T}}}| = K \cdot \mu_{\mathbf{T}}$. Substitution and algebraic manipulation of $ \sum^K_{k=1} \text{Var}(\boldsymbol{T}_k)$ in a manner analogous to Proposition 1 yields the identity.
\end{proof}
Therefore, it is then implied---provided the consistency of $\mathbf{\hat{C}}$ and a multivariate central limit theorem holds---that $\mathbf{\hat{C}}^{-1/2} (\mathbf{B} -\boldsymbol{\beta}) \overset{d}{\to} \mathbf{N}(\boldsymbol{0},  \boldsymbol{I} + \mu_{\mathbf{T}} \boldsymbol{\phi}_{\boldsymbol{T}})$, where $\overset{d}{\to}$ connotes convergence in distribution. Although this identity provides a medium for representing bias, $\boldsymbol{\phi}_{\boldsymbol{T}}$ is not an intuitive object in comparison to the other forms explored. Since most inferential settings truly require valid statements for the variances only, we retreat to this forum. For instance, $\mathbf{B}_s = \sum_{k=1}^K \mathbf{w}_{s, k} \mathbf{Y}_k = \sum_{k=1}^K \sum_{j=1}^{n_k} w_{s, k, j} Y_{k, j} = \sum_{k=1}^K T_{s, k}$. We can thus resort to an iterated application of Proposition 1 and state $\text{Var}(\mathbf{B}_s) = \{1 + \mu_K \phi_K \} \sum_{k=1}^K \{1 + \mu_{n_k} \phi_{s, n_k} \} \sum_{j=1}^{n_k} w_{s, k, j}^2 \sigma_{k, j}^2$, where $(\mu_K, \phi_K)$ correspond to a linear dependency graph connected to $\{T_{s, k} \}_{k \in \mathcal{K}}$ and the $(\mu_{n_k}, \phi_{s,n_k})$ correspond to the cluster-specific dependency graphs for the outcome variables in each cluster $\zeta_k$. However, more simply, we can settle for $\text{Var}(\mathbf{B}_s) = \{1 + \mu_K \phi_K \} \sum_{k=1}^K \text{Var}(T_{s, k})$ to represent the bias of cluster-robust methods when consistent estimation of each $\text{Var}(T_{s, k})$ is possible.

Some additional preliminaries help to frame this exploration. Pertinently, we provisionally assume that A2 applies to an arbitrary $\mu_{n_r}$ and $\mu_{n_t}$ for $r,t \in \mathcal{K}$. This is not a critical assumption. However, it greatly simplifies notation. If one does not wish to apply it, the forthcoming statements hold by replacing $\mu_{n_k}$ with $\underset{k \in \mathcal{K}}{\text{max}}(\mu_{n_k})$. Now, recall that $\text{Var}(\mathbf{B}_s) = O\{ n^{-1} \text{max}(\mu_n, 1) \}$. From the above cluster expression and our working assumption, we also know that $\text{Var}(\mathbf{B}_s) = O \{n^{-2} \text{max}(\mu_K, 1) \cdot K n_k \cdot \text{max}(\mu_{n_k}, 1)  \}$, which simplifies to $ O \{n^{-1} \text{max}(\mu_K, 1) \cdot \text{max}(\mu_{n_k}, 1)  \}$ since $K n_k = O(n)$ for all $k$ under A2. This implies that $\text{max}(\mu_n, 1)$ and $\text{max}(\mu_K, 1) \cdot \text{max}(\mu_{n_k}, 1)$ are equivalent in order.

To progress, we note that $\hat{\text{Var}}(T_{s, k}) = \mathbf{\hat{C}}_{s, k} = \sum_{j=1}^{n_k} w_{s, k, j}^2 \hat{e}^2_{k, j} + \sum_{r \neq t} w_{s, k, r} w_{s, k, t} \hat{e}_{k, r} \hat{e}_{k, t}$ for an arbitrary cluster $\zeta_k$. The usual machinery requires that $\text{Var}(\mathbf{\hat{C}}_{s, k}) \to 0$ as $n \to \infty$. Since $\text{Var}(\mathbf{\hat{C}}_{s, k}) \leq 2 \{ \text{Var} ( \sum_{j=1}^{n_k} w_{s, k, j}^2 \hat{e}^2_{k, j}) + \text{Var}( \sum_{r \neq t} w_{s, k, r} w_{s, k, t} \hat{e}_{k, r} \hat{e}_{k, t}) \}$, it is sufficient to reason about the individual variances on the right-hand side. From a previous proposition, we know the form of $\text{Var} ( \sum_{j=1}^{n_k} w_{s, k, j}^2 \hat{e}^2_{k, j})$. Otherwise, $\text{Var}( \sum_{r \neq t} w_{s, k, r} w_{s, k, t} \hat{e}_{k, r} \hat{e}_{k, t}) = \{1 + \mu_{n_k(n_k -1)} \phi_{s,n_{k}(n_k -1)} \} \sum_{r \neq s}  w^2_{s, r, i} w^2_{s, t, i}  \text{Var}(\hat{e}_{k, r} \hat{e}_{k, t})$. For finite samples, we can conservatively expect $\mu_{n_k (n_k -1)} = n_k (n_k - 1) - 1$ in most circumstances. Moreover, for simplicity, we can assume that $\text{Var}(\hat{e}_{k, r} \hat{e}_{k, t}) \leq M$ and $\phi_{s, n_k(n_k-1)} \leq C_{\phi}$ WLOG in addition.

Now, let $\mathbf{\hat{C}}_s = \sum_{k=1}^K \mathbf{\hat{C}}_{s, k}$. We also require $\text{Var}(\mathbf{\hat{C}}_s) \to 0$, at least as a typical sufficient condition to establish that $\mathbf{\hat{C}}_s$ is a consistent estimator of the variance model that results from a theorized partition. However, now, we have: $\text{Var}(\mathbf{\hat{C}}_s) = \{1 + \mu_{K, *} \phi_{K, *} \} \sum_{k=1}^K \text{Var}(\mathbf{\hat{C}}_{s, k})$. We must show that this expression tends to zero as $n \to \infty$ to establish consistency. A small handful of cases will be considered. Essentially, we will need to reason about the asymptotic orders of $\mu_{K_*}, \mu_{n_k}, K$, and $n_k$. Note that we do not need to reason about $\mu_{n_k (n_k -1)}$ since---outside of exceptional situations---it is safe to assume that it has an order of $(n_k -1) \text{max}( \mu_{n_k}, 1)$. This is because---in the limit---each $\hat{e}_r \hat{e}_t$ can still be expected to be correlated, even if marginally, with other products that share at least one index in the minimum. If we hold one pair of subscripts for some $\hat{e}_r \hat{e}_t$ constant, there are $2 \cdot (n_k-1)$ others that meet this criteria. Since $\mu_{n_k (n_k -1)} \leq n_k (n_k -1) - 1$, the remaining factor should have an order that is equivalent to $\text{max}(\mu_{n_k}, 1)$. Although this has not been rigorously substantiated, we can provisionally assume this to be the case for this exploration without much loss. This is also because we truly only need it for Case 3.

\paragraph*{Case 1: $n_M \leq Q \in \mathbb{N}$} Under this assumption, $\text{Var}(\mathbf{\hat{C}}_{s, k}) = O(n^{-4})$. This follows from the fact that the estimator is a sum of a finite number of $O(n^{-4})$ terms. Hence, $\text{Var}(\mathbf{\hat{C}}_s) = O\{ \text{max}(n^{-3} \mu_{K, *}, n^{-3}) \}$ since $K=O(n)$. Under this scenario, the asymptotic order of $\text{Var}(\mathbf{\hat{C}}_s)$ is less than or equal to $O(n^{-2})$ due to the fact that $\mu_{K, *} \leq K-1$. Therefore, $\mathbf{\hat{C}}_s$ is consistent at a rate of $n^{-1}$ for its identified portion of $\text{Var}(\mathbf{B}_s)$, even when \textit{all} clusters are correlated. This also makes estimators of this case a good fit for plug-in Wald estimation. This is because it is then implied that $n^2 \text{Var}(\mathbf{\hat{C}}_s) = O\{ \text{max}(n^{-1} \mu_{K, *}, n^{-1}) \} \to 0$ as $n$ becomes arbitrarily large. This follows from the fact that $\mu_{K_*} \leq \mu_n$ and $\mu_n = o(n)$ by assumption. 

\paragraph*{Case 2: $n_m \to \infty, \mathcal{L}_{k}$ fully connected for $\forall k \in \mathcal{K}$} For this case, $\text{Var}(\mathbf{\hat{C}}_{s, k}) = O(n^{-4} n_k^2 \mu_{n_k(n_k -1)}) = O\{ (n^{-1} n_k)^4 \}$ since the sum of weighted $\hat{e}_r\hat{e}_t$ s.t. $r \neq t$ dominates the variance. Thus, we already know that a sufficient condition for consistency is that $n_k = o(n)$. This is violated when $K = O(1)$. Assuming that $n_k = o(n)$ implies that $\text{Var}(\mathbf{\hat{C}}_s) =  O\{ (n^{-1} n_M)^4 \cdot K \cdot \text{max}(\mu_{K, *}, 1)\} $. Thus, a sufficient condition for consistency is that $n_M^3 \text{max}(\mu_{K, *}, 1) = o(n^3)$. This condition is fulfilled since $n$ has an asymptotic order that is equivalent to $K n_M$. Utilization of a Wald-like statistic, however, requires that $n_M^3 \text{max}(\mu_{K, *}, 1) = o(n)$, which is obviously non-trivial.

These statements are easier to grasp when cluster sizes are equal as a special case. Then $n = K n_k$ for all $k$ and the general expression for the order of $\text{Var}(\mathbf{\hat{C}}_s)$ is $O\{K^{-3} \text{max}(\mu_{K, *}, 1) \}$. Since this implies $\text{Var}(\mathbf{\hat{C}}_s) \leq C_*K^{-2}$ for some $C_* \in \mathbb{R}^+$ even when all clusters are correlated, it is a consistent estimator of its portion of the variance. Again, $n^2 K^{-3} \text{max}(\mu_{K, *}, 1) $ converging to zero is sufficient for Wald usage. Using substitution, we can translate this expression to $K^{-1} n^2_k \text{max}(\mu_{K, *}, 1)$. This means that one sufficient condition for Wald usage provided this specific setup is that $\text{max}(\mu_{K, *}, 1)$ is $O(1)$ and for $K^{-1} n^2_k \to 0$, which is more restrictive and less forgiving of invalid specifications.

For an example of when this becomes problematic, consider a social network analysis context s.t. the total sample size of possible relationships in a directed network is $N=n(n-1)$, where $n$ is the number of actors sampled. Researchers often choose to cluster on actor identity. Then $K=n$ and $n_k = n-1$. Obviously, then, $K^{-1} n_k = O(1)$. As a result, the denominator of the Wald statistic would fail to converge to a constant. It would remain random in behavior.

\paragraph*{Case 3: $n_m \to \infty, \mu_{n_k} = o(n_k)$, $K = O(1)$} This case necessitates that $n_k =O(n)$. Under the condition that $\mu_{n_k (n_k -1)}$ has the same asymptotic order as $ (n_k -1) \text{max}(\mu_{n_k}, 1)$, it is then implied that $\text{Var}(\mathbf{\hat{C}}_{s, k}) = O\{ n^{-1} \text{max}(\mu_{n_k}, 1) \}$ and $\text{Var}(\mathbf{\hat{C}}_s) =  O\{ n^{-1} \text{max}(\mu_{n_k}, 1) \}$. Therefore, although $\mathbf{\hat{C}}_s$ is consistent, it will converge at a rate that also prevents its straightforward application for Wald statistics. This is because $n^2 \text{Var}(\mathbf{\hat{C}}_s)$ does not converge to zero. To see this, first consider the sub-case s.t. $ \text{max}(\mu_{n_k}, 1) = \mu_{n_k}$. When this is true, it is also true that $\mu_{n_k} = O(\mu_n)$. As a consequence, $n^2 \text{Var}(\mathbf{\hat{C}}_s) = O(n \mu_n) \to \infty$ and $\mathbf{\hat{C}}_s$ does not converge to a constant when used to construct a test statistic; it remains random. This situation obviously does not change for the remaining sub-case.

\subsubsection{A Problem of Choice}

Let us recall our central problem to re-frame the last investigation. We have an additive estimator $\mathbf{w}_s \mathbf{Y}$. However, although we believe that it is consistent, we have no access to $\mathcal{L}$ or $D$ and thus have limited, imperfect---or even incorrect---information pertaining thow to model $\text{Var}(\mathbf{w}_s \mathbf{Y})$. There are two important situations to now consider: the situation s.t. it is believed a central limit theorem holds and the one where it is believed that one does not. We temporarily consider the former.

If the researcher believes that a central limit theorem holds, then we have two non-mutually exclusive choices for Wald-like confidence sets: choose a defensible correction via Proposition 1 or use cluster-robust variance estimation under the auspices of a user-specified partition. Hitherto, the primary strategy is to use $\mathbf{\hat{C}}_s$. We know that Case 1 estimators are the most robust against invalid partition specifications from the above analysis. However, since they require small cluster sizes, even if asymptotic normality holds, a lot of variability is likely to be missed and a correction will still need to be applied.

Estimators in the spirit of Case 2 are also feasible choices, but their 'safe' Wald-like use is predicated upon restrictive conditions. Put shortly, under A2, we require $n_k$ to be $o(n^{1/3})$ when the number of correlated clusters is asymptotically bounded. If $\mu_{K_*}$ diverges with sample size, then each $n_k$ must grow at an even slower, but ultimately unknown rate. Still, even if these estimators can be validly used for plugin-Wald statistics, there is no guarantee that further correction is unnecessary. Asymptotically expanding cluster sizes still does not guarantee that $\mu_K = 0$.

Recall, however: the only explored case s.t. $\mathbf{\hat{C}}_s$ was inconsistent for its portion of the variance was Case 2 when $K = O(1)$. Otherwise, it is a consistent estimator of the target quantity identified by the theoretical variance model. The rate of convergence might limit its direct use as a plug-in estimator for confidence sets, but it is important to note that it can still be used to investigate the variance structure. Recall that Proposition 2 (in a circumstance of copious dependencies) almost guarantees that $\phi$ is non-negative. Condition on this supposition, then, and note that $\mathbf{w}_s \mathbf{Y} = \sum_{k=1}^{K_1} T_{s, k} = \sum_{k=1}^{K_2} T_{s, k, *}$ for any two partitions into $K_1$ and $K_2$ groups. Proposition 1 then allows us to state that $\{1 + \mu_{K_1} \phi_{K_1} \} \sum_{k=1}^{K_1} \text{Var} (T_{s, k}) = \{1 + \mu_{K_2} \phi_{K_2} \} \sum_{k=1}^{K_2} \text{Var} (T_{s, k, *})$. It then follows WLOG that $\{  \sum_{k=1}^{K_2} \text{Var} (T_{s, k, *}) \}^{-1} \sum_{k=1}^{K_1} \text{Var} (T_{s, k}) < 1$ implies that $ \mu_{K_2} \phi_{K_2} <  \mu_{K_1} \phi_{K_1}$. The latter inequality, of course, implies that the $K_2$ partition possesses \textit{less} missed variability and is more robust. Unsurprisingly, then, one should always aim to choose the partition that results in the \textit{highest} estimated variance. For ample sample sizes, one can use differences such as $\mathbf{\hat{C}}_{s, 2} - \mathbf{\hat{C}}_{s, 1}$ for this purpose. Nevertheless, this strategy does not guarantee a correct choice of partition---and still---a correction might be required.

One strategy that synthesizes both approaches is to use various theories of dependence and hence partitions to compare $\mathbf{\hat{C}}_s$ estimators from Cases 1, 2, and 3. If $n$ is large enough, this should---in the minimum---give the researcher at least imperfect snapshots of the magnitudes of portions of the missed variance. She can then condition on this new (incomplete) knowledge and use it to inform a choice of correction via Proposition 1 after making use of the $\mathbf{\hat{C}}_s$ estimator that yields the highest estimate. If it is believed that a $\mathbf{\hat{C}}_s$ estimator from Case 2 possesses a poor rate of convergence, a Case 1 or Proposition 6 type estimator is more appropriate since they possess a better rate of convergence and can be 'treated' as constants more safely for ample sample sizes. The information acquired from the comparisons of cluster-robust estimators can still be factored into the choice of a deterministic correction. For the alternative situation---the one s.t. dependencies are so thick that a central limit theorem fails to hold---the variance estimators can still be used for exploratory ends insofar as at least A3 is true.

\section{Finite Sample Inference with Dense Dependence \label{sec:4}}

Recall from Section 2: for any additive statistic, $\mu_D \phi_D = \mu_n \phi_n$ and $\phi_D = |L_D|^{-1}|L| \cdot \phi_n$. Hence, if $\mu_n = o(\mu_D)$, $\phi_D \to 0$. This offers a new perspective on a known possibility. Statistical dependencies can diverge in number, but $\mu_n$ can be $o(n)$ or even bounded. When the latter is stipulated, this simply means that $\mu_D(n) \phi_D(n)$ must be two such functions s.t. this is requirement is met. We do not need to know the details of these functions, otherwise. Their importance is that they inform us---even in the presence of dense statistical dependencies that prevent convergence to normality---that we can still reason about mean convergence for linear statistics and explore strategies for inference. This is what we accomplish here. We will be working under the supposition that $\mu_D \to \infty$ or that $\mu_D$ is even quite possibly $O(n)$. We will \textit{prefer} that $\mu_n  = o(n)$, but we will see that this is not a strict requirement. Provided this setup, we show here that comfortably bounding $\mu_n \phi_n$ for a given sample is feasibly sufficient for constructing cogent confidence sets through the use of Hoeffding's or Bernstein's inequalities. To do this, we first introduce a type of random variable. The reason we are discussing these variables is that they permit us to extend a handful of known concentration inequalities to circumstances s.t. $\mu_D =O(n)$.

\subsection{$\mathcal{U}$ Random Variables}

Defining this class of random variable requires familiarity with the average functional value from basic analysis. Say $R_i = \int_{\mathbb{R}} 1_{y_i \in \mathcal{S}} dy_i$, where $\mathcal{S}_i$ is the support of the cumulative distribution function (CDF) of $Y_i$, say $F(y_i)$. For counting measures, $R_i = |\mathcal{S}_i|$. When $Y = g(X_1, \ldots, X_k)$ for some function $g$, we can generalize this WLOG to $R_{\mathbf{x}} = \int_{\mathbb{R}^k} 1_{ (x_1, \ldots, x_k )\in \mathcal{S}^k} dx_1 \hdots dx_k$, where $\mathcal{S}^k \subseteq \mathcal{S}_1 \times \hdots \times \mathcal{S}_k$. Given this setup, and temporarily abandoning subscripts for readability, the average function value is then $\text{Av}(Y) = R^{-1} \int_{\mathcal{S}} y dy$. Although we can also define $\text{Av}_{\mathbf{x}}(Y) = R_{\mathbf{x}}^{-1} \int_{\mathcal{S}^k} g(x_1, \ldots, x_k) dx_1 \hdots dx_k$, the former is mostly the focus here for exposition. A random variable $Y$ is in the $\mathcal{U}$ class if and only if $\text{E}Y = \text{Av}(Y)$. Similarly, it is said to be in the $\mathcal{U}$ class w.r.t. $\mathbf{X}$ if  $\text{E}Y = \text{Av}_{\mathbf{x}}(Y)$. Finally, a random variable is said to be \textit{regular} if its associated support is a single interval of real numbers or a set of integers $\{m, m+1, \ldots, M-1, M \}$ s.t. no integer is missing between $m$ and $M$. In this case, it is easy to verify that $\text{Av}(Y) = 2^{-1} (M+m)$ and membership in the $\mathcal{U}$ class implies that $\text{E}Y = 2^{-1}(M+m)$.

A set of basic properties for this class are proven in the supplementary materials. Here, it is sufficient to note only a small handful of them. First, if $Y$ is a regular random variable s.t. $\text{E}Y=0$, it is a member of the $\mathcal{U}$ class if and only if $m= -M$, i.e., the CDF of $Y$ is defined on symmetric support. Moreover, a regular and continuous random variable $Y$ is in the $\mathcal{U}$ class if and only if $\int_{\mathcal{S}} F(y) dy = \int_{\mathcal{S}} S(y) dy$, where $S(y) =  1- F(y)$, and hence equivalently if $M-\text{E}Y = \text{E}Y - m$. Pertinently, although random variables with symmetric probability distributions are $\mathcal{U}$ random variables, they are only a special case. $\mathcal{U}$ random variables can be asymmetric. Lastly, it can be easily shown that for any constant $c$, if $Y \in \mathcal{U}$, i.e., it is in the class of $\mathcal{U}$ random variables, then $cY \in \mathcal{U}$ and $Y + c \in \mathcal{U}$.

These variables are quite common. For instance, we have already established that all regular and symmetric bounded random variables are in this group. If $Y$ has a density $f(y) \propto \sigma_Y$ and $\sigma_y \to 0$ as $n \to \infty$, then $Y$ behaves more and more like a $\mathcal{U}$ random variable as $n$ becomes arbitrarily large. We will also see that the additive error variables associated with density or mass approximation are also elements when the approximating function is correct on average. More important for our purposes, however, is the connection between $\mathcal{U}$ status and the common linear model: $\mathbf{Y} = \mathbf{x} \boldsymbol{\beta} + \boldsymbol{\epsilon}$. Two conditions are ubiquitously supposed for its usage: (1) $\text{E}(\boldsymbol{\epsilon} | \mathbf{x}) = 0$, and (2) $\epsilon_i \sim N(0, \sigma_i^2)$ for $\forall i$. The latter assumption, however, is simply fecund. In reality, no observable error distribution for a $Y_i$ on finite support could ever truly be normally distributed. This common assumption is best understood as a mathematically convenient statement that results in (hopefully) negligible error. For instance, we could replace the traditional assumption with the assertion that each $\epsilon_i$ possesses a normal distribution that has been symmetrically truncated at a value that is extreme enough to leave out negligibly small probabilities. Such an assertion would be empirically isomorphic to the traditional assumption---and importantly---it would result in each $\epsilon_i$ being a $\mathcal{U}$ random variable. Therefore, it is safe to work in a universe s.t. $\mathbf{w}_s \mathbf{Y} - \beta_s = \sum_{i=1}^n w_{s, i} \epsilon_i$ is a sum of $\mathcal{U}$ functions. Positing that $\{ \epsilon_i \}_{i \in I}$ is a set of regular and symmetric $\mathcal{U}$ random variables is also less restricting than assuming truncated normality since many more distributions meet these criteria.

In the fixed regression setting, $\mathcal{U}$ status can be feasibly verified using the same plot that is used to check for violations of linearity. Namely, we would look at the residual versus fitted plot to confirm that the scatter of points is mean-zero at any location of the graph, and that they are randomly dispersed between two points that are roughly equidistant from the horizontal axis. Under the null assumption of a well-specified model, this is sufficient. More generally, however, when the $Y_i$ are continuous and identically distributed, we could verify that the behavior of the empirical CDF matches the sum-symmetric behavior described in the previous paragraph. Otherwise, for the discrete case, the empirical probability mass function could be plotted against the observed values to see if it demonstrates a sum-symmetric reallocation of mass from the area around the mode of the distribution to the tails.

\subsection{Finite Sample Inference for $\mathcal{U}$ Random Variables}

This section accomplishes two primary objectives. One, it introduces and discusses an additional regularity condition that is required for further theoretical work. Two, it utilizes this condition to extend the Hoeffding and Bernstein inequalities to sums of regular and symmetric $\mathcal{U}$ random variables under dense dependence. For this section, $\mathbf{w}$ is any $1 \times n$ vector of constants. Moreover, $\text{Av}_* ( \cdot)$ indicates a functional average taken over a Cartesian product of sets, i.e., over $\mathcal{S}_*^k = \mathcal{S}_1 \times \hdots \times \mathcal{S}_k$. Since the random variables being considered are bounded, it is helpful to know that all referenced objects exist for finite $n$.
\begin{enumerate}
\item[A5.] Let $\boldsymbol{\epsilon}$ be a $n \times 1$ vector of bounded random variables s.t. $\text{E}\boldsymbol{\epsilon} = 0$. For all $s > 0$, it is then true that $\text{max} \left ( \text{E} \left ( \text{exp} \{ s \cdot \mathbf{w}\boldsymbol{\epsilon} \} \right), \text{E} \left ( \text{exp} \{ s \cdot (-\mathbf{w}\boldsymbol{\epsilon}) \} \right) \right ) \leq  \text{Av}_* \left ( \text{exp} \{s \cdot \mathbf{w}\boldsymbol{\epsilon}  \} \right )$
\end{enumerate}
This new condition places an implicit bound on the behavior of the distributions of the marginal random variables and their joint distribution simultaneously. This differs from traditional approaches, which mostly place restrictions on the dependency structure. Put succinctly, for A5 to be true, the allocation of mass or density has to be somewhat balanced for the marginal distributions, and the joint distribution has to be biased away from elements of the joint support that produce larger values of the target statistic in addition.

Mutual independence alone does not guarantee A5. To see this, assume that $\text{E}(\text{exp}\{ s w_i \epsilon_i \}) > \text{Av}(\text{exp}\{ s w_i \epsilon_i \})$ for $\forall i$. Then $\text{E} \left ( \text{exp} \{ s \cdot \mathbf{w}\boldsymbol{\epsilon} \} \right) = \prod_i^n \text{E}(\text{exp}\{ s w_i \epsilon_i \}) > \prod_i^n \text{Av}(\text{exp}\{ s w_i \epsilon_i \}) =  \text{Av}_* \left ( \text{exp} \{s \cdot \mathbf{w}\boldsymbol{\epsilon}  \} \right )$ and A5 is false. Pertinently, this does not make A5 more restrictive than mutual independence in general. It simply informs us that the conditions that would supply it are somewhat complex when there are no constraints placed on the marginal distributions. The informal description in the previous paragraph hints that unimodal and symmetric marginal distributions prevent this quandary from occurring under mutual independence. This indeed turns out to be true. Before showing this, we prove a pivotal lemma. Note that the $\implies$ symbol is readable as 'implies.'
\begin{lemma}
Let $Z \in \mathcal{U}$ be regular and continuous s.t. $\text{E}Z = 0$ and $\text{max}(\mathcal{S}) = M$. Let $s >0$ and $w \in \mathbb{R}$ be arbitrary constants. Then $\text{Av}(\text{exp} \{s w Z \}) \leq \text{exp} \{24^{-1} s^2 w^2 R^2 \}$. 
\end{lemma}
\begin{proof}
We note that $\int_{\mathcal{S}} z^k dz = 0$ for all odd integers $k$ when $Z$ is a continuous and regular $\mathcal{U}$ random variable s.t. $\text{E}Z=0$. This is because $0 = 2^{-1} (M + m)$, which implies that $m = -M$. Hence $\int_{\mathcal{S}} z^k dz = (k+1)^{-1} \{ M^{k+1} - (-M)^{k+1} \} = 0$ since $k+1$ is even. This implies that $\text{Av}(Z^k) = 0$ for all odd integers $k$. Additionally, when $k$ is even, observe that $\int_{\mathcal{S}} z^k dz = (k+1)^{-1} 2M^{k+1}$ and hence $\text{Av}(Z^k) = (k+1)^{-1} M^k$.

Let $2\mathbb{N}_0 = \{ n \in \mathbb{N} | n = 2k, k \in \mathbb{N} \cup \{ 0 \} \}$. From here, for any constants $s > 0$ and $w \in \mathbb{R}$:
$$\text{exp} \{s w Z \} = \sum^{\infty}_{i=0} \{ i! \}^{-1} s^i w^i Z^i \implies$$
$$\text{Av}(\text{exp} \{s w Z \}) = \sum^{\infty}_{i=0} \{ i! \}^{-1} s^i w^i \text{Av}(Z^i) = \sum_{i \in 2\mathbb{N}_0 }  \{ i! \}^{-1} \{i + 1 \}^{-1} s^i w^i M^i = \sum^{\infty}_{i = 0}  \{ 2i! \}^{-1} \{2i + 1 \}^{-1} w^{2i} s^{2i} M^{2i}$$
Now, we compare two sequences: $S_1 = (2k)! \cdot (2k+1)$ and $S_2 = k! \cdot 6^k$ for $k \in \mathbb{N}$. We will prove that $S_1 \geq S_2$ for $\forall k \in \mathbb{N}$ by induction. Since $S_1 (0) = S_2(0) = 1$ and $2 \cdot (2k + 1) \geq 6$ for $k \geq 1$, the base cases are established. Now, suppose $(2n)!\cdot (2n+1) \geq n! \cdot 6^n$ for $n \geq 1$.
\begin{align*}
(2n)!\cdot (2n+1) & \geq n! \cdot 6^n \\
2 \cdot (2n + 1) & \geq 6 \implies \\
2 \cdot (2n + 1)\cdot (2n)!\cdot (2n+1) & \geq n! \cdot 6^{n+1} \implies \\
2 \cdot (n + 1) \cdot (2n + 1)\cdot (2n)!\cdot (2n+1) & \geq (n + 1) \cdot n! \cdot 6^{n+1} \implies \\
\{2 (n+1) \}! \cdot (2n+1) & \geq (n+1)! \cdot 6^{n+1} \implies \\
\{2 (n+1) \}! \cdot \{2(n+1) +1 \} & \geq (n+1)! \cdot 6^{n+1}
\end{align*}
Therefore, $S^{-1}_1 = \{ (2k)! \cdot (2k+1) \}^{-1} \leq S^{-1}_2 = \{ k! \cdot 6^k \}^{-1}$ for all $k \in \mathbb{N}$. Hence:
\begin{align*}
\text{Av}(\text{exp} \{s w Z \}) = \sum^{\infty}_{i = 0}  \{ 2i! \}^{-1} \{2i + 1 \}^{-1} w^{2i} s^{2i} M^{2i} \leq  \sum^{\infty}_{i = 0}  \{ i! \}^{-1} \{ 6^i \}^{-1} w^{2i} s^{2i} M^{2i} = \sum^{\infty}_{i = 0}  \{ i! \}^{-1} \{6^{-1}  s^2w^2M^2 \}^i = \text{exp} \{6^{-1} s^2w^2M^2 \}
\end{align*}
Since $M= 2^{-1} R$ for $\mathcal{U}$ variables, $\text{Av}(\text{exp} \{s w Z \}) \leq \text{exp} \{24^{-1} s^2w^2R^2 \} $.
\end{proof}
\begin{proposition}
Suppose $Z \in \mathcal{U}$ is a regular and continuous random variable s.t. $\text{E}Z=0$. If $\text{E}Z^k \leq \{2^k \cdot (k+1) \}^{-1} R^k$ for all even integers $k$ and $\text{E}Z^k \leq 0$ for odd integers $k$, then $\text{E} (\text{exp} \{ s Z \}) \leq \text{Av} (\text{exp} \{ s Z \})$ for any constant $s > 0$. If $Z$ has a symmetric probability distribution, $s$ can be any constant.
\end{proposition}
\begin{proof}
From Lemma 1, we know that $\text{Av}(\text{exp} \{s Z \}) = \sum_{i \in 2\mathbb{N}_0}  \{ i! \}^{-1} \{ 2^i \cdot (i + 1) \}^{-1} s^i R^i$. Therefore:
\begin{align*}
\text{E}(\text{exp} \{s Z \}) =  \sum^{\infty}_{i=0} \{ i! \}^{-1} s^i \text{E}(Z^i) \leq  \sum_{i \in 2\mathbb{N}_0}  \{ i! \}^{-1} s^i \text{E}Z^i \leq \sum_{i \in 2\mathbb{N}_0}  \{ i! \}^{-1} \{ 2^i \cdot (i + 1) \}^{-1} s^i R^i = \text{Av}(\text{exp} \{s Z \})
\end{align*}
The expansion of the proposition to the case of symmetric probability distributions is quick. Under this case, all odd moments are also zero. Therefore, the first inequality of the latter sequence of logic becomes a strict equality. Since only even indexes are left, $s$ can be any constant.
\end{proof}
\begin{lemma}
If $Z$ is a symmetric, regular, and unimodal random variable s.t. $\text{E}Z =0$ and $\text{max}(\mathcal{S}) =M$, then $\text{E}Z^k \leq \text{Av}(Z^k)$ for all even integers $k$ when $Z$ is continuous. If $Z$ is discrete and symmetric with support $\mathcal{S} = \{-c_{M}, -c_{M-1}, \ldots, 0, \ldots, c_{M-1}, c_M \}$ w.r.t. constants $c_1 < c_2 < \cdots < c_{M-1} < c_M$, then $\text{E}Z^k \leq \{| \mathcal{S}| -1 \}^{-1} |\mathcal{S}| \cdot \text{Av}(Z^k)$.
\end{lemma}
\begin{proof}
We provide a proof for the continuous case first.

By definition, $f(z)$ is increasing on $[-M, 0]$ and decreasing on $[0, M]$ if $Z$ is unimodal and symmetric. However, $z^k$ is decreasing on $[-M, 0]$ and increasing on $[0, M]$ when $k$ is an even integer. WLOG, we reason w.r.t. $[-M, 0] \subset \mathcal{S}$. By Chebyshev's integral inequality, it is then true that $\int^0_{-M} z^k f(z) dz \leq M^{-1} \int^0_{-M} z^k dz \cdot \int^0_{-M} f(z) dz$ since $z^k$ and $f(z)$ have opposite monotonicities on $[-M, 0]$. Furthermore, since $\int^0_{-M} f(z) dz = 2^{-1}$ and $(-c)^k f(-c) = c^k f(c)$ for $\forall c \in \mathcal{S}$ when $k$ is an even integer and $f(z)$ is symmetric, $2 \cdot \int^0_{-M} z^k f(z) dz = \text{E}Z^k$. It is also obviously the case that $ 2 \cdot \int^0_{-M} z^k dz = \int^M_{-M} z^k dz$ since $k$ is an even integer. Therefore, $\int^0_{-M} z^k f(z) dz \leq M^{-1} \int^0_{-M} z^k dz \cdot \int^0_{-M} f(z) dz$ implies that $\text{E}Z^k \leq  (2M)^{-1} \cdot \int^M_{-M} z^k dz = \text{Av}(Z^k)$.

Now, for the discrete case, we will say that $\mathcal{S} = \{-M, -(M-1), \ldots, M-1, M \}$ WLOG and that $\sum^{-1}_{z=-M} f(z) = \sum_{z=1}^M f(z)$ and therefore that $2 \cdot \sum^{-1}_{z=-M} f(z) = 1 - f(0)$. Once again, we also note that $f(z)$ is increasing on $\{-M, \ldots, -1 \}$ and $z^k$ is decreasing on this same set for even integers $k$. Thus, by the Chebyshev sum inequality, $\sum^{-1}_{z=-M} z^k f(z) \leq M^{-1} \sum^{-1}_{z=-M} z^k \cdot \sum^{-1}_{z=-M} f(z) \implies 2 \cdot \sum^{-1}_{z=-M} z^k f(z) \leq M^{-1} \cdot 2 \cdot \sum^{-1}_{z=-M} z^k \cdot \{ 2^{-1} - 2^{-1} f(0) \}$. Since $z \cdot f(z) = 0$ and $z^k=0$ when $z=0$, it is then implied that $\text{E}Z^k \leq (2M)^{-1} \sum_{\mathcal{S}} z^k \cdot  \{ 1 - f(0) \} \leq  (2M)^{-1} \sum_{\mathcal{S}} z^k = \{| \mathcal{S}| -1 \}^{-1} |\mathcal{S}| \cdot \text{Av}(Z^k)$. The last sequence of logic is true because $|\mathcal{S}| = 2M + 1$ and $f(0) > 0$.
\end{proof}
\begin{corollary}
Suppose $Z$ is a regular, symmetric $\mathcal{U}$ random variable s.t. $\text{E}Z = 0$ and $\text{max}(\mathcal{S}) =M$. If $Z$ is continuous, then $\text{E}(\text{exp} \{s Z \}) \leq \text{Av}(\text{exp} \{s Z \})$ for any $s \in \mathbb{R}$. If $Z$ is discrete as defined in Lemma 2, then $\text{E}(\text{exp} \{s Z \}) \leq \{ 2M \}^{-1} \{2M + 1 \} \cdot \{ \text{Av}(\text{exp} \{s Z \}) - \{2M +1 \}^{-1} \}$ for any $s \in \mathbb{R}$.
\end{corollary}
\begin{proof}
Let $s$ be arbitrary. For the continuous case, Lemma 2 and Proposition 8 imply the result. For the discrete case, we again note that $\text{Av}(Z^k) = 0$ when $k$ is an odd integer because $\sum^{-1}_{z=-M} z^k = - \sum^M_{z=1} z^k$ and $z^k =0$ when $z=0$.  From the premises and Lemma 2, we know that $\text{E} (\text{exp} \{s Z \}) \leq 1 + (2M)^{-1} (2M +1) \cdot \sum^{\infty}_{i=1} \{ i! \}^{-1} s^{i} \text{Av}(Z^{i})$. Then:
$$\text{E} (\text{exp} \{s Z \}) \leq 1 + (2M)^{-1} (2M +1) \cdot \sum^{\infty}_{i=1} \{ i! \}^{-1} s^{i} \text{Av}(Z^{i}) = 1 + (2M)^{-1} (2M +1) \cdot \{ \text{Av}(\text{exp} \{s Z \}) -1 \}$$
The statement of interest arrives via basic algebra and is therefore omitted.
\end{proof}
Proposition 8 and the subsequent statements establish that symmetric and regular $\mathcal{U}$ random variables---precisely the type of variable that often arises in linear regression settings, as aforementioned---are good candidates for A5. Again, this does not mean that other types of marginal distributions that are not strictly $\mathcal{U}$, regular, or symmetric are not. It is possible to imagine how different dynamics between marginal and joint distributions could still deliver A5. However, a class of variables that can serve as a basic case under mutual independence provides an anchor.

Constraints on the joint distribution are next. From here, the use of A5 can be justified via recourse to another set of identities concerning $\mathcal{U}$ variables. Unfortunately---in the continuous case---these identities require the existence of a high-dimensional density. Again say $Z$ is an arbitrary bounded random variable with mass or density function $f(z)$ and let $L = f(z_1, \ldots, z_n)$ denote a joint mass function or density w.r.t. $Y = g(Z_1, \ldots, Z_n)$ for some function $g$. The two identities are:
\begin{equation}
\text{Av}(Z) = \text{E}Z + R^{-1} \sigma_{Z, f^{-1}(Z)}
\end{equation}
\begin{equation}
\text{Av}\{ g(Z_1, \ldots, Z_n) \} = \text{E}Y + R_{\mathbf{z}}^{-1} \sigma_{g(Z_1, \ldots, Z_n), L^{-1}}
\end{equation}
For random variables on finite support, inverse mass functions---both joint and marginal---will always exist for finite $n$. Remember: the support of a marginal or joint distribution is the closure of the set of values s.t. $f(z) > 0$ and $L > 0$ respectively. Hence, $L^{-1}$ and $f^{-1}(z)$ are well-defined for mass functions. They are also well-defined for marginal and joint densities of bounded random variables if the basic objects exist: a precondition that requires the absolute continuity of the cumulative probability functions.

Immediately, we can see that a variable $Z$ is in the $\mathcal{U}$ class if and only if it is uncorrelated with its own inverse mass or density. Just the same w.r.t. Eq. (2), we can see that $\text{E}Y \leq \text{Av}\{ g(Z_1, \ldots, Z_n) \}$ if and only if $\sigma_{g(Z_1, \ldots, Z_n), L^{-1}} \geq 0$. This of course means that as $g(Z_1, \ldots, Z_n)$ takes larger values, $L^{-1}$ tends to take larger ones as well. Informally, this indicates that $L$ tends to take \textit{smaller} values as $g(Z_1, \ldots, Z_n)$ takes larger ones. Hence, at least in a linear sense, the joint mass function or density tends to place smaller amounts of probability or density on the $(z_1, \ldots, z_n) \in \mathcal{S}^n$ that produce large values of $Y$. In other words, if the joint density or mass function displays any behavior akin to the concentration of measure and the marginal distributions behave in the manner previously stated, one can expect $\text{E}Y \leq \text{Av}\{ g(Z_1, \ldots, Z_n) \}$ to hold as a condition. Recall that a probability measure is said to concentrate if it places most of its measurement on a subset of values in the support that neighbor the expected value of the random variable. If the values in an arbitrarily small neighborhood of the expected value continue to accumulate measure with sample size, the random variable converges in probability to a constant. 

Vitally, this suggests that A5 holds when---together with the aformenetioned informal restrictions on the marginal distributions---A3 is true (1) and it is \textit{also} the case that $L$ exists (2) and $\mathcal{S}^n$ remains rectangular for arbitrarily large (but finite) samples (3). Consider the first property again: $\mu_n =o(n)$. When the variances of the $\epsilon_i$ are finite, $\mu_n =o(n)$ implies that $\mathbf{w}_n \boldsymbol{\epsilon}$ converges in probability to zero. Hence, it will be the case that $\sum_{i=1}^n w_i \epsilon_i = g(\epsilon_1, \ldots, \epsilon_n) \approx 0$ for sufficiently sized $n$ with high probability, which implies that the $(e_1, \ldots, e_n) \in \mathcal{S}^n$ that distance $\mathbf{w}_n \boldsymbol{\epsilon}$ from zero in higher magnitudes will be afforded smaller density, at least overall in some sense, as $n$ gets larger. This, in turn, will afford more of a likelihood that $\text{E} (\text{exp} (s \mathbf{w} \boldsymbol{\epsilon})) $ is bounded by the target value. Again, the second property (and bounded supports) allow for Eq. (2) and the referenced objects to exist. The third property is important to explore in more detail. Essentially, it states that the mutual dependence between the $\epsilon_i$ does not remove all density from some $(e_1, \ldots, e_n) \in \mathcal{S}_*^n$, even as $n$ gets large. Importantly, this does \textit{not} mean that each $(e_1, \ldots, e_n) \in \mathcal{S}_*^n$ needs to be afforded a non-negligible density for all sample sizes. The density can become arbitrarily small. It simply needs to be non-zero for all $n$ considered. If this holds, then $\mathcal{S}^n = \mathcal{S}_*^n$ and thus $\text{Av}( \text{exp} \{ s \mathbf{w} \boldsymbol{\epsilon}\}) = \text{Av}_*( \text{exp} \{ s \mathbf{w} \boldsymbol{\epsilon}\})$. Importantly, we do not need the last two of these properties to hold in the true limit. We only need them to hold for sufficiently sized, but ultimately finite $n$.

For clarity, we explore yet another characterization. Define a function $\eta(z) = f(z) - h(z)$, where $f(z)$ and $h(z)$ are two densities defined on the same support WLOG. Then it is apparent that $\int_{\mathcal{S}} \eta(z) dz = 0$. Now rearrange and multiply by $z$: $z \cdot f(z) = z \cdot h(z) + z \cdot \eta(z)$. Integrating across, we arrive to $\text{E}Z = \text{E}_h Z + \int_{\mathcal{S}} z  \cdot \eta(z) dz$, where $\text{E}_h$ indicates an expectation taken w.r.t. $h(z)$. Another set of identities results when $h(z) = R^{-1}$, the uniform density. Then the above becomes $\text{E}Z = \text{Av}(Z) + \int_{\mathcal{S}} z  \cdot \eta(z) dz$ and it follows that $\int_{\mathcal{S}} z  \cdot \eta(z) dz = - R^{-1} \sigma_{Z, f^{-1}(Z)}$. Just the same, $\int_{\mathcal{S}^n} g(z_1, \ldots, z_n) \cdot \eta(z_1, \ldots, z_n) dz_1 \cdots dz_n = - R_{\mathbf{z}}^{-1} \sigma_{g(Z_1, \ldots, Z_n), L^{-1}}$. Therefore, $\mathcal{U}$ status is directly related to $\eta(\mathbf{z})$ and $g(\mathbf{z})$ being orthogonal. In the supplementary materials, it is also proven that $\eta(Z) \in \mathcal{U}$ is a sufficient condition for $Z \in \mathcal{U}$ when $h(z) = R^{-1}$. Although this condition seems abstract, it simply means that $\text{E} \{\eta(Z) \} = 0$, i.e., that the expected value of the 'approximating' density matches the expected value of the true one. If $\text{E}L \to R_{\mathbf{z}}^{-1}$ as $n$ grows, this is also sufficient for $\text{E} \{ g(Z_1, \ldots, Z_n) \} \to \text{Av} \{ g(Z_1, \ldots, Z_n) \}$ and hence A5 when $\mathcal{S}_*^n = \mathcal{S}^n$.

\subsection{Main Results}
Before using Lemma 1 to derive a sharper version of Hoeffding's inequality, we first extend classical results. A sub-$\mathcal{U}$ random variable is one such that $\text{Av}(Z) \leq \text{E}Z$.
\begin{lemma}[Extension of Hoeffding's lemma]
Let $Z$ be a random variable defined on bounded support $\mathcal{S}$ with minimum $m$ and maximum $M$. Then for any $s>0$, $Av \left ( exp \{ sZ \} \right ) \leq exp \{s Av (Z) + 8^{-1} s^2 R^2 \}$. If $\text{E}Z = 0$ and $Z$ is sub-$\mathcal{U}$, $Av \left ( exp \{ sZ \} \right ) \leq exp \{8^{-1} s^2 R^2 \}$. 
\end{lemma}
\begin{proof}
This proof follows the logic of the proof for Hoeffding's lemma since $Av(\cdot)$ is a linear operator with the monotonic property over inequalities.

Since $exp \{s z \}$ is a convex function of $z$, for all $z \in \mathcal{S}$:
$$exp \{s z \} \leq R^{-1} \left ( (M - z) exp \{ s m \} + (z - m) exp \{ s M \} \right )$$
Then, denoting $Av(Z) = Av$:
\begin{align*}
Av \left ( exp \{s Z \} \right ) \leq  R^{-1} \left ( (M - Av) exp \{ s m \} + (Av - m) exp \{ s M \} \right )
\end{align*}
Now, specify a function $g(x) =  R^{-1}m x + log \left ( (M - Av)  + (Av - m) exp \{x\} \right ) - log(R)$. Then it is easily demonstrable that $g(0) = 0, g'(0) = R^{-1} Av$ and $g''(x) \leq 4^{-1}$ for $\forall x$. Hence, utilizing a Taylor expansion around $0$, since $g(x) = g(0) + x g'(0) + 2^{-1}x^2 g''(x_*)$ for some $x_*$ that is between $0$ and $x$, it is true that $g(x) \leq R^{-1} x Av + 8^{-1} x^2$. This implies that $g(sR) \leq s Av + 8^{-1}s^2 R^2$. Hence, $Av \left ( exp \{s Z \} \right ) \leq exp \{ s Av + 8^{-1}s^2 R^2 \}$.

Now, since $s > 0$, when $\text{E}Z = 0$ and $Z$ is sub-$\mathcal{U}$, it is true that $Av(Z) \leq 0$. Therefore, $Av \left ( exp \{s Z \} \right ) \leq exp \{ 8^{-1}s^2 R^2 \}$.
\end{proof}
\begin{theorem}[\textbf{Extension of Hoeffding's Inequality}]
Let $\epsilon_1, \ldots, \epsilon_n$ be defined on bounded supports $\mathcal{S}_{1}, \ldots, \mathcal{S}_{n}$ such that for an arbitrary $i \in \{1, \ldots, n \}$, $m_i \leq \epsilon_i \leq M_i$ with probability one and $\text{E}\epsilon_i = 0$. Let $\mathbf{w}$ be a vector of constants s.t. $S_n = \mathbf{w} \boldsymbol{\epsilon} = \sum_{i=1}^n w_i \epsilon_i$. Suppose A5 for some fixed $n$ and let $\tau > 0$ be arbitrary. Then $\text{Av}(\epsilon_i) \leq 0$ for $\forall i$ and $\mathbf{w} > \boldsymbol{0}$ $\implies \text{Pr}(|S_n| > \tau) \leq  2 exp \{ - (\sum_{i=1}^n w_i^2 R^2_i)^{-1} 2 \tau^2 \}$.

Furthermore, if there exists some $N \in \mathbb{N}$ such that for $\forall n > N$ it is true that the above conditions hold, then the stated inequality is true for all $n > N$ provided all objects exist.
\end{theorem}
\begin{proof}
Again, note that for any $s > 0$, all integrals exist---including those of A5---since each $\epsilon_i$ is a bounded random variable. This can easily be shown with an extended application of Fubini's theorem. Let $\tau$ and $s$ be arbitrary real numbers such that $\tau > 0$ and $s > 0$. We will proceed with the statement $Pr ( S_n > \tau)$. Then by Markov's inequality and A5:
$$\text{Pr} \left ( \text{exp} \{ s S_n \} > \text{exp} \{ s \tau \} \right ) \leq \text{E} \left ( \text{exp} \{ s S_n \} \right ) \text{exp} \{- s \tau \} \leq \text{Av}_* \left ( \text{exp} \{s (\sum_{i=1}^n w_i \epsilon_i) \} \right )  \text{exp} \{- s \tau \}$$
However, $\text{Av}_* \left ( \text{exp} \{s (\sum_{i=1}^n w_i \epsilon_i) \} \right )  \text{exp} \{- s \tau \} = \prod_{i=1}^n \text{Av}\left ( \text{exp} \{s w_i \epsilon_i \} \right ) \text{exp} \{- s \tau \}$. By Lemma 3, then, since each $s w_i >0$ and our premise asserts that $\epsilon_i$ is sub-$\mathcal{U}$ for $\forall i$:
$$\prod_{i=1}^n \text{Av}\left ( \text{exp} \{s w_i \epsilon_i \} \right ) \text{exp} \{- s \tau \} \leq \prod_{i=1}^n \text{exp} \{ 8^{-1}s^2 w_i^2 R_i^2 \}  \text{exp} \{- s \tau \} =  \text{exp} \{ 8^{-1}s^2 \sum_{i=1}^n w_i^2 R_i^2 - s \tau \} $$
From here, we proceed by finding the minimum of $\text{exp} \{ 8^{-1}s^2 \sum_{i=1}^n w_i^2 R_i^2 - s \tau \}$ for $s \in \mathbb{R}^+$. It is easy to verify that this function is minimized at $s = (\sum_{i=1}^n w_i^2 R^2_i)^{-1} 4 \tau$. Thus, $\text{Pr}( S_n > \tau) \leq \text{exp} \{ - (\sum_{i=1}^n w_i^2 R^2_i)^{-1} 2 \tau^2 \}$.

This proves one direction of the main statement. For the other, one is required to proceed in an analogous fashion for $\text{Pr}(-S_n > \tau)$. One simply uses A5 and Lemma 3 again. For this reason, the steps are omitted. Hence $\text{Pr}(S_n > \tau) + \text{Pr}(S_n < -\tau)   \leq 2 exp \{ - (\sum_{i=1}^n w_i^2 R^2_i)^{-1} 2 \tau^2 \} \implies \text{Pr}(|S_n| > \epsilon)  \leq  2 exp \{ - (\sum_{i=1}^n w_i^2 R^2_i)^{-1} 2 \tau^2 \}$.

The asymptotic statement of the theorem can be achieved by setting some $N \in \mathbb{N}$ such that A5 holds for all natural numbers greater than $N$ and then letting $n \in \mathbb{N}$ be arbitrary such that $n > N$. Then the proofs proceed exactly as before.
\end{proof}
This theorem is very useful for the construction of valid confidence intervals under almost arbitrary conditions of probabilistic dependence, and without any need to specify or completely understand the latent system of dependencies, or even many features of the marginal or joint probability distributions. Again, $\mu_D$ could be $O(n)$. All that is required in terms of dependence is A5.
\begin{example}
Say $\{ Y_i \}_{i \in \zeta}$ is an identically distributed sample of sub-$\mathcal{U}$ random variable s.t. $\text{E}Y_i = \mu$. Also say $m, M$ are the minimum and maximum of the associated support, respectively. Furthermore, say $\mu_D = n-1$, but A5 is fulfilled. Then $[\bar{Y} - R \cdot \sqrt{(2n)^{-1}\text{log}(2/\alpha)}, \bar{Y} + R \cdot \sqrt{(2n)^{-1}\text{log}(2/\alpha)}]$ is at least a $1-\alpha$ confidence set for $\mu$ and $\alpha \in (0,1)$. If $m \geq 0$ and $m,M$ are unknown, then $\{ \mu \hspace*{1mm} | \hspace*{1mm} \left ( 1+\sqrt{2 n^{-1}log(\alpha^{-1}2)} \right )^{-1} \bar{Y} \leq \mu \leq \left ( 1-\sqrt{2n^{-1}log(\alpha^{-1}2)} \right )^{-1} \bar{Y}  \}$ is also at least a $1-\alpha$ confidence set when $n > 2 log(\alpha^{-1}2)$.
\end{example}
This next lemma allows us to extend Bernstein's inequality, which can sometimes provide much sharper confidence sets. It is also an adaptation of a classic result.
\begin{lemma}
Let $Z$ be a random variable such that $|Z| \leq M$ almost surely, $\text{E}Z=0$, and $Z$ is sub-$\mathcal{U}$. Then for any $0 < s$, $\text{Av}(\text{exp} \{sZ \}) \leq \text{exp} \{ M^{-2} \text{Av}(Z^2) (exp \{sM\} -1 - sM) \}$.
\end{lemma}
\begin{proof}
$$\text{exp} \{s Z \} = 1 + sZ + \sum^{\infty}_{k=2} (k!)^{-1} s^k Z^k \leq 1 + sZ + \sum^{\infty}_{k=2} (k!)^{-1} s^k |Z|^k = 1 + sZ + \sum^{\infty}_{k=2} (k!)^{-1} s^k Z^2 |Z|^{k-2}$$
However, the last expression is less than or equal to:
$$1 + sZ + \sum^{\infty}_{k=2} (k!)^{-1} s^k Z^2 M^{k-2} = 1 + sZ + M^{-2}  Z^2 \sum^{\infty}_{k=2} (k!)^{-1} s^k M^{k} = 1 + sZ + M^{-2}  Z^2 (exp \{sM \} - 1 -sM)$$
Hence, $\text{Av}(\text{exp} \{sZ \}) \leq 1 + M^{-2}\text{Av}(Z^2) (exp \{sM\} -1 - sM) \leq \text{exp} \{ M^{-2} \text{Av}(Z^2) (exp \{sM\} -1 - sM) \}$. The first of the previous inequalities follows from the fact that $\text{Av}(Z) \leq 0$. The second follows from the fact that $1+x \leq exp \{x\}$ for $\forall x \in \mathbb{R}$.
\end{proof}
\begin{theorem}
Let $\epsilon_1, \ldots, \epsilon_n$ be defined on bounded supports $\mathcal{S}_{1}, \ldots, \mathcal{S}_{n}$ and $\mathbf{w}$ be a vector of positive constants s.t. $S_n = \mathbf{w} \boldsymbol{\epsilon} = w \cdot \sum_{i=1}^n \epsilon_i$ for $w >0$. Moreover, for an arbitrary $i \in \{1, \ldots, n \}$, say $|\epsilon_i| \leq M$ with probability one and $\text{E}\epsilon_i = 0$. Suppose A5 for some fixed $n$ and let $\tau > 0$ be arbitrary. Finally, define a function $h(u) = log(u + 1)(u + 1) - u$. Then $\text{Pr}(|S_n| > \tau) \leq  2 \text{exp} \left ( - M^{-2} \sum_{i=1}^n \text{Av}(\epsilon_i^2)\cdot  h\left ( \{ w \cdot \sum_{i=1}^n \text{Av}(\epsilon_i^2) \}^{-1} \tau M \right ) \right )$ when $\epsilon_i$ is sub-$\mathcal{U}$ for $\forall i$.
\end{theorem}
\begin{proof}
The proof is almost exactly to Theorem 1. Thus, we provide only a sketch. Let $\tau > 0$ and $s > 0$ be arbitrary. Then $\text{Pr} \left ( \text{exp} \{ s S_n \} > \text{exp} \{ s \tau \} \right ) \leq \prod_{i=1}^n  \text{exp} \{ M^{-2} \text{Av}(\epsilon_i^2) \cdot (\text{exp} \{swM\} -1 - swM) 
\}  \text{exp} \{- s \tau \} = \text{exp} \{ M^{-2} \sum_{i=1}^n \text{Av}(\epsilon_i^2) \cdot (exp \{swM\} -1 - swM) -s\tau \}$.
These statements follow from from A5 and Lemma 4. Now, we proceed by minimizing the last expression with respect to $s$. It is easy to observe that the minimum is achieved at $s = (wM)^{-1} log \left ( \{ w \cdot \sum_{i=1}^n \text{Av}(\epsilon_i^2) \}^{-1} \tau M + 1 \right )$. Algebraic rearrangement achieves one side of the bound. Parallel logic, as in Theorem 1, achieves the other.
\end{proof}
\begin{corollary}{(\textbf{Extension of Bernstein's Inequality}.)}
Suppose the same setup as Theorem 2. Then for an arbitrary $\tau > 0$:
$$\text{Pr}(|S_n| > \tau) \leq  2 \text{exp} \left ( - \{ 2w^2 \sum_{i=1}^n \text{Av}(\epsilon_i^2) + 3^{-1} 2 w \tau M \}^{-1} \tau^2 \right )$$
\end{corollary}
\begin{proof}
We simply use Theorem 2 and then note that $h(u) \geq (2 + 3^{-1}2 u)^{-1} u^2$ for $u \geq 0$. Algebraic rearrangement supplies the result.
\end{proof}
To use Theorem 2 and Corollary 3 for confidence sets, one only needs to note that $\text{Av}(\epsilon_i^2)\leq 3^{-1} 2 R_i^2$ for sub-$\mathcal{U}$ variables of the form $\epsilon_i = Y_i - \text{E}Y_i$, where $R_i$ here is the range of $Y_i$. One can then substitute these values into the previous results and use additional substitutions depending on the known features of the conditional $Y_i$ or marginal $Y$. For $\mathcal{U}$ random variables, we already know that $\text{Av}(\epsilon_i^2) = 12^{-1} R_i^2$.
\begin{theorem}
Let $\epsilon_1, \ldots, \epsilon_n$ be continuous, regular $\mathcal{U}$ random variables defined on supports $\mathcal{S}_{1}, \ldots, \mathcal{S}_{n}$ such that for an arbitrary $i \in \{1, \ldots, n \}$, $\text{max}(\mathcal{S}_i) = M_i$ and $\text{E}\epsilon_i = 0$. Let $\mathbf{w}$ be a vector of constants s.t. $S_n = \mathbf{w} \boldsymbol{\epsilon} = \sum_{i=1}^n w_i \epsilon_i$. Suppose A5 for some fixed $n$ and let $\tau > 0$ be arbitrary. Then $\text{Pr}(|S_n| > \tau) \leq  2 exp \{ - (\sum_{i=1}^n w_i^2 R^2_i)^{-1} 6 \tau^2 \}$.
\end{theorem}
\begin{proof}
The proof will largely be omitted since it also mirrors Theorem 1. Instead of using Lemma 3, we can use Lemma 1 for $\mathcal{U}$ variables. This improves the bound on each $\text{Av}(\text{exp} \{ s w_i \epsilon_i\})$ by a factor of $3^{-1}$ within each exponentiation. It also allows for each $w_i$ to be negative. The rest of the proof follows the exact same steps. 
\end{proof}
\begin{example}
Now, say the objective is the estimation of $\boldsymbol{\beta}$ w.r.t. $\mathbf{Y} = \mathbf{x} \boldsymbol{\beta} + \boldsymbol{\epsilon}$. Suppose the mean model is well specified and $\boldsymbol{\epsilon}$ is \textit{any} vector of regular and continuous $\mathcal{U}$ random variables that are densely dependent, but A5 applies. Let $\mathbf{B}_s = \mathbf{w}_s \mathbf{Y}$ be arbitrary. Then $[\mathbf{B}_s -  \sqrt{\sum_{i=1}^n w_{s, i}^2 R_i^2} \cdot \sqrt{6^{-1}\text{log}(2/\alpha)}, \hspace*{1mm} \mathbf{B}_s + \sqrt{\sum_{i=1}^n w_{s, i}^2 R_i^2} \cdot \sqrt{6^{-1}\text{log}(2/\alpha)}]$ is at least a $1-\alpha$ confidence interval for $\beta_s$. Note that if each $Y_i$ is non-negative, since the range of each $\epsilon_i$ equals the range of each $Y_i$, $R_i \leq 2 \text{E}Y_i$ and hence one can feasibly replace $\sqrt{\sum_{i=1}^n w_{s, i}^2 R_i^2}$ with $2 \sqrt{\sum_{i=1}^n w_{s, i}^2 (\hat{\text{E}}Y_i)^2}$ when each $R_i$ is unknown. This will still be accurate asymptotically insofar as $\mu_n = o(n)$. One could also replace the former with $R \sqrt{\sum_{i=1}^n w_{s, i}^2}$, where $R$ is the range of the marginal distribution of $Y$, provided it is known. If $Y$ is sub-$\mathcal{U}$ marginally and $R$ is unknown, one can also replace the former with $2\bar{Y} \sqrt{\sum_{i=1}^n w_{s, i}^2}$ for large enough $n$. Lastly, we could also make use of the sample range of $\hat{e}$. Although downwardly biased, it is still a reasonable choice. It is also consistent until mild regularity conditions. This topic is explored further in Section 6 and the supplementary materials.
\end{example}

\section{A Quick Extension to Estimating Equations\label{sec:5}}

In this section, we extend the results of the previous sections to estimating equations, hence expanding their utility to a larger class of estimators. Consider a random variable $h(Y_i, \boldsymbol{\theta})$ for some measurable function $h$. Although $h(Y_i, \boldsymbol{\theta})$ will often also depend upon a set of fixed constants for $\forall i$, this will be left implicit. The estimating equations of interest here will take the following form, although sometimes it will be relevant to standardize them by a different function of $n$:
\begin{equation}
Q(\theta) = n^{-1} \sum_{i=1}^n \text{E}h(Y_i , \boldsymbol{\theta})
\end{equation}
\begin{equation}
Q_n(\theta) = n^{-1} \sum_{i=1}^n h(Y_i , \boldsymbol{\theta})
\end{equation}
Equations (3) and (4) together are a bedrock of statistical theory and practice and are well understood. To quicken developments, we will draw heavily upon the prior results of the authors mentioned in the introduction. The main contribution of this section is to demonstrate that the variance identities can be used to establish a uniform weak law of large number (UWLLN) for Eq. (4) under exceptionally unfavorable conditions and to show that the estimators resulting from this setup are conditionally or asymptotically additive. The latter fact thus allows for the use of the exponential inequalities of the previous section, even in the face of 'apocalyptic' dependence, provided the right error structure.

Denote $\boldsymbol{\beta} \in \Theta^q$ to be the target parameter of interest once again. If $\boldsymbol{\beta}$ minimizes (maximizes) $Q(\boldsymbol{\theta})$, then $\mathbf{\hat{\beta}_n}$ represents a sequence of random variables that constitute a class of extremum estimators, or M-estimators. For instance, when $h(Y_i, \boldsymbol{\theta})$ is a log-likelihood, $\mathbf{\hat{\beta}_n}$ is a (quasi) maximum likelihood estimator. When $Q(\boldsymbol{\beta}) = 0$, then $\mathbf{\hat{\beta}_n}$ is instead a sequence of root estimators. All generalized method of moment estimators fall into this class \cite{hall2005generalized}. This paper will focus on the latter in the spirit of \citet{yuan1998asymptotics}, although some attention will be provided to the former as well.

The reason for this is simple. The main objective is to (minimally) establish that the sequence of estimators $\mathbf{\hat{\beta}_n}$ derived in concordance with Eq. (3) and Eq. (4) converges in probability to $\boldsymbol{\beta}$. Theory associated with the minimization (maximization) of some $Q(\boldsymbol{\theta})$ can accomplish this; however, since this path ultimately reduces to a consideration of $\text{E}\boldsymbol{\nabla Q(\boldsymbol{\theta})} = \boldsymbol{0}$ as a moment condition, where $\boldsymbol{\nabla}$ is the gradient, one can ultimately explore this case with greater generality. For clarity, $Q$ will henceforth be used for extremum estimating equations, while $U$ will be used for root equations.

Only a couple new adjustments will be made to the typical set of regularity conditions, and partially for simplicity and intuition. The assumptions utilized are provided below:
\begin{enumerate}
\item[R1.] $\Theta^q$ is a compact Euclidean space
\item[R2.] $h(y_i , \boldsymbol{\theta})$ is continuous in $\boldsymbol{\theta}$ and is at least twice continuously differentiable w.r.t. $\boldsymbol{\theta} \in \Theta$ for $\forall y_i$
\item[R3.] $| \partial h^2 / \partial \boldsymbol{\theta} \hspace*{1mm} h(y_i , \boldsymbol{\theta})| \leq Z_i$ for some r.v. $Z_i$ and $\text{sup}_{1 \leq i \leq n} Z_i = O(1)$ for $\forall i$
\item[R4.] $\mathbf{\hat{\beta}_n} \in \Theta$ and $\mathbf{U_n(\hat{\beta}_n)} = \boldsymbol{0}$ for $\forall n$
\item[R5.] $\boldsymbol{\theta} \in \Theta$ and $\mathbf{U(\boldsymbol{\theta})} = \boldsymbol{0} \iff \boldsymbol{\theta} = \boldsymbol{\beta}$
\item[R6.] Denote $\mathcal{L}_n$ as a linear dependency graph s.t. $\zeta = \{h(Y_1, \boldsymbol{\theta}), \ldots, h(Y_n, \boldsymbol{\theta}) \}$. Then $\mu_n = o(n)$ for $\forall \boldsymbol{\theta} \in \Theta$
\item[R7.] Let $\boldsymbol{\theta} \in \Theta$ be arbitrary. Then $\mathbf{\nabla U(\boldsymbol{\theta})}$ is non-singular and $\mathbf{\nabla U_n(\boldsymbol{\theta})}$ converges in probability to $\mathbf{\nabla U(\boldsymbol{\theta})}$
\item[R8.] For all $\boldsymbol{\theta} \in \Theta$ and $\forall i$, $\text{Var} \{ h(Y_i , \boldsymbol{\theta}) \} < \infty$ 
\item[R9.] $Q(\boldsymbol{\beta}) < \underset{\boldsymbol{\theta} \in \Theta; \boldsymbol{\theta} \neq \boldsymbol{\beta}}{\text{inf}} Q(\boldsymbol{\theta})$
\item[R10.] For $\hat{\mathbf{\beta}_n} \in \Theta$, $Q_n (\mathbf{\hat{\beta}_n}) \leq \underset{\boldsymbol{\theta} \in \Theta}{\text{inf}} Q_n(\boldsymbol{\theta}) + o_p(1)$
\end{enumerate}

Note that, in addition to these conditions, there is an implicit assumption that all inverse matrices that are utilized exist. R8 will be useful for establishing uniform convergence in probability. Utilization of R8 is a stronger assumption than is usually required; however, it is not expensive. The usual assumption is that $\text{E} \underset{\boldsymbol{\theta}\in \Theta}{\text{sup}}|h(Y_i, \boldsymbol{\theta})| < \infty$ for $\forall i$ (R8'). In the usual maneuver, R8' is employed in conjunction with R1, R2, and the assumption that that a WLLN exists for the estimating equation to establish stochastic equicontinuity and thereby a UWLLN for Eq. (4). Outside of artificial pathological examples, however, such as when it is false that the variance of Eq. (4) tends to zero, this logic feasibly implies that R8 is also likely to hold. Regardless, R8 simply restricts $\Theta$ to a space of 'reasonable' values. R8 is also trivially true for well-behaved functions and bounded random variables, which is precisely our universe of concern.

Regardless, R8 is not strictly necessary. All results can also be achieved with R8'. As aforementioned, the most important contribution in this section arrives via R6 (or analogous statements). Although it is a slightly higher moment condition, it provides more intuition in practice. Additionally, it makes the proofs quick and easier to follow for non-specialists. 

The next lemma presented establishes that a UWLLN can be obtained even when the mean dependency diverges as a function of $n$. The proof is straightforward and similar to Proposition 5.
\begin{lemma}
Let $Z_i = h(Y_i, \boldsymbol{\theta})$ be any random variable for some measurable function $h$ s.t. R8 holds. Denote $\mu_n$ as the mean degree of the linear dependency graph $\mathcal{L}_n$ for $\zeta = \{Z_1, \ldots, Z_n \}$, say $\bar{Z}(\boldsymbol{\theta}) = n^{-1} \sum_{i=1}^n Z_i$, and denote $W = \underset{\boldsymbol{\theta} \in \Theta}{\text{sup}}|\bar{Z}(\boldsymbol{\theta}) - \text{E} \bar{Z}(\boldsymbol{\theta})|$. If R6 also holds, then $\bar{Z}(\boldsymbol{\theta})$ converges uniformly in probability to $\text{E}\bar{Z}(\boldsymbol{\theta})$.
\end{lemma}
\begin{proof}

Let $\boldsymbol{\theta}_*$ denote the $\boldsymbol{\theta} \in \Theta$ corresponding to $W=\underset{\boldsymbol{\theta} \in \Theta}{sup} |\bar{Z}(\boldsymbol{\theta}) - \text{E}\bar{Z}(\boldsymbol{\theta})|$. Hence, $W = |\bar{Z}(\boldsymbol{\theta}_*) - \text{E}\bar{Z}(\boldsymbol{\theta}_*)|$.

Also note by assumption that $\text{Var}( Z_i)$ is finite for $\forall i$ and that $\mu_n = o(n)$. Let $\epsilon > 0$ be arbitrary. Then:
\begin{align*}
\text{Pr}( W > \epsilon) = Pr ( W^2 > \epsilon^2 ) \leq \epsilon^{-2} \cdot \text{E} \{ \bar{Z}(\boldsymbol{\theta}_*) - \text{E}\bar{Z}(\boldsymbol{\theta}_*) \}^2 \implies \\
\text{Pr} ( W > \epsilon ) \leq  \epsilon^{-2} \cdot Var \left ( \bar{Z}(\boldsymbol{\theta}_*) \right ) \leq \epsilon^{-2} \cdot  (1 + \mu_n \phi_n ) \cdot n^{-2} \sum_{i=1}^n \text{Var} ( Z_i ) \implies \\
\lim_{n \to \infty} \text{Pr}( W > \epsilon ) \leq \lim_{n \to \infty} \epsilon^{-2} \cdot  \underset{1 \leq i \leq n}{\text{sup}} \text{Var}(Z_i) \cdot (1 + \mu_n \phi_n) \cdot n^{-1}  = 0
\end{align*}

Hence, $\underset{\boldsymbol{\theta} \in \Theta}{\text{sup}} |\bar{Z}(\boldsymbol{\theta}) - \text{E} \bar{Z}(\boldsymbol{\theta})|$ converges in probability to zero.
\end{proof}
\citet{xie2003asymptotics} explored a set of exact conditions for weak and strong convergence of generalized estimating equations under the assumption that the weighted sums of clustered observations were uncorrelated with those of other clusters. Again say that $n = \sum_{k=1}^K n_k$ for $K$ unique clusters. They did so under three settings: 1) $K \to \infty$ and $n_M$ is bounded for all $K$, 2) $K$ is bounded but $n_m \to \infty$, and 3) $n_m \to \infty$ as $K \to \infty$. 

An extension to Lemma 5 can be used to cover these settings. Similar to before, say $\mathcal{L}_{n_k}$ refers to the linear dependency graph for random variables within cluster $\zeta_k$ and $\mathcal{L}_{T_k}$ refers to the linear dependency graph of the set of random variables $\{T_{1}, \ldots, T_{K} \}$, where now $T_k = \sum_{j=1}^{n_k} h(Y_j, \boldsymbol{\theta})$. From here, again let $\mu_{n_k}$ signify the mean degree of $\mathcal{L}_{n_k}$ and $\mu_K$ the mean degree of $\mathcal{L}_{T_k}$.
\begin{lemma}
Suppose R1 and R8 and consider $\bar{Z}(\boldsymbol{\theta}) = K^{-1} \sum_{k=1}^K T_k$. Alternatively, for $n_M = \underset{k \in \mathcal{K}}{\text{max}}(n_k)$, define $\bar{Z}_M(\boldsymbol{\theta}) = {n_M}^{-1} \sum_{k=1}^K T_k$ and $n_m = \underset{k \in \mathcal{K}}{\text{min}}(n_k)$ as before. Then:
\begin{enumerate}
\item[1)] Provided $K(n) \to \infty$ as $n \to \infty$ and $n_M$ is bounded for all $K$, if $\mu_K = o(K)$, then $\bar{Z}(\boldsymbol{\theta}) \overset{p}{\to} E\bar{Z}(\boldsymbol{\theta})$ uniformly
\item[2)] If $K(n)$ is bounded but $n_m \to \infty$, and if $\mu_{n_k} = o(n_k)$ for $\forall k$, then $\bar{Z}_M(\boldsymbol{\theta}) \overset{p}{\to} E\bar{Z}_M(\boldsymbol{\theta})$ uniformly
\item[3)] If $n_m \to \infty$ as $K(n) \to \infty$, then $\bar{Z}_M(\boldsymbol{\theta}) \overset{p}{\to} \text{E}\bar{Z}_M(\boldsymbol{\theta})$ uniformly if $\mu_K = O(1)$, $K = o(n_M)$, and $K \cdot \mu_{n_k} = o(n_M)$ for $\forall k$. Alternatively, the same result holds for $\bar{Z}(\boldsymbol{\theta})$ if $\mu_{n_k} = O(1)$ for $\forall k$, $n_M=o(K)$, and $n_M \cdot \mu_K = o(K)$ 
\end{enumerate}
\end{lemma}
\begin{proof}

The first case is very similar to that of Lemma 5, except now the cluster structure of each $T_k$ is being considered. We will only prove the first two statements since the proof is largely repetitive.

\textbf{Case 1:} Suppose $K \to \infty$ and $n_M$ is bounded for all $K$ as $n \to \infty$. Call this bounding constant $M_n$. Furthermore, suppose $\mu_K = o(K)$. Since each $T_{k}$ is a finite linear combination of correlated random variables with finite variance, $\text{Var}(T_{k}) < \infty$ for $\forall k$. This can be seen by observing that $\text{Var}(T_{k}) \leq n_k  \sum_{j=1}^{n_k} \text{Var} \{ h(Y_j, \theta) \} \leq M_n^2 \underset{1 \leq j \leq n_k}{\text{sup}} \text{Var} \{ h(Y_j, \theta) \} <  \infty$ under the premises. Hence, by Lemma 5, $\bar{Z}(\boldsymbol{\theta})$ obeys a UWLLN.

\textbf{Case 2:} Suppose $K(n)$ is bounded but $n_m \to \infty$. Denote $M_K$ and $C_{*}$ as the asymptotic bounds of $K(n)$ and $\phi_{n_k}$ for $\forall k$ respectively. Furthermore, suppose $\mu_{n_k} = o(n_M)$ for $\forall k$. Now, let $W=\underset{\boldsymbol{\theta} \in \Theta}{\text{sup}}|\bar{Z}_M(\boldsymbol{\theta}) - \text{E}\bar{Z}_M(\boldsymbol{\theta})| = |\bar{Z}_M(\boldsymbol{\theta}_*) - \text{E} \bar{Z}_M(\boldsymbol{\theta}_*)|$ for an arbitrary $\epsilon > 0$. Additionally, denote $\sigma^2 = \underset{1 \leq u \leq n}{\text{max}} \text{Var} \{ h(Y_u, \boldsymbol{\theta}_*) \}$, $\mu_{*} = \underset{k \in \mathcal{K}}{\text{max}}(\mu_{n_k})$ WLOG, and $\phi_* \leq C_*$. Then $\text{Pr} \left ( W > \epsilon \right ) \leq \epsilon^{-2} \text{Var}\{ \bar{Z}_M(\boldsymbol{\theta}_*) \} \leq \epsilon^{-2} K \sum_{k=1}^K n_M^{-2} \text{Var} (T_{k; \boldsymbol{\theta}_*})$. And:
$$ \epsilon^{-2} K \sum_{k=1}^K n_M^{-2} \text{Var} (T_{k; \boldsymbol{\theta}_*}) \leq \epsilon^{-2} K \sum_{k=1}^K n_M^{-2}(1+\mu_{n_k} \phi_{n_k}) [ \sum_{j=1}^{n_k} \text{Var} \{ h(Y_{k, j}, \boldsymbol{\theta}_*) \} ]  \leq \epsilon^{-2} K^2 n_M^{-1} (1 + \mu_* \phi_*) \sigma^2 \implies $$
$$\lim_{n \to \infty} \text{Pr}  \left ( W > \epsilon \right ) \leq \epsilon^{-2} M_K^2 \sigma^2 \lim_{n \to \infty} n_M^{-1} (1 + \mu_* C_*) = 0$$
\end{proof}
We note that Lemma 5 can be used with known proof strategies to demonstrate that $\mathbf{\hat{\beta}_n} \overset{p}{\to} \boldsymbol{\boldsymbol{\beta}}$ for M-estimators when it is possible for $\mu_n (n) \to \infty$, and only under the auspices that $\mu_n (n)$ grows at a rate that is sufficiently sub-linear \citep{hall2005generalized}. Now, Lemma 6 is used to demonstrate that $\mathbf{\hat{\boldsymbol{\beta}}_n}$ is consistent and asymptotically additive. For this matter, recall that $\mathbf{U_n(\boldsymbol{\theta})} = n^{-1} \sum_{i=1}^n \mathbf{h(Y_i, \boldsymbol{\theta})}$ is now a $q \times 1$ vector and $\{ \mathbf{\nabla U_n(\boldsymbol{\theta})} \}^{-1}$ is a $q \times q$ matrix.
\begin{proposition}
Suppose R1-R8 in conjunction with $\mathbf{U_n(\boldsymbol{\theta})}$ and also that $\mathbf{\hat{\boldsymbol{\beta}}_n}$ satisfies R4. Additionally, suppose any of the three scenarios stated in Lemma 6. Then $\mathbf{\hat{\boldsymbol{\beta}}_n} \overset{p}{\to} \boldsymbol{\beta}$ and $\sqrt{n}(\mathbf{\hat{\boldsymbol{\beta}}_n} - \boldsymbol{\beta})$ converges in distribution to a vector of additive statistics as $n \to \infty$.
\end{proposition}
\begin{proof}
Recall that the premises of Lemma 6 make use of two different equations. One was standardized by $n_M$ and one was standardized by $K$ w.r.t. $\mathbf{U_n(\boldsymbol{\boldsymbol{\theta}})} = n^{-1} \sum_{k=1}^K \mathbf{T_k} = \sum_{k=1}^K n^{-1} \sum_{j=1}^{n_k} \mathbf{h(Y_{k, j}, \boldsymbol{\theta})}$. Denote these re-expressions $\mathbf{U_M(\boldsymbol{\theta})} = n_M^{-1}n \mathbf{U_n(\boldsymbol{\theta})}$ and $\mathbf{U_K(\boldsymbol{\theta})} = K^{-1}n \mathbf{U_n(\boldsymbol{\theta})}$ respectively. To establish that $\mathbf{U_K(\boldsymbol{\theta})}$ and $\mathbf{U_M(\boldsymbol{\theta})}$ follow a UWLLN under the relevant conditions of Lemma 6, one simply needs to specify an arbitrary component $s$ of either random vector. Since this $s$th component is of the form covered by Lemma 6, under the auspices of R1 and R8, it follows that the $s$th component converges uniformly in probability to its expectation. Therefore, since $s$ was arbitrary, $\mathbf{U_K(\boldsymbol{\theta})}$ converges uniformly in probability to $\text{E}\mathbf{U_K(\boldsymbol{\theta})}$ and $\mathbf{U_M(\boldsymbol{\theta})}$ does the same to $\text{E}\mathbf{U_M(\boldsymbol{\theta})}$. Thus, under Theorem 3 of \citet{yuan1998asymptotics} and R1, R2, R4, and R5 it is then implied that $\mathbf{\hat{\boldsymbol{\beta}}_n} \overset{p}{\to} \boldsymbol{\beta}$. 

Since $\{\mathbf{\nabla U_M(\boldsymbol{\theta})} \}^{-1} \mathbf{U_M(\boldsymbol{\theta})} = \{\mathbf{\nabla U_K(\boldsymbol{\theta}} \}^{-1} \mathbf{U_K(\boldsymbol{\theta})} = \{\mathbf{\nabla U_n(\boldsymbol{\theta})} \}^{-1} \mathbf{U_n(\boldsymbol{\theta})}$ for all $\boldsymbol{\theta} \in \Theta$, it is possible to proceed WLOG utilizing only $\mathbf{U_n(\boldsymbol{\theta})}$.

Now, note that under R1, R2, R3, R7, and the continuous mapping theorem, $\{ \mathbf{\nabla U_n(\boldsymbol{\theta})} \}^{-1}$ converges uniformly in probability to $\{ \mathbf{\nabla U(\boldsymbol{\theta})} \}^{-1}$ for $\forall \boldsymbol{\theta} \in \Theta$. Also, for some $\tilde{\boldsymbol{\beta}} \in \Theta$ that is between $\mathbf{\hat{\boldsymbol{\beta}}_n}$ and $\boldsymbol{\beta}$ and an arbitrary $q \times 1$ vector of constants $\boldsymbol{\lambda}$:
\begin{align*}
\sqrt{n} \cdot \boldsymbol{\lambda}^{\top} \mathbf{U_n(\boldsymbol{\beta})} &= -\sqrt{n} \cdot \boldsymbol{\lambda}^{\top} \{ \mathbf{\nabla U_n(\tilde{\boldsymbol{\beta}})} \} (\mathbf{\hat{\boldsymbol{\beta}}_n} - \boldsymbol{\beta}) 
\end{align*}
Therefore, since $\mathbf{\hat{\boldsymbol{\beta}}_n} \overset{p}{\to} \boldsymbol{\beta}$, it is also true under our conditions that $\sqrt{n} \cdot \boldsymbol{\lambda}^{\top} \mathbf{U_n(\boldsymbol{\beta})} \overset{p}{\to} -\sqrt{n} \cdot  \boldsymbol{\lambda}^{\top} \{ \mathbf{\nabla U(\boldsymbol{\beta})} \} (\mathbf{\hat{\boldsymbol{\beta}}_n} - \boldsymbol{\beta})$, which implies that $\sqrt{n} (\mathbf{\hat{\boldsymbol{\beta}}_n} - \boldsymbol{\beta}) \overset{d}{\to} -  \sqrt{n} \{ \mathbf{\nabla U(\boldsymbol{\beta})} \}^{-1} \mathbf{U_n(\boldsymbol{\beta})}$ by the Cramer-Wold device \cite{feng2013mean, hall2005generalized, yuan1998asymptotics}. This completes the proof since $\sqrt{n} \{ \mathbf{\nabla U(\boldsymbol{\beta})} \}^{-1} \mathbf{U_n(\boldsymbol{\beta})}$ is a vector of random sums.
\end{proof}
It is important to note that $\sqrt{n} (\mathbf{\hat{\boldsymbol{\beta}}_n} -\boldsymbol{\beta})$ will have finite variance without additional stabilization only if the maximum mean degree of linear dependency associated with the score functions is $O(1)$. We provide two general examples to contextualize these results. 
\begin{example}{(\textbf{Quasi-MLEs.})} In this example, standard generalized linear models (GLMs) will be considered for dependent data. Recall that it has often been stated that GLMs are inappropriate when the response variables are dependent.

Say $Q_n = n^{-1} \sum_{i=1}^n ln \{ f(Y_i, \boldsymbol{\theta}) \}$ s.t. $Y_i$ is a member of an exponential family with distribution function $f(y_i, \boldsymbol{\theta})$ for $\forall i$. Also, say $\text{E}Y_i = \mu_i = g(\mathbf{x_i}\boldsymbol{\beta})$ for a differentiable, monotonic canonical link function $g$ and conformable vector of constants $\mathbf{x_i}$. For $\mathbf{Y} \in \mathbb{R}^{n \times 1}$, $\mathbf{\mu} \in \mathbb{R}^{n \times 1}$, and $\mathbf{x} \in \mathbb{R}^{n \times q}$, denote $\mathbf{d} = \mathbf{\partial \boldsymbol{\mu} / \partial \hspace*{1mm} \boldsymbol{\beta}}$ and $\mathbf{W}$ as the diagonal $n \times n$ matrix s.t. $\mathbf{W}_{i,i} = \phi \partial \mu_i / \partial \mathbf{x_i} \boldsymbol{\beta}$ for dispersion constant $\phi$. Then $\mathbf{\nabla Q_n} = \mathbf{U_n} = n^{-1} \mathbf{d^{\top} W^{-1} (\mathbf{Y - \boldsymbol{\mu}})}$. Lastly, define a $\mathcal{L}$ graph for $\{ \epsilon_i \}_{i \in I}$. From the previous results, it is then implied that $\sqrt{n} (\boldsymbol{\hat{\beta}}_n - \boldsymbol{\beta}) \overset{d}{\to} \sqrt{n} \cdot (\mathbf{d}^{\top} \mathbf{W}^{-1} \mathbf{d})^{-1} \mathbf{d}^{\top} \mathbf{W}^{-1} \boldsymbol{\epsilon} =\sqrt{n} \cdot \mathbf{w}\boldsymbol{\epsilon}$, say. Then $\text{Var}\{ \sqrt{n} (\boldsymbol{\hat{\beta}}_n - \boldsymbol{\beta}) \} = n \cdot \mathbf{w}\mathbf{V}\mathbf{w}^{\top} \{\boldsymbol{1}^{p \times p} + \mu_n \boldsymbol{\Gamma} \}$ asymptotically, where $ \{\boldsymbol{1}^{p \times p} + \mu_n \boldsymbol{\Gamma} \}$ has the same definition as in Section 2, and where $\mathbf{V}$ is a diagonal matrix of variances. If $\mu_n = O(1)$ and the conditions of Section 4 hold, then an exponential inequality can be used for constructing confidence sets. Otherwise, one needs to additionally use a proper choice of $ \{1 + \mu_n \phi_{n, s} \}^{-1/2}$ to stabilize the variance of $\hat{\beta}_s$ and to ensure that all relevant objects exist. Of course, if it is believed that a central limit theorem holds, a deterministic correction can also be employed with a Wald-like statistic.
\end{example}

\begin{example}{(\textbf{Iteratively re-weighted least squares.})} Since the variance identity and results of Section 4 can be applied to additive statistics with the necessary error structures, it is applicable to estimators calculated via iteratively re-weighted least squares (IRWLS). This extends their utility to a vast number of non-linear contexts. For instance, it then applies to the (weighted) minimization of $Q_n(\boldsymbol{\theta}) = (pn)^{-1} \sum_{i=1}^n \{ Y_i - g(\mathbf{x_i}\boldsymbol{\theta})  \}^p$ for some measurable function $g$ and $p > 0$, insofar as IRWLS is used as a fitting procedure. Here, the special case will be considered s.t. $p=2$ under a $K$ cluster partition. This case will also use the notation from the previous example. It is apropos to note that the pragmatic form of quasi-MLEs are a special case of this setup when IRWLS is used as an approximate fitting algorithm.

IRWLS is well studied and has been found to have good, reliable qualities. Under the requirement that $\sum_{k=1}^K \mathbf{d^{\top}_{k, t} W^{-1}_{k, t} d_{k, t}}$ is positive definite or the assumptions previously stated hold, the IRWLS estimator $\boldsymbol{\hat{\beta}_t}$ can approximate $\boldsymbol{\hat{\beta}_n}$ to an almost arbitrary precision \citet{yuan1998asymptotics}. 

To briefly show this, let $\tau \equiv 0$ as a theoretical exercise. Then, only imposing the condition that $\mathbf{W_{k, t}}$ be positive definite: $\boldsymbol{\hat{\beta}_t} - \boldsymbol{\hat{\beta}_{t-1}} = ( \sum_{k=1}^K \mathbf{d^{\top}_{k, t} W^{-1}_{k, t} d_{k, t}} )^{-1} \sum_{k=1}^K \mathbf{d^{\top}_{k, t} W^{-1}_{k, t} (Y_k - \boldsymbol{\mu}_{k, t})} \equiv \boldsymbol{0}$, which implies that $K^{-1} \sum_{k=1}^K  \mathbf{d^{\top}_{k, t} W^{-1}_{k, t} (Y_k - \boldsymbol{\mu}_{k, t})} = \mathbf{U_K(\boldsymbol{\hat{\beta}}_{t})} \equiv \mathbf{U_K(\boldsymbol{\hat{\beta}}_{t-1})} \equiv \boldsymbol{0}$.
Under the assumptions of Proposition 9, the IRWLS estimator converges in distribution to the correct target and is asymptotically additive. Hence, once again, all previous results apply for large enough $n$, provided $\mu_n = O(1)$ or proper stabilization is employed, and the other relevant conditions of Section 4 hold. 

However, if one is willing to condition on the sigma-algebra of events generated by some $\boldsymbol{\hat{\beta}_{t-1}}$ following the occurrence of the event that $d(\boldsymbol{\hat{\beta}_{t-1}},  \boldsymbol{\hat{\beta}_{t-2}}) < \tau$ for a sufficiently small $\tau > 0$, and for a long enough run, then $\boldsymbol{\hat{\beta}}_t$ is an additive estimator for \textit{all} finite samples. Therefore, the variance identity and exponential inequalities of Section 4 can be used for conducting cogent inference with functional approximations that minimize error in $\ell_p$ space. Importantly, this can again be accomplished when unknown, intractable, and possibly non-sparse dependency structures are present.
\end{example}
\section{Simulations and a Data Application \label{sec:6}}
This section offers two simulation experiments under unfavorable dependence conditions: one for $\bar{Y}$ and one for $\boldsymbol{\hat{\beta}}$. Both for symmetric $\mathcal{U}$ variables. An approximation to gauge the robustness of A5 will help since, in general, the exact form of $\text{E} (\text{exp} \{s \mathbf{w} \boldsymbol{\epsilon} \})$ is intractable. To this end, we can use the fact that $\text{E} (\text{exp} \{s \mathbf{w} \boldsymbol{\epsilon} \}) \approx 1 + 2^{-1}s^2 \{1 + \mu_n \phi_n \} \sum_{i=1}^n w_i^2 \sigma_i^2$ by Taylor approximation and that the latter expression is bounded by $ \text{exp} \{ 2^{-1}s^2 (1 + \mu_n \phi_n ) \sum_{i=1}^n w_i^2 \sigma_i^2 \}$. For our $\mathcal{U}$ variables, it is then implied that $ \text{exp} \{ 2^{-1}s^2 (1 + \mu_n \phi_n ) \sum_{i=1}^n w_i^2 \sigma_i^2 \} \leq \text{exp} \{ 24^{-1} s^2 \sum_{i=1}^n w_i^2 R_i^2\}$, at least as an approximate rule of thumb for A5. Although the value on the right side of the inequality bounds the functional average specified in A5, it does so tightly. For instance, if $w_i = n^{-1}$ for $\forall i$, each error variable has the same upper bound for its support, and $s= C \sqrt{n}$ for some positive constant $C$, $\text{Av}_* (\text{exp} \{s \mathbf{w} \boldsymbol{\epsilon} \}) \to \text{exp} \{24^{-1} C^2 R^2  \}$ quickly. Regardless, further algebraic manipulation of this setup implies the following rough bound on the summary values of the dependency structure: $ \mu_n \phi_n \leq 12^{-1} \{\sum_{i=1}^n w_i^2 \sigma_i^2 \}^{-1} \sum_{i=1}^n w_i^2 R_i^2 - 1$. Homogeneity of variances and ranges results in a simpler weight-invariant bound of $12^{-1} \sigma^{-2} R^2 -1$. For simplicity, all ranges and variances will be equal. Insofar as this approximate bound holds, A5 should approximately hold as well and inference should remain robust, even with mild violations.

All experiments possess $N=10,000$ simulations. To get a more accurate assessment of robustness to violations of A5, however, we also compare $\hat{A} = \text{max}(N^{-1} \sum_{i=1}^N \text{exp} \{s \mathbf{w} \boldsymbol{e}_i \}, N^{-1} \sum_{i=1}^N \text{exp} \{-s \mathbf{w} \boldsymbol{e}_i \})$ to $\text{Av}_* = \prod_{i=1}^n \{s w_i M_i \}^{-1} \text{sinh}(s w_i M_i)$, which is an exact form of $\text{Av}_* (\text{exp} \{s \mathbf{w} \boldsymbol{\epsilon} \})$ for $\mathcal{U}$ observations of our type. Investigating all arbitrary values of $s$ isn't possible or necessary. Instead, we set $s= \{M^2  \cdot c_* \cdot \sum_{i=1}^n w_i^2 \}^{-1/2}6 \cdot \{ 6^{-1} \text{log}(2/\alpha) \}^{1/2}$ for a context-dependent value of $c_*$ that limits the size of the exponential and average values for readability. This value of $s$ also corresponds to the optimal value of Theorem 3 with homogeneous ranges and $\tau$ set to the required expression for conservative, two-sided $1-\alpha$ confidence sets. Additionally, we employ a value of $s$ that is $O(\sqrt{n})$ since this order better corresponds to the rule of thumb. Since we will be evaluating $95$\% confidence sets, $\alpha = .05$ for all experiments.

\paragraph{Setup} To test 'apocalyptic' scenarios, each simulation employs fully connected linear dependency graphs. In other words, we set $\mu_n = n-1$ for $n \in \{ 100, 500, 1500\}$. For the $\bar{Y}$ setup, $Y_i \sim Beta(\alpha, \alpha)$. This translates our rule of thumb to $3^{-1} (2\alpha +1) -1$. The first set of simulations uses $\alpha = 10$ universally. According to the rule of thumb, then, confidence sets should start to lose their nominal values when $\phi_n = .06, .01$, and $.004$ respectively. For comparison, we establish a baseline at $\phi=0$ and make use of values that neighbor these breakdown points. In many unfavorable settings, we would still not expect $\mu_n$ to always attain its maximum or for $\phi_n$ to take moderate values when it does. Hence, this series of simulations is built to demonstrate the robustness of the results for finite sample inference in practical settings since decent performance under these conditions suggest an adequate level of dependability in more modest ones. A second experiment fixes $\phi$ to $.1$ for $n=500$ observations and varies only $\alpha$ for values in $\{10, 25, 50, 100 \}$ to further examine the relationship between the range-variance ratio and robustness to statistical dependence. For all of these simulations, $c_* = 10$. For the $\mathcal{U}$ regression setup, each $\epsilon_i \sim TruncNorm(M=20, \mu = 0, \sigma^2 =25)$. The linear model is characterized by $Y_i = 20 + 10 \cdot t_i + \epsilon_i$, where each $t_i$ is a fixed draw from $T_i \sim TruncNorm(m=-5, M=5, \mu =1, \sigma^2 = 1)$. Here, $c_* = \sigma = 5$ for control of exponential size.

Dependencies in the outcome variables are induced with Gaussian copulas. Literature on this method is available elsewhere \citep{embrechts2001modelling, demarta2005t}. \textbf{R} version 4.2.2 statistical software and the package 'copula' are employed for all experiments \citep{R, yan2007enjoy}. The Beta distribution simulations make use of an exchangeable correlation matrix with off-diagonal cells populated by $\phi_n$. The strategy for the regression simulations is more complicated. Essentially, we specify an unstructured correlation matrix and set its non-diagonal values to the corresponding elements of $ 25^{-1}\phi_* n^2 \cdot \mathbf{w}_1\mathbf{w}^{\top}_1$ for $\phi_* \in \{0, .05, .1, .15 \}$. The first row of weights is used as a basis since the product of its values will often match the valence of the weight products in the sum of covariances. This will induce a dense mosaic of positive summands and inflate the variance.

As usual, empirical coverage is estimated by $N^{-1} \sum_{i=1}^N 1_{\theta \in \mathcal{C}}$, where $\mathcal{C}$ is the constructed confidence set. Three coverage values are ultimately estimated---$\hat{\text{CI}}_{Wald}, \hat{\text{CI}}_{\mathcal{U}}$, and $\hat{\text{CI}}_{\hat{R}}$---although the third is estimated for regression simulations only. The estimated Wald confidence sets make use of cluster-robust standard errors for the Beta simulations and robust standard errors in the style of generalized estimating equations with an exchangeable correlation structure for the regression cases \citep{hojsgaard2006r}. They suppose asymptotic normality. To mimic the specification of a partially correct but invalid partition, the outcome variables are assigned to $10^{-1}n$ clusters sequentially. The $\hat{\text{CI}}_{\mathcal{U}}$ estimator corresponds to the confidence sets that are constructed in accordance with Theorem 3. Importantly, these sets treat $M$ as known. This is not the case for the confidence sets targeted by $\hat{\text{CI}}_{\hat{R}}$. Here, $M$ is treated as unknown. The sets are still constructed in accordance with Theorem 3. However, the sample range of the residuals are used in place of $2M$. This method is explored to gauge the robustness of the plug-in strategy, which will often be required in practice. Finally, the average lower and upper endpoint of the $\mathcal{U}$ constructed confidence sets for the case that $M$ is known are also provided for reference. The results of these simulations are available in \textbf{Table 1}, \textbf{Table 2}, and \textbf{Table 3}. For each table, A5 is estimated to hold when $\hat{A} \leq \text{Av}_*$ and is estimated to be violated when '$>$' is shown.
\begin{table}[H]
\centering
\caption{Beta($10, 10$) Simulations} 
\begin{adjustbox}{width=1\textwidth}
\begin{tabular}{lccccccccc}
  \hline
   \\\\[-3.5\medskipamount]
 & $n$ & $\phi_n$ & Mean Lower Endpoint & Mean Upper Endpoint & $\hat{\text{CI}}_{Wald}$ & $\hat{\text{CI}}_{\mathcal{U}}$ & $\hat{A}$&   & $\text{Av}_*$ \\ 
  \\\\[-3.5\medskipamount]
 \hline
  \\\\[-3.5\medskipamount]
 & 100 & 0 & 0.42159 & 0.57841 & 0.92 & 1 & 10.52 & $<$ & 10.60 \\ 
   &  & 0.06 $\dagger$ & 0.42159 & 0.57841 & 0.51 & 0.99 & 10.60 & $=$ & --- \\ 
   &  & 0.1 & 0.42158 & 0.57840 & 0.41 & 0.97 & 10.66 & $>$ & --- \\ 
   &  & 0.2 & 0.42158 & 0.57840 & 0.29 & 0.88 & 10.81 & $>$ & --- \\ 
   & 500 & 0 & 0.46486 & 0.53499 & 0.92 & 1 & 192.57 & $<$ & 194.25 \\ 
   &  & 0.01 $\dagger$ & 0.46476 & 0.53489 & 0.55 & 0.996 & 193.65 & $<$ & --- \\ 
   &  & 0.05 & 0.46458 & 0.53471 & 0.29 & 0.84 & 198.45 & $>$ & --- \\ 
   &  & 0.1 & 0.46444 & 0.53457 & 0.2 & 0.68 & 204.8 & $>$ & --- \\ 
   & 1500 & 0 & 0.47974 & 0.52023 & 0.92 & 1 & 9059.27 & $<$ & 9132.95 \\ 
   &  & 0.004 $\dagger$ & 0.47972 & 0.52021 & 0.51 & 0.99 & 9128.91 & $<$ & --- \\ 
   &  & 0.01 & 0.47971 & 0.52020 & 0.36 & 0.93 & 9236.57 & $>$ & --- \\ 
   &  & 0.02 & 0.47969 & 0.52018 & 0.26 & 0.80 & 9420.2 & $>$ & --- \\ 
    \\\\[-3.5\medskipamount]
  \hline
\end{tabular}
\end{adjustbox}
\caption*{\small The $\dagger$ symbol denotes the predicted threshold value s.t. coverage will begin to falter for fully connected linear dependency graphs.}
\end{table}
\begin{table}[H]
\centering
\caption{Beta($\alpha, \alpha$) Simulations: $n=500, \phi_n = .1$} 
\begin{adjustbox}{width=1\textwidth}
\begin{tabular}{lccccccccc}
  \hline
 \\\\[-3.5\medskipamount]
 & $\alpha$ & Threshold & Mean Lower Endpoint & Mean Upper Endpoint &  $\hat{\text{CI}}_{Wald}$ &  $\hat{\text{CI}}_{\mathcal{U}}$ & $\hat{A}$ &  & $\text{Av}_*$ \\ 
  \\\\[-3.5\medskipamount]
  \hline
   \\\\[-3.5\medskipamount]
& 10 & 0.012 & 0.46444 & 0.53457 & 0.203 & 0.68 & 204.79 & $>$ & 194.25 \\ 
 & 25  & 0.032 & 0.46462 & 0.53475 & 0.203 & 0.88 & 8.22 & $>$ & 8.21 \\ 
 & 50 & 0.065 & 0.46471 & 0.53484 & 0.204 & 0.97 & 2.86 & $=$ & 2.86 \\ 
 & 100  & 0.132 & 0.46477 & 0.53491 & 0.204 & 0.997 & 1.69 & $=$ & 1.69 \\ 
  \\\\[-3.5\medskipamount]
   \hline
\end{tabular}
\end{adjustbox}
\caption*{\small The 'Threshold' column represents the predicted breakdown point at $\alpha$ with $n$ held constant.}
\end{table}
\begin{table}[H]
\centering
\caption{Regression Simulations}
\begin{adjustbox}{width=1\textwidth}
\begin{tabular}{rcccccccccc}
  \hline
   \\\\[-3.5\medskipamount]
$\beta_0=20$  &$n$ & $\phi_*$ & Mean Lower Endpoint & Mean Upper Endpoint & $\hat{\text{CI}}_{Wald}$ & $\hat{\text{CI}}_{\mathcal{U}}$ & $\hat{\text{CI}}_{\hat{R}}$ & $\hat{A}$ &  & $\text{Av}_*$ \\
 \\\\[-3.5\medskipamount]
  \hline
   \\\\[-3.5\medskipamount]
 & 100 & 0 & 15.458 & 24.556 & 0.884 & 1 & 1 & 1.047 & $<$ & 1.262 \\ 
   &  & 0.05 & 15.459 & 24.557 & 0.823 & 1 & 0.999 & 1.066 & $<$ & --- \\ 
   &  & 0.1 & 15.460 & 24.558 & 0.770 & 1 & 0.995 & 1.086 & $<$ & --- \\ 
   &  & 0.15 & 15.461 & 24.559 & 0.725 & 1 & 0.987 & 1.106 & $<$ & --- \\ 
   & 500 & 0 & 17.934 & 22.063 & 0.904 & 1 & 1 & 1.044 & $<$ & 1.247 \\ 
   &  & 0.05 & 17.933 & 22.061 & 0.672 & 1 & 0.990 & 1.144 & $<$ & --- \\ 
   &  & 0.1 & 17.932 & 22.061 & 0.549 & 0.993 & 0.956 & 1.254 & $>$ & --- \\ 
   &  & 0.15 & 17.931 & 22.060 & 0.475 & 0.978 & 0.914 & 1.374 & $>$ & --- \\ 
   & 1500 & 0 & 18.833 & 21.170 & 0.902 & 1 & 1 & 1.046 & $<$ & 1.258 \\ 
   &  & 0.05 & 18.836 & 21.173 & 0.487 & 0.979 & 0.938 & 1.382 & $>$ & --- \\ 
   &  & 0.1 & 18.838 & 21.175 & 0.366 & 0.908 & 0.835 & 1.840 & $>$ & --- \\ 
   &  & 0.15 & 18.839 & 21.176 & 0.306 & 0.839 & 0.753 & 2.474 & $>$ & --- \\ 
    \\\\[-3.5\medskipamount]
   \hline
    \\\\[-3.5\medskipamount]
$\beta_1=10$  &$n$ & $\phi_*$ & Mean Lower Endpoint & Mean Upper Endpoint & $\hat{\text{CI}}_{Wald}$ & $\hat{\text{CI}}_{\mathcal{U}}$ & $\hat{\text{CI}}_{\hat{R}}$ & $\hat{A}$ &  & $\text{Av}_*$ \\ 
 \\\\[-3.5\medskipamount]
 \hline
  \\\\[-3.5\medskipamount]
 & 100 & 0 & 6.836 & 13.151 & 0.885 & 1 & 1 & 1.023 & $<$ & 1.111 \\ 
   &  & 0.05 & 6.835 & 13.150 & 0.854 & 1 & 0.999 & 1.028 & $<$ & --- \\ 
   &  & 0.1 & 6.835 & 13.150 & 0.823 & 1 & 0.997 & 1.033 & $<$ & --- \\ 
   &  & 0.15 & 6.834 & 13.149 & 0.794 & 1 & 0.996 & 1.037 & $<$ & --- \\ 
   & 500 & 0 & 8.546 & 11.451 & 0.903 & 1 & 1 & 1.022 & $<$ & 1.114 \\ 
   &  & 0.05 & 8.546 & 11.452 & 0.758 & 1 & 0.998 & 1.047 & $<$ & --- \\ 
   &  & 0.1 & 8.547 & 11.452 & 0.662 & 1 & 0.988 & 1.072 & $<$ & --- \\ 
   &  & 0.15 & 8.547 & 11.453 & 0.594 & 0.997 & 0.971 & 1.098 & $<$ & --- \\ 
   & 1500 & 0 & 9.189 & 10.807 & 0.903 & 1 & 1 & 1.024 & $<$ & 1.114 \\ 
   &  & 0.05 & 9.188 & 10.805 & 0.601 & 0.998 & 0.985 & 1.097 & $<$ & --- \\ 
   &  & 0.1 & 9.187 & 10.804 & 0.482 & 0.976 & 0.936 & 1.176 & $>$ & --- \\ 
   &  & 0.15 & 9.186 & 10.804 & 0.409 & 0.942 & 0.882 & 1.260 & $>$ & --- \\ 
    \\\\[-3.5\medskipamount]
   \hline
\end{tabular}
\end{adjustbox}
\end{table}
\textbf{Table 1} shows that the coverage of the confidence sets resulting from A5 and Theorem 3 starts to dissipate in quality around the predicted points, and also when A5 ceases to be true. \textbf{Table 2} provides evidence that, even if A5 does not hold exactly, the robustness of the confidence sets that result from its employment is a function of the extrema of each $\mathcal{S}_i$. In general, random variables with higher absolute extremes allow for more dense dependencies to exist without undermining inference: a fact that has been apparent since at least Hoeffding's lemma. Notably, \textbf{Table 3} substantiates the utility of the plug-in strategy. Although, as expected, it is not as robust as when the upper extreme of the support is known, the simulation evidence shows that the confidence sets that use the sample range maintain their semi-conservative nominal coverage value while A5 holds. Each table also shows that the standard methods employed absolutely fail to uphold nominal coverage. Overall, the key point supported by these experiments is that A5 is a feasible condition that allows for cogent finite sample inference in many important settings, and even when every outcome variable is statistically dependent. Additive statistics---from linear models and from estimating equations---are a critical tool for scientific discovery. Our job here was to demonstrate that they remain dependable enough for cogent finite sample inference in complicated modern settings. This has been accomplished, at least in some dimension.

\subsection{Carbon Dioxide and Global Warming}

In this section, we estimate the association between global temperature change and carbon dioxide levels between the years of 1979 and 2022. Monthly averages for global carbon dioxide levels ($\text{CO}^2$; ppm) and global temperature anomalies (Temp; °C) were acquired from the Global Monitoring Laboratory and the Goddard Institute for Space Studies respectively \citep{carbon, temp}. More information pertaining to the latter source and the methods utilized for measurement are obtainable elsewhere \citep{lenssen2019improvements}. Although the estimation of causality is beyond the scope of this analysis, monthly data for an industrial production index for the G-20 countries (Index) was also accessed for these years to act as a rudimentary adjusting variable \citep{g20}. The index score for each country was summed for each month to construct a single index. Alternative metrics were also considered for this analysis. For instance, measurements on global population growth and the proportion of landmass covered by forests each year were also obtained. However, they were not utilized due to issues of multicollinearity. The approach of this paper is relevant for this question since, although useful and informative as conceptualizations, there is no reason for any complex ecological time series to strictly abide by the neat schematics of a typical moving average or auto-regressive model. Unknown unknowns likely impact the process across time and location. Even if auto-correlations diminish, this does not imply that dependencies do.

\paragraph{Methods} The cardinal association is investigated w.r.t. two units of time: monthly and yearly. Only one time lag is utilized in both cases. This provided $n=527$ observations for the monthly analysis and $n=43$ observations for the yearly. For the latter exploration, all monthly variables are averaged for each year. For the former, a categorical variable for the season (December-February: Winter; March-May: Spring; June-August: Summer; September-November: Fall) is constructed to adjust for additional time trends. $\text{CO}^2$ and Index are also log-transformed for both analyses. Two baseline models are estimated: $\text{Temp}_t = \beta_0 + \beta_1 \text{Temp}_{t-1} + \beta_2 \text{log} \{ \text{CO}_{t-1}^2 \} + \beta_2 \text{log} \{ \text{Index}_{t-1} \}+ \sum_{i=1}^{3} \alpha_i \text{Season}_i + \epsilon_t$ and $\overline{\text{Temp}}_t = \beta_0 + \beta_1 \overline{\text{log} \{ \text{CO}^2 \}}_{t-1} + \beta_2 \overline{\text{log} \{ \text{Index} \}}_{t-1} + \epsilon_t$. Non trend adjusting variables are dropped from the model if they do not induce at least a $10$ percent change in the magnitude of the estimate. All models are fitted via ordinary least squares with a Type I error rate of $\alpha = .05$.

Confidence sets are constructed by way of Theorem 3, which also applies to stochastic regressions. For instance, when $\mathbf{W} = (\mathbf{X}^{\top} \mathbf{X})^{-1} \mathbf{X}^{\top}$, then $\mathbf{B}_s = \beta_s + \sum_{i=1}^n W_{s, i} \epsilon_i$ is additive in $Z_{s, i} = W_{s, i} \epsilon_i$. Strict stationarity, regularity, and symmetry of the $\{ Z_{s, i} \}_{i \in I}$ is sufficient for the application of the theorem in this analysis. Pertinently, $\mathbf{B}_s$ is not unbiased for $\beta_s$ for the first model since $\text{Temp}_{t-1}$ appears on the right-hand side of the equation. However, it can still be consistent. Consistency is obtained if A3 applies to the mean degrees of the graphs associated with $\{ Z_{s, i} \}_{i \in I}$ and $\{ \epsilon_i \}_{i \in I}$. All conditions are feasibly checked with time series plots, histograms, and empirical CDF plots for $\{W_{s, i} \hat{e}_i \}_{i \in I}$ and $\{ \hat{e}_i \}_{i \in I}$. Here, we use the plug-in estimator $\hat{R}_s = \underset{i \in I}{\text{max}}(W_{s, i} \hat{e}_i) - \underset{i \in I}{\text{min}}(W_{s, i} \hat{e}_i)$. Again, this will underestimate the true range for finite samples. However, it is a consistent estimator under the mild regularity conditions we suppose. Proof of this is offered in the supplementary material. Moreover, the magnitude of the underestimation is not likely to undercut the conservatism of the method. In accordance with Theorem 3, then, confidence sets have the form $\mathbf{B}_s \pm \sqrt{n} \hat{R}_s \cdot \{6^{-1} \text{log} (2/\alpha) \}^{1/2}$. We use the same rule of thumb to gauge the robustness of A5: $\phi_{s, n} \leq (n-1)^{-1} \cdot  \{ 12^{-1} S_{W_s \hat{e}}^{-2} \hat{R}^2_s -1 \}$, where $S_{W_s \hat{e}}^{2}$ is the sample variance of the $W_{s, i} \hat{e}_i$. This approximate bound is compared to auto-correlation estimates. Using the sample range for the calculation of the rule of thumb helps to counter its limitation as a plug-in.
\paragraph{Results} Including $\text{log} \{ \text{Index}_{t-1} \}$ resulted in an approximate .06 decrease in the effect estimate for $\text{log} \{ \text{CO}_{t-1}^2 \}$ after also adjusting for the season. This did not meet the specified threshold. Hence, $\text{log} \{ \text{Index}_{t-1} \}$ was dropped from the model. Season was retained to adjust for time trends. Per every unit increase in the logarithm of the previous month's $\text{CO}^2$ levels, there is an estimated mean increase of 1.53 °C ($\geq 95$\% CI: .25, 2.8) in the next month's global temperature after accounting for the season and the previous month's global temperature. This effect was larger than the estimated .6 °C change in average global temperature ($\geq 95$\% CI: .24, .95) per unit increase in the previous month's temperature. For the annual model, $\overline{\text{log} \{ \text{Index} \}}_{t-1}$ was found to decrease the effect of interest by approximately .07 percent and was therefore removed. Per each unit change in the logarithm of the average yearly $\text{CO}^2$ level of the previous year, the mean global temperature of the next year increased by an estimated 3.87 °C ($\geq 95$\% CI: 3, 4.73).
\paragraph{Model checking} \textbf{Figure 1} demonstrates the time series, empirical CDF, histogram, and auto-correlation plots for the $W_{s, i} \hat{e}_i$ related to $\text{CO}^2$ for each model. Both histograms provide evidence of symmetric and roughly regular distributions since they are unimodal, bell-shaped, and (mostly) vary in a smooth fashion over an interval with endpoints that are roughly symmetric about zero. The empirical CDF plots also display roughly the same amount of area above and below the curve, providing evidence of $\mathcal{U}$ status. For our rule of thumb, we require $\phi_{s, n}$ to be bounded by $0.014$ and $0.018$ respectively for the monthly and annual models. It is reasonable to assert that this bound is fulfilled due to the positive-negative oscillating character of the auto-correlation plots. With the number of lags set to $\lfloor 10 \text{log}_{10} (n) \rfloor$ and $\lfloor 2^{-1} (n-1) \rfloor$, $\hat{\phi}_{s, n} = -.007$ and $-.002$ respectively for the monthly data and $-.03$ and $-.02$ for the annual data, of course under the supposition of strict stationarity. The time series plot for the monthly data shows no remarkable departure from this latter supposition, although there does appear to be some departures for the annual data. \textbf{Figure 2} shows the auto-correlation and time series plots for the residuals alone. These plots appear consistent with the assumption of stationarity and sub-linear mean dependencies.
\begin{figure}[H]
\centering
\caption{Model Diagnostics for $\text{CO}^2$ Weighted Residuals}
\includegraphics[scale = .5]{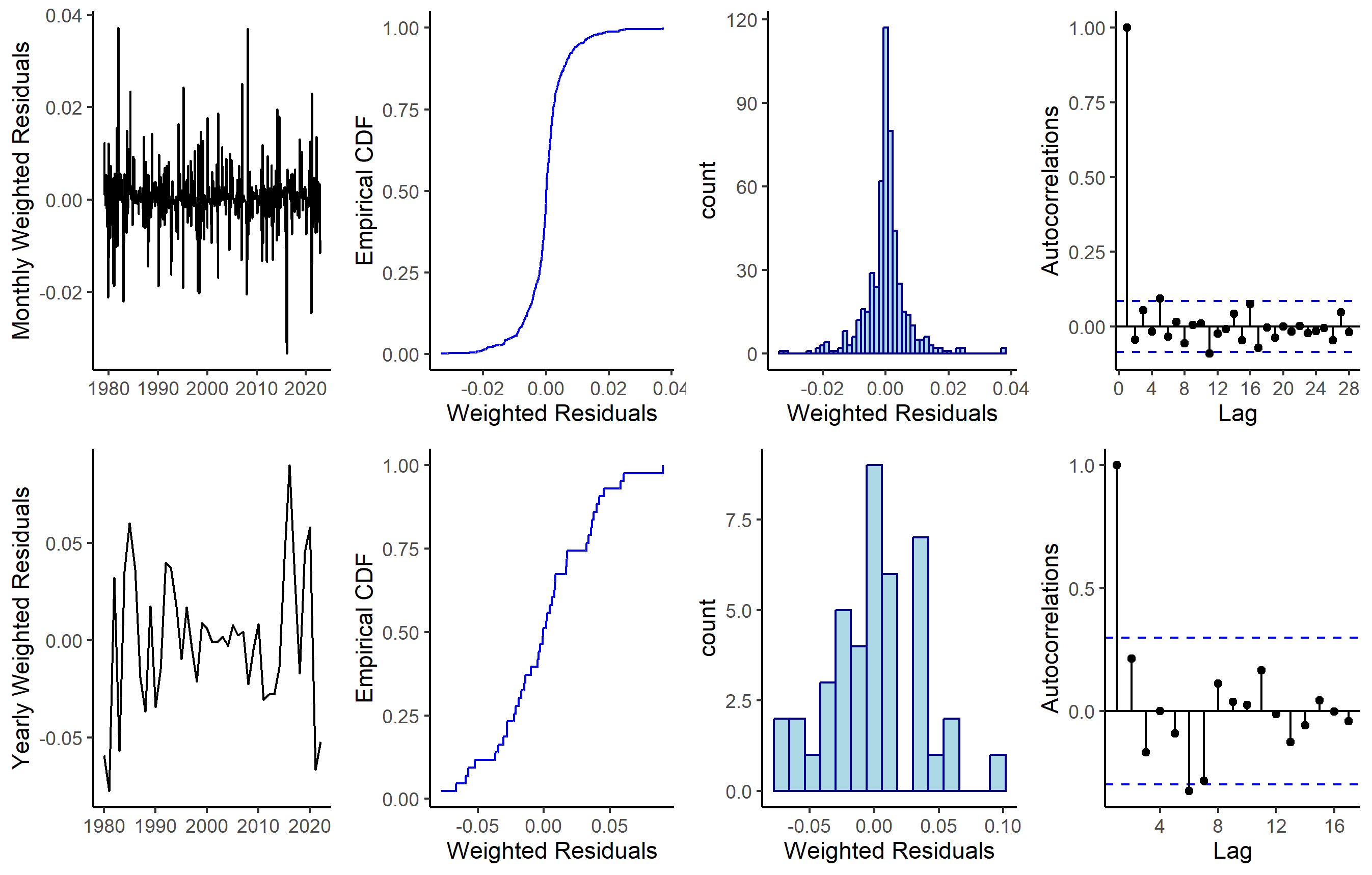}
\caption*{\small The auto-correlation plots make use of the \textit{acf} function of the \textbf{stats} package in \textbf{R} under default settings \citep{stats}. More lags are considered when estimating $\phi_{s, n}$.} 
\end{figure}
\begin{figure}[H]
\centering
\caption{Model Diagnostics for $\text{CO}^2$ Residuals}
\includegraphics[scale = .5]{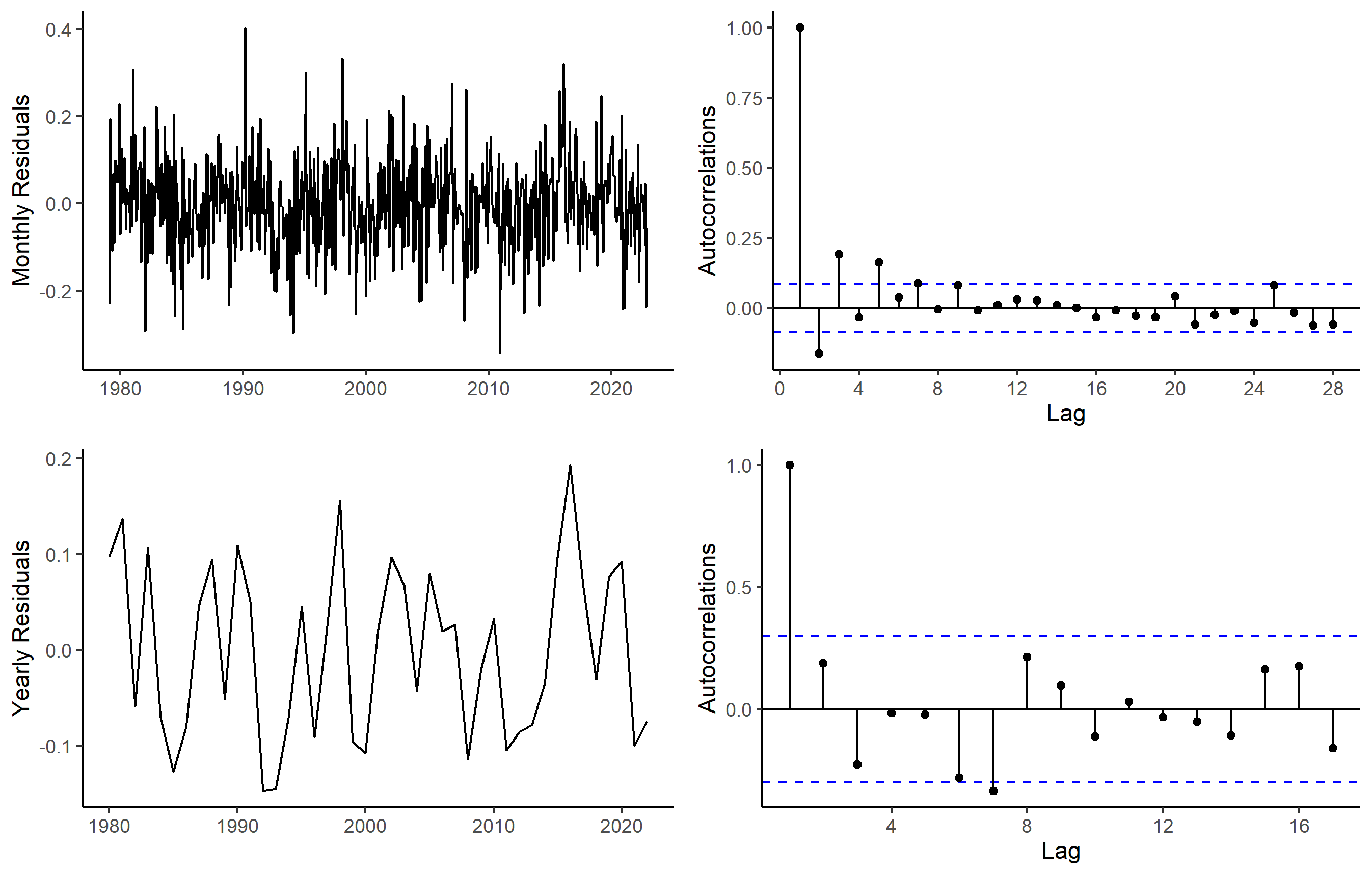}
\end{figure}
Altogether, there is sufficient evidence that the previous year's monthly or yearly average $\text{log}(\text{CO}^2)$ levels predict a non-negligible increase in the average global temperature of the following month or year. Since this approach places no specific constraint on the underlying structure of statistical dependencies, e.g., it does not adopt preconceptions about strong or weak mixing or some form of $m$-dependence, the results of this analysis are arguably stronger. They arrive with fewer theoretical caveats. Nonetheless, important limitations lurk. Model diagnostics still make use of the residuals, which are a biased and constrained representation of the true error. This situation, however, applies to all regression diagnostic procedures that make use of the residuals and does not rule out their use for the estimation of restricted portions of auto-correlation. Ultimately, these constraints complicate, but do not eliminate, the utility of these objects' employment as a diagnostic mechanism.
\section{Concluding remarks\label{sec:7}}
This manuscript accomplished three main objectives. First, it established a small set of related identities for the variance of a vector of random sums. These identities require reasoning about the mean number of outcome variables correlated within a sample and their average correlation only. Since the true dependency structure of any collection of random variables is safely posited to be unknown and empirically unidentifiable in whole, removing the strict need to specify $n \choose 2$ parameters is useful. It was shown that a researcher can elect to reason about these two intuitive summary constants instead, or that she can employ them in conjunction with popular covariance modeling methods to capture at least some of the variability that is missed by an invalid specification. Although these constants are unknown, so are the $n \choose 2$ covariance parameters that statisticians specify on a daily basis. Furthermore, the cogent defense of a conservative choice for these values is a much less demanding task than the alternative in a majority of circumstances. Pertinently, these identities were also used to affirm the consistency of additive estimators---including cluster-robust variance estimators w.r.t. their identified portion of the overall variance---under the very general condition that the average number of correlated variables in a sample is asymptotically sub-linear as a function of $n$. For cluster variance estimators, this was shown to be the case even when no valid partition of the sample exists. A second accomplished objective was to extend these results to estimating equations and hence to the estimators of statistical approaches such as the generalized linear model.

The third and most important contribution of this paper was to prove a sharpened version of Hoeffding's inequality for a class of commonly encountered random variables. Notably, it was proven that this inequality can apply even when every single outcome variable in a sample is statistically dependent, insofar as the magnitude of their average correlation is at least moderately controlled. This result is certainly valuable for many fields where the assumption of weak or local dependence is especially untenable, such as in climate science, social network analysis, finance, and really any ecological or sociological domain. That said, more work is due. Like all statistical models, the valid application of this inequality relies on a set of assumptions that can only be feasibly verified in practice. Although imperfect, however, the diagnostic processes available are equivalent to those used to check the assumptions of common regression models. In this sense, the approach established here at least possesses nomological validity.

\section*{CRediT Statement}

\textbf{Shane Sparkes}: Conceptualization, Methodology, Validation, Formal analysis, Data Curation, Writing - Original Draft, Writing - Review and Editing, Visualization. \textbf{Lu Zhang}: Writing - Review and Editing, Supervision, Project administration.

\section*{Acknowledgments}

We would like to thank the referees for their thoughtful commentary, encouragement, and insights. Moreover, we would also like to acknowledge Dr. Thomas Valente and Dr. Juan Pablo Lewinger for their commentary on a first scalar form of the variance identity.


\bibliographystyle{myjmva}
\bibliography{ms}
\ifarXiv
    

\section*{Supplementary material}

# Properties and Deviations of Random Sums of Densely Dependent Random Variables: Supplemental document

## 1. Introduction

The manuscript introduces the $\mathcal{U}$ class of random variables and utilizes a handful of related identities to offer intuition about A5. The purpose of this additional document is to prove some of the claims made in the paper about these variables. They are mostly foundational and elementary, although this does not mean that they are not informative. Moreover, this document seeks to further familiarize the reader with this type of variable and to justify model diagnostic methods and plug-in methods under the assumption of copious dependencies. Stated informally, the following ideas will be explored and proven here:

1. $\mathcal{U}$ variables are characterized by the absence of a type of intrinsic linear distributional bias
2. Equivalent definitions of $\mathcal{U}$ status exist that connect these variables to a generalized notion of sum-symmetry
3. Every continuous uniform distribution can be decomposed into the sum of an arbitrary unimodal $\mathcal{U}$ variable with sub-uniform left tails and an additional error variable s.t. the latter is uncorrelated with all functions of the former
4. $\mathcal{U}$ variables (in a sense) generalize the classic regression condition of exogeneity
5. The preservation of the intrinsic linear bias of a distribution during sampling is sufficient for unbiased moment estimation, and even with informative sampling
6. When $\max(\mathcal{S}_i) = M$ for $\forall i$, where once again $\mathcal{S}_i$ is the support of $Y_i$, then $Y_{(n)} = \max(Y_1, Y_2, \ldots, Y_n)$ is a consistent estimator of $M$ under mild regularity conditions, even when dependencies are plethora

For simplicity, we will be avoiding any statements that are required for the existence of referenced objects. We require some additional setup for the fourth and fifth points. Consider a population of random variables $P = \{Y_k\}_{k \in K}$ for $K = \{1, \ldots, N\}$ and a sample $\zeta = \{Y_i\}_{i \in I}$ for $I \subset K$ and hence $\zeta \subset P$. We will say $\boldsymbol{\delta} = (\delta_1, \ldots, \delta_N)$ s.t. $\delta_k = 1$ if and only if a $Y_k \in P$ is an element in $\zeta$ and it is also uncensored. We will also assume that $\mathrm{E}\delta_k > 0$ for all $k$. It is well known, then, that an arbitrary $Y_i \in \zeta$, say $Y_{\delta_i}$, actually has a sampling distribution that is not in general equal to its population counterpart, i.e., $Y_{\delta_i} \sim f_\delta(y_i) = \{\mathrm{E}\delta_i\}^{-1}\mathrm{E}(\delta_i|y_i)f(y_i)$, where $f(y_i)$ is the density or mass function of $Y_i$. Provided this setup, we only need one new main assumption for this exploration:

C1. Let $\mathcal{S}$ be the support of the distribution of $Y$ and $\mathcal{S}_\delta$ be the support of the distribution of $Y_\delta$. Then $\mathcal{S} = \mathcal{S}_\delta$.

Essentially, C1 stipulates that sampling preserves the support. This is ultimately a very mild condition since it places no additional constraint on the distortion caused by informative (or biased) sampling mechanisms. All propositions are numbered separately from the main document, i.e., numbering begins anew here.

## 2. $\mathcal{U}$ Random Variables

The first proposition characterizes a precursor to the distance between $\mathrm{E}Y$ and $\mathrm{Av}(Y)$ by relating it to the aforementioned type of intrinsic distributional bias. It shows that the difference between $\mathrm{E}Y$ and $\mathrm{Av}(Y)$ is proportional to the covariance between a random variable and its inverse density or mass function. It is this covariance that we can conceptualize as an inherent bias within a probability distribution. Roughly put, the more a random variable is correlated with its own (inverse) density or mass function, the farther its expected value will be from its average value. This correlation also imparts deviations from uniformity and skews the shape of the distribution.

**Proposition 1.** *Let $Y = g(X_1, \ldots, X_k)$ be a bounded random variable and construct an indicator variable $1_R$ to designate when $(x_1, \ldots, x_k) \in \mathcal{S}$. Then $Av_\mathbf{x}\{g(X_1, \ldots, X_k)\} = E\{g(X_1, \ldots, X_k)\} + R_\mathbf{x}^{-1} \sigma_{g(X_1, \ldots, X_k), 1_R f(X_1, \ldots, X_k)^{-1}}$. Similarly, $Av(Y) = EY + R^{-1} \sigma_{Y, 1_{R_Y} f^{-1}(Y)}$.*

*Proof.* Suppose the premises.

$$E\{g(X_1, \ldots, X_k) f(X_1, \ldots, X_k)^{-1} 1_R\} - E\{g(X_1, \ldots, X_k)\} E\{1_R f^{-1}(X_1, \ldots, X_k)\} =$$
$$\int_\mathcal{S} g(x_1, \ldots, x_k) d(x_1, \ldots, x_k) - E\{g(X_1, \ldots, X_k)\} R_\mathbf{x} \implies$$
$$R_\mathbf{x}^{-1} \sigma_{g(X_1, \ldots, X_k), 1_R f(X_1, \ldots, X_k)^{-1}} = Av_\mathbf{x}\{g(X_1, \ldots, X_k)\} - E\{g(X_1, \ldots, X_k)\}$$

Rearrangement supplies the result. The proof is exactly analogous for the second identity. □

As an important side note, functional averages with respect to (w.r.t.) different spaces are not equal in general. The following proposition states when they are.

**Corollary 1.** *$Av_\mathbf{x}\{g(X_1, \ldots, X_k)\} = Av(Y)$ if and only if $R^{-1} \sigma_{Y, 1_{R_Y} f^{-1}(Y)} = R_\mathbf{x}^{-1} \sigma_{g(X_1, \ldots, X_k), 1_R f(X_1, \ldots, X_k)^{-1}}$.*

*Proof.* This proof follows easily from Proposition 2. Only one direction will be supplied.

$$Av_\mathbf{x}\{g(X_1, \ldots, X_k)\} = Av(Y) \implies$$
$$E\{g(X_1, \ldots, X_k)\} + R_\mathbf{x}^{-1} \sigma_{g(X_1, \ldots, X_k), 1_R f(X_1, \ldots, X_k)^{-1}} = EY + R^{-1} \sigma_{Y, 1_{R_Y} f^{-1}(Y)} \implies$$
$$R_\mathbf{x}^{-1} \sigma_{g(X_1, \ldots, X_k), 1_R f(X_1, \ldots, X_k)^{-1}} = R^{-1} \sigma_{Y, 1_{R_Y} f^{-1}(Y)}$$

□

A special case of when Corollary 1 holds is when $Y = \sum_i^k X_i$ and it is also the case that $Y$ and each $X_i$ are $\mathcal{U}$ random variables, and the joint support of $\sum_i^k X_i$ is equal to the Cartesian product of the marginal supports. We prove the most important part of this special case. This will then be used to show that if $Y \in \mathcal{U}$, $cY \in \mathcal{U}$ and $Y + c \in \mathcal{U}$ for all $c \in \mathbb{R}$.

**Proposition 2.** *Let $Y = \sum_i^k X_i$. Say $\mathcal{S}$ is the support of $\sum_i^k X_i$ and $\mathcal{S}_i$ is the support of $X_i$ for $\forall i$. If $\mathcal{S} = \mathcal{S}_1 \times \mathcal{S}_2 \times \cdots \times \mathcal{S}_k$ and an arbitrary $X_i \in \mathcal{U}$, then $Av_\mathbf{x}(Y) = EY$.*

*Proof.* We prove the continuous case without loss of generality (WLOG).

$$Av_\mathbf{x}(Y) = \{\prod_i^k R_i\}^{-1} \int_\mathcal{S} \{\sum_i^k X_i\} dx_1 \cdots dx_k = \sum_i^k Av(X_i) = \sum_i^k EX_i = EY$$

□

**Proposition 3.** *Suppose $Y \in \mathcal{U}$ and let $c \in \mathbb{R}$. Then $cY \in \mathcal{U}$ and $Y + c \in \mathcal{U}$.*

*Proof.* Let $c \in \mathbb{R}$ be arbitrary. Since expectations are linear, $E(cY) = cEY = cAv(Y)$. Since integrals and sums are linear operators, $Av(cY) = R^{-1} \int_\mathcal{S} cY dy = c \cdot R^{-1} \int_\mathcal{S} Y dy = cAv(Y)$, which completes this part of the proof WLOG.

Now, say $Z = Y + c$. Since $c$ is a constant, $Av_\mathbf{x}(Z) = Av(Y) + c = EZ$ by Proposition 2. Note that $\mathcal{S}_Z$, the support of $Z$, is just a shift of $\mathcal{S}$, the support of $Y$. Hence, for the discrete case, $|\mathcal{S}_Z| = |\mathcal{S}| = R$. Then, for $c_i \in \mathcal{S}$, $Av(Z) = R^{-1} \sum_i^R \{c_i + c\} = Av(Y) + c = Av_\mathbf{x}(Z) = EZ$. The continuous case is analogous. We can let $\mathcal{S}$ be a union of $p$ intervals of real numbers and say $R_i = M_i - m_i$ is the length of the $i$th interval. Then $Av(Y) = \{\sum_i^p R_i\}^{-1} \{\sum_i^p 2^{-1} R_i (M_i + m_i)\}$. Since $Z$ is simply a location shift:

$$Av(Z) = \{\sum_i^p R_i\}^{-1} \int_{\mathcal{S}_Z} z dz = \{\sum_i^p R_i\}^{-1} \{\sum_i^p 2^{-1} R_i (M_i + m_i + 2c)\} = \{\sum_i^p R_i\}^{-1} \{\sum_i^p 2^{-1} R_i (M_i + m_i)\} + c = Av(Y) + c$$

□

The next sequence of results establish equivalent definitions of $\mathcal{U}$ status under the additional assumption of regularity, i.e., that the support of the random variable is a single interval of real numbers or a set of integers $\{m, m + 1, m + 2, \ldots, M - 1, M\}$. We will condition on $1_R$ to relieve some notation burden.

**Lemma 1.** *Let $Y$ be a regular random variable defined on support $\mathcal{S} = [m, M]$ with cumulative distribution function $F(y)$ and survival function $S(y)$. Say $k \neq 0$ is an arbitrary real number. Then:*

1. $\int_m^M F(y)^k dy = \sigma_{F(Y)^k, f(Y)^{-1}} + R \cdot E\left(F(Y)^k\right)$
2. $\int_m^M S(y)^k dy = \sigma_{S(Y)^k, f(Y)^{-1}} + R \cdot E\left(S(Y)^k\right)$
3. $\sigma_{F(Y), f(Y)^{-1}} = R^{-1} \sigma_{Y, f(Y)^{-1}}$

*Proof.* Statement 1 and 2 are a direct application of Proposition 1.

For statement 3, apply the special case s.t. $k = 1$ to statement 1:

$$\int_m^M F(y) dy = \sigma_{F(Y), f(Y)^{-1}} + 2^{-1} R \implies$$
$$M - EY = \sigma_{F(Y), f(Y)^{-1}} + 2^{-1} R \implies$$
$$M - 2^{-1} M + 2^{-1} m - EY = \sigma_{F(Y), f(Y)^{-1}} \implies$$
$$\text{Av}(Y) - EY = \sigma_{F(Y), f(Y)^{-1}}$$

The last statement immediately follows by again applying Proposition 1. □

**Proposition 4.** *Suppose $Y$ is a regular, continuous random variable with support $\mathcal{S} = [m, M]$, CDF $F(y)$, survival function $S(y) = 1 - F(y)$, and define $\epsilon = Y - EY$. Then the following statements are equivalent:*

1. $Y \in \mathcal{U}$
2. $\sigma_{Y, f(Y)^{-1}} = 0$
3. $M - EY = EY - m$
4. $\int_m^M \{F(y) - S(y)\} dy = 0$
5. $\sigma_{F(Y), f(Y)^{-1}} = 0$
6. $F(Y) \in \mathcal{U}$
7. $\int_{m-EY}^{M-EY} \epsilon^k d\epsilon = 0$ *for all odd integers $k$*

*Proof.* Equivalences between statements 1, 2, 3, 5, and 6 follow easily from Proposition 1, the definition of being $\mathcal{U}$ class, and regularity since $\text{Av}(Y) = 2^{-1}(M + m)$, and Lemma 1. To see how equivalences with statements 3 and 4 follow:

$$EY = 2^{-1}(M + m) \implies$$
$$2EY = M + m \implies$$
$$EY - m = M - EY$$

However, the last line is true if and only if $\int_m^M F(y) dy = \int_m^M S(y) dy$. This can easily be shown by applying integration by parts to each side of the equality. For the last statement, note that statements 1 through 6 imply that $\epsilon^k$ is an odd function integrated over an interval that is symmetric about zero. For the opposite implication, assuming statement 7 implies that $(M - EY)^s = (EY - m)^s$ for some even integer $s$. Since $M \geq EY$ and $EY \geq m$, it is then true that $M - EY = EY - m$. This completes the proof.

□

The next two propositions are for integer values random variables. They can be generalized, but this is not accomplished here.

**Proposition 5.** *Suppose $Y$ is an integer valued random variable with $S = \{1, 2, 3, \ldots, M\}$, CDF $F(y)$, and survival function $S(y) = 1 - F(y)$. Then the following statements are equivalent:*

1. $Y \in \mathcal{U}$
2. $\sigma_{Y, f(Y)^{-1}} = 0$
3. $\sum_i^M F(i) = EY$
4. $\sum_i^M \{F(i) - S(i)\} = 1$
5. $Av\{F(Y)\} = 2^{-1}(M^{-1} + 1)$

*Proof.* The first two equivalences follow from Proposition 1 and definition. By Gauss' summation identity, $\sum_i^M i = 2^{-1}M(M+1)$. Hence, $Av(Y) = 2^{-1}(M+1)$. This implies that $M + 1 = 2EY$, which implies statement 3 via the fact that $\sum_i^M F(i) = M + 1 - \sum_i^M i f(i)$. Assuming statement 3 then implies statements one and two, also by this last identity. For statement 4, suppose $Y \in \mathcal{U}$. Then:

$$\sum_i^M \{F(i) - S(i)\} = 2\sum_i^M F(i) - M$$
$$= 2EY - M$$
$$= 1$$

Now, suppose $\sum_i^M \{F(i) - S(i)\} = 1$. Then:

$$1 = M + 2 - 2EY \implies$$
$$EY = 2^{-1}(M + 1)$$

For the last statement:

$$M^{-1} \sum_{j=1}^M F(j) = 2^{-1}(M^{-1} + 1) \implies$$
$$M^{-1} \sum_{j=1}^M 2F(j) - 1 = M^{-1} \implies$$
$$\sum_{j=1}^M 2F(j) - M = 1 \implies$$
$$\sum_{j=1}^M \{2F(j) - 1\} = 1 \implies$$
$$\sum_{j=1}^M \{F(j) - S(j)\} = 1 \implies$$
$$Y \in \mathcal{U}$$

One reverses the direction for the other equivalences.

□

The next three propositions bound the distance between the expected and average values.

**Proposition 6.** *Let $Y = g(X_1, X_2, \ldots, X_k)$ be an measurable function of bounded random variables. Then:*

$$|E(Y) - Av_{\mathbf{x}}(Y)| \leq \sqrt{Av_{\mathbf{x}}(Y^2)} \sqrt{R_{\mathbf{x}} \cdot E\{f(X_1, X_2, \ldots, X_k)\} - 1}$$

$$|E(Y) - Av(Y)| \leq \sqrt{Av(Y^2)} \sqrt{R \cdot E\{f(Y)\} - 1}$$

4

*Proof.* We prove the case for continuous random variables w.r.t. **x** WLOG.

$$\int_S g(x_1, x_2, \ldots, x_k) f(x_1, x_2, \ldots, x_k) d\mathbf{x} - R_\mathbf{x}^{-1} \int_S g(x_1, x_2, \ldots, x_k) d\mathbf{x}$$

$$= \int_S g(x_1, x_2, \ldots, x_k)(f(x_1, x_2, \ldots, x_k) - R_\mathbf{x}^{-1}) d\mathbf{x} \implies$$

$$|\mathrm{E}(Y) - \mathrm{Av}_\mathbf{x}(Y)| \leq \sqrt{\int_S g(x_1, x_2, \ldots, x_k)^2 d\mathbf{x} \int_S \{f(x_1, x_2, \ldots, x_k) - R_\mathbf{x}^{-1}\}^2 d\mathbf{x}}$$

Now, for the right-hand side of the inequality:

$$\sqrt{\int_S g(x_1, x_2, \ldots, x_k)^2 d\mathbf{x} \int_S \{f(x_1, x_2, \ldots, x_k) - R_\mathbf{x}^{-1}\}^2 d\mathbf{x}}$$

$$= \sqrt{R_\mathbf{x} \mathrm{Av}_\mathbf{x}(Y^2)} \sqrt{\int_S f(x_1, x_2, \ldots, x_k)^2 d\mathbf{x} - 2R_\mathbf{x}^{-1} + R_\mathbf{x}^{-1}}$$

$$= \sqrt{\mathrm{Av}_\mathbf{x}(Y^2)} \sqrt{R_\mathbf{x} \cdot \mathrm{E}\{f(X_1, X_2, \ldots, X_k)\} - 1}$$

□

**Proposition 7.** *Let $Y \sim f(y)$ be continuous and regular with density $f(y)$. Suppose $\min_{y \in S}\{f(y)\} = m_f > 0$ and say $\max_{y \in S}\{f(y)\} = M_f$. Then:*

1. $|EY - Av(Y)| \leq (2\sqrt{3})^{-1} \sqrt{R} \sqrt{Av\{f^{-1}(Y)\} - R}$
2. $|EY - Av(Y)| \leq (4\sqrt{3})^{-1}\{m_f^{-1} - M_f^{-1}\}$

*Proof.* By Lemma 1:

$$|EY - \mathrm{Av}(Y)| = |\sigma_{F(Y), f^{-1}(Y)}|$$
$$\leq \sigma_{F(Y)} \sigma_{f^{-1}(Y)}$$
$$= (2\sqrt{3})^{-1} \sigma_{f^{-1}(Y)}$$

Statement 1 follows from the fact that $\mathrm{Var}\{f^{-1}(Y)\} = \int_m^M f^{-1}(y) dy - R^2$. Statement 2 follows from the fact that $\mathrm{Var}\{f^{-1}(Y)\} \leq 4^{-1}(m_f^{-1} - M_f^{-1})^2$ since it is a bounded random variable. □

**Proposition 8.** *Let $Y \sim f(y)$ be a regular random variable with CDF $F(y)$. Then $f(Y) \in \mathcal{U}$ implies that $Y \in \mathcal{U}$ if and only if $F(Y) \in \mathcal{U}$. If $Y$ is discrete, this is also provided that $S = \{1, 2, 3, \ldots, M\}$.*

*Proof.* First we prove the discrete case. Note that, in general, $\mathrm{E}\{F(Y)\} = 2^{-1}[1 + \mathrm{E}\{f(Y)\}]$. This can be shown by expanding the sum $\sum_{i \in S} F(y) f(y)$, which is equal to $\sum_{i \in S} f(y_i)^2 + \sum_{i<j: i,j \in S} f(y_i) f(y_j)$. Since it is then the case that $\{\sum_{i \in S} f(y_i)\}^2 = \sum_{i \in S} f(y_i)^2 + 2 \cdot \sum_{i<j: i,j \in S} f(y_i) f(y_j) = 1$, $\mathrm{E}\{F(Y)\} = 1 - \sum_{i<j: i,j \in S} f(y_i) f(y_j)$, and $\mathrm{E}\{f(Y)\} = \sum_{i \in S} f(y_i)^2$, the aforementioned identity follows.

Now, suppose $f(Y) \in \mathcal{U}$. Then $\mathrm{E}\{f(Y)\} = M^{-1} \sum_{i=1}^M f(i) = M^{-1}$. By Proposition 6, then, $|\mathrm{Av}(Y) - \mathrm{E}Y| \leq 0$, which implies that $Y \in \mathcal{U}$. Since $\mathrm{E}\{f(Y)\} = M^{-1}$, $F(Y) \in \mathcal{U}$ by Proposition 5.

The continuous case follows from Proposition 6 and Proposition 4 since $\mathrm{E}\{f(Y)\} = R^{-1}$. □

**Proposition 9.** *Let $Y \sim f(y)$ be a regular and continuous random variable. Then $f^{-1}(Y) \in \mathcal{U}$ implies that $Y \in \mathcal{U}$. Furthermore, if $Y_k = g(X_1, \ldots, X_k)$ for some function $g$ s.t. $\min_{y \in S}\{f(y)\} = m_f \to \infty$ as $k \to \infty$ and $\lim_{k \to \infty} Y_k = Z$ is bounded, then $Y_k$ converges in distribution to a $\mathcal{U}$ random variable.*

*Proof.* Both statements are consequences of Proposition 7. Since $E\{f^{-1}(Y)\} = R$, the first statement follows from Proposition 7 since $|Av(Y) - EY| \leq 0$. For the second statement, note that $|Av(Y_k) - EY_k| \leq (4\sqrt{3})^{-1} m_f^{-1}$, again by Proposition 7 and also the fact that $M_f > 0$. Say $Z = \lim_{k \to \infty} Y_k$. Then:

$$\lim_{k \to \infty} |Av(Y_k) - EY_k| = |Av(\lim_{k \to \infty} Y_k) - E(\lim_{k \to \infty} Y_k)| \leq 0$$

Hence, $Y_k$ converges in distribution to $Z \in \mathcal{U}$. □

An informal corollary is interesting. Say $Y_k$ has a density in a scale family distributions, i.e., that $Y_k \sim \sigma^{-1} f(\sigma^{-1} y_k)$, where $\sigma = \sqrt{Var(Y_k)}$ and $f(\cdot) \leq Q \in \mathbb{R}^+$. Then, if $Var(Y_k) \to 0$ as $k \to \infty$, $\min_{y_k \in \mathcal{S}}\{\sigma^{-1} f(\sigma^{-1} y_k)\} \to \infty$ and Proposition 9 applies. This is useful because—insofar as it is reasonable to assert that, say, $\bar{Y}$ is in the scale family of distributions as stated—convergence in mean square is sufficient for convergence to $\mathcal{U}$ status.

The next proposition provides insight into how regular $\mathcal{U}$ variables can be constructed. To this end, we use the fact that an arbitrary density can be expressed as $f(y) = \{\int_\mathcal{S} g(y) dy\}^{-1} g(y)$ w.r.t. an appropriate function $g$ and, just the same, an arbitrary mass function can be expressed as $f(y) = \{\sum_{i \in \mathcal{S}} g(y_i)\}^{-1} g(y)$ for an appropriate function.

**Corollary 2.** *Suppose $Y \sim f(y)$ is regular with CDF $F(y)$ and survival function $S(y)$. Suppose $f(y) = \{\int_\mathcal{S} g(y) dy\}^{-1} g(y)$ and $G(y) = \int_m^y g(t) dt$ for some function $g$. Then $Av\{G(y)\} = 2^{-1} \int_m^M g(y) dy$ if and only if $Y \in \mathcal{U}$. For the discrete case on $\mathcal{S} = \{1, \ldots, M\}$, $Av\{G(y)\} = 2^{-1}(R^{-1} + 1) \sum_{i=1}^R g(i)$ if and only if $Y \in \mathcal{U}$, where $f(y) = \{\sum_{i \in \mathcal{S}} g(y_i)\}^{-1} g(y)$ and $G(y) = \sum_{i=1}^y g(i)$.*

*Proof.* The proof is largely omitted. It follows from the re-expression of the mass function or density and Propositions 5 and 4 respectively. □

We provide two simple examples of functions that generate continuous $\mathcal{U}$ distributions before moving forward.

*Example 1.* Set $g(y) = 1$. Then $G(y) = (y - m)$ and $Av\{G(y)\} = R^{-1}\{R2^{-1}(M + m) - Rm\} = 2^{-1} R = 2^{-1} \int_m^M 1 dy$. This generates a $Unif(m, M)$ random variable.

*Example 2.* Set $g(y) = sin(y)$ on $[0, \pi]$. Then $G(y) = 1 - cos(y)$ and $Av\{G(y)\} = \pi^{-1}\pi = 1 = 2^{-1} 2 = 2^{-1} \int_0^\pi sin(y) dy$. Then $f(y) = 2^{-1} sin(y)$ generates a $\mathcal{U}$ distribution.

## 2.1. Decomposition of Continuous Uniform Random Variables

This next proposition is theoretically interesting and useful. It essentially states that all continuous uniform distributions can be decomposed into a type of $\mathcal{U}$ variable and error such that the random error is uncorrelated with all functions of that $\mathcal{U}$ variable.

**Theorem 1** (Uniform Decomposition Theorem). *Consider $U \sim Unif(m, M)$ on $\mathcal{S} = [m, M]$. Suppose $Y \in \mathcal{U}$ is an arbitrary regular random variable that has a unimodal distribution on $\mathcal{S}$. Furthermore, suppose $\inf_{y \in \mathcal{S}}\{f(y)\} \leq Av\{f(Y)\}$ s.t. $\inf_{y \in \mathcal{S}}\{f(y)\}$ occurs in the left tail. Then $U = Y + \epsilon$ for some random variable $\epsilon$ s.t. $E(\epsilon | Y) = 0$.*

*Proof.* Observe that $Av\{f(Y)\} = R^{-1}$, which is the density of $U$. Since $Y$ has a unimodal distribution on the same support as $U$, their cumulative distribution functions possess the single-crossing property, i.e., that there is some point $c \in \mathcal{S}$ s.t. $F(y) \leq G(y)$ for $y \in [m, c]$ and $G(y) \leq F(y)$ for $y \in [c, M]$, where $G(y)$ is the uniform CDF. Since $\inf_{y \in \mathcal{S}}\{f(y)\} \leq R^{-1}$ in the left tail, this implies that $\int_m^{\tilde{c}} F(y) dy \leq \int_m^c G(y) dy$ for this $c$ in the support.

Now, since $Y \in \mathcal{U}$, $EY = E_U Y$. Now, for contradiction, suppose there exists a $t \in \mathcal{S}$ s.t. $\int_m^t F(y)dy > \int_m^t G(y)dy$. We know from the above that $t > c$. We now follow the implied logic:

$$\int_m^t F(y)dy > \int_m^t G(y)dy \implies$$
$$M - EY - \int_t^M F(y)dy > M - EU - \int_t^M G(y)dy \implies$$
$$\int_t^M F(y)dy < \int_t^M G(y)dy$$

Recall, since $F(y)$ and $G(y)$ have the single crossing property, $F(y) \geq G(y)$ for $y \in [c, M]$. Hence, since $t > c$, $\int_t^M F(y)dy \geq \int_t^M G(y)dy$. Therefore, a contradiction is reached and $\int_m^t F(y)dy \leq \int_m^t G(y)dy$ for $\forall t \in \mathcal{S}$.

By the Rothschild-Stiglitz theorem, then, this implies the existence of a random variable $\epsilon$ s.t. $U = Y + \epsilon$ and $E(\epsilon|Y) = 0$. □

**Corollary 3.** *Let $Y$ be a unimodal $\mathcal{U}$ random variable defined on $\mathcal{S} = [m, M]$ and suppose $\inf_{y \in \mathcal{S}}\{f(y)\} \leq Av\{f(Y)\}$ s.t. $\inf_{y \in \mathcal{S}}\{f(y)\}$ occurs in the left tail. Then $Var(Y) \leq 12^{-1}R^2$.*

*Proof.* Suppose the premises. Then by Theorem 1, $U = Y + \epsilon$, where $U \sim Unif(m, M)$ and $\epsilon$ is some random variable s.t. $E(\epsilon|Y) = 0$. Therefore, $Var(U) = Var(Y) + Var(\epsilon)$. This means that $12^{-1}R^2 \geq Var(Y)$. □

### 2.2. $\mathcal{U}$ variables and regression

Up until this point, we have omitted the notation that is related to sampling. We did this to make things more focused and readable. Now, there is an explicit need for it. When we use $E_\delta(\cdot)$, this signifies that the expectation is taken w.r.t. the sampling distribution.

A short word will be provided on the applicability of $\mathcal{U}$ concepts to stochastic linear regressions, mostly for inference about associations. Since observational studies do not truly fix the analysis on a design matrix very often, stochastic regression concepts are sometimes more appropriate. This implies the following data generating mechanism, which is stated more generally w.r.t. $\mathbf{X}_\delta \in \mathbb{R}^{n \times p}$: $\mathbf{Y}_\delta = \mathbf{X}_\delta \boldsymbol{\beta} + \boldsymbol{\epsilon}_\delta$.

Again, strict exogeneity, i.e., that $E_\delta(\boldsymbol{\epsilon}|\mathbf{X}) = 0$, is a sufficient condition for asymptotically unbiased estimation of $\boldsymbol{\beta}$ provided other weak dependency and regularity conditions. $\mathcal{U}$ concepts can offer an alternative condition for consistent estimation that is milder in some respects. To start, we will say two functions $h(x)$ and $g(x)$ are orthogonal if $\int_{\mathcal{S}_\delta} h(x)g(x)dx = 0$ when $\mathcal{S}_\delta$ is a union of intervals of real numbers or $\sum_{\mathcal{S}_\delta} g(x_i)h(x_i) = 0$ when $\mathcal{S}_\delta$ is discrete.

The following lemma provides a general basis for constructing instrumental variables. It states that if a variable $g(Z)$ is orthogonal to the expected error function, conditional on $Z$, then $g(Z)$ weighted by its inverse (mass) density is a valid instrumental variable.

**Lemma 2.** *Let $\epsilon$ be an arbitrary random variable s.t. $E_\delta \epsilon = 0$ and say $Z \sim f_\delta(z)$ is an arbitrary random variable. Then $g(z)$ and $E_\delta(\epsilon|z)$ are orthogonal if and only if $E_\delta\{f_\delta(Z)^{-1}g(Z)\epsilon\} = 0$.*

*Proof.* Suppose the premises. We will prove the statement for continuous $Z$ WLOG.

$$\begin{aligned}
0 &= E_\delta\{f_\delta(Z)^{-1}g(Z)\epsilon\} \\
&= E_\delta[E_\delta\{f_\delta(Z)^{-1}g(Z)\epsilon\}|Z] \\
&= E_\delta\{f_\delta(Z)^{-1}g(Z)E_\delta(\epsilon|Z)\} \\
&= \int_{\mathcal{S}_Z} f_\delta(z)^{-1}g(z)E_\delta(\epsilon|z)f_\delta(z)dz \\
&= \int_{\mathcal{S}_Z} g(z)E_\delta(\epsilon|z)dz
\end{aligned}$$

□

Proposition 10 generalizes the notion of exogeneity. Further word is provided following the proof.

**Proposition 10.** *Let $\epsilon$ be an arbitrary random variable s.t. $E_\delta \epsilon = 0$ and say $Z \sim f_\delta(z)$ is an arbitrary bounded random variable. Then $E_\delta(\epsilon|Z) \in \mathcal{U}$ if and only if $E_\delta\{f_\delta(Z)^{-1}\epsilon\} = 0$.*

*Proof.* Suppose the premises. We will prove the statement for continuous $Z$ WLOG. Set $g(Z) = 1$.

($\leftarrow$) In Lemma 2, we saw that $0 = E_\delta\{f_\delta(Z)^{-1}\epsilon\} = \int_{S_Z} E_\delta(\epsilon|z) dz$. This, however, implies that $\text{Av}\{E_\delta(\epsilon|Z)\} = 0$. Thus, since $0 = E_\delta\{E_\delta(\epsilon|Z)\} = \text{Av}\{E_\delta(\epsilon|Z)\}$, it is implied that $E_\delta(\epsilon|Z) \in \mathcal{U}$.

($\rightarrow$) Suppose $E_\delta(\epsilon|Z) \in \mathcal{U}$. Then it is true that $0 = R_Z^{-1} \int_{S_Z} E_\delta(\epsilon|z) dz$ implies $\int_{S_Z} E_\delta(\epsilon|z) dz = E_\delta\{f_\delta(Z)^{-1}\epsilon\} = 0$. □

The typical exogenous condition states that $E_\delta(\epsilon|Z) = 0$. This stipulation is recognizable as a special case of Proposition 10 since if $E_\delta(\epsilon|Z) = 0$, it is trivially true that $E_\delta(\epsilon|Z) \in \mathcal{U}$. Hence, Lemma 2 and Proposition 10 offer another route to identifying and consistently estimating regression parameters, although this is not explored further here.

Next, recall that C1 stipulates that a sampling process has preserved the support of a random variable and $Y_\delta$ is a *sampled* random variable, i.e., its probability distribution is a sample distribution. The next proposition is useful to know since it establishes that preserving the intrinsic linear bias of a population distribution during sampling is sufficient and necessary for population moment identification, provided the support has been preserved.

**Proposition 11.** *Denote a random variable $Y_\delta$ and let $g$ be some measurable function such that C1 holds. Then $E_\delta\{g(Y)\} = E\{g(Y)\}$ if and only if $\sigma_{g(Y_\delta), 1_R f_\delta^{-1}(Y)} = \sigma_{g(Y), 1_R f^{-1}(Y)}$.*

*Proof.* Suppose C1. Then $\text{Av}\{g(Y_\delta)\} = \text{Av}\{g(Y)\}$.

($\leftarrow$) Suppose $\sigma_{g(Y_\delta), 1_R f_\delta^{-1}(Y)} = \sigma_{g(Y), 1_R f^{-1}(Y)}$. Then:

$$\text{Av}\{g(Y_\delta)\} = \text{Av}\{g(Y)\} \implies$$
$$E_\delta\{g(Y)\} + R^{-1}\sigma_{g(Y_\delta), 1_R f_\delta^{-1}(Y)} = E\{g(Y)\} + R^{-1}\sigma_{g(Y), 1_R f^{-1}(Y)} \implies$$
$$E_\delta\{g(Y)\} = E\{g(Y)\} + R^{-1}\sigma_{g(Y), 1_R f^{-1}(Y)} - R^{-1}\sigma_{g(Y_\delta), 1_R f_\delta^{-1}(Y)} \implies$$
$$E_\delta\{g(Y)\} = E\{g(Y)\} + 0$$

($\rightarrow$) This case is analogous in logic and is thus omitted. □

An easy corollary is that $\sigma_{g(Y), E\{\delta_i|g(Y)\}} = 0$ if and only if the intrinsic linear bias of a distribution is preserved, conditional on the support of the distribution remaining intact under sampling. Since many different sample distributions of $Y$ exist that meet these conditions, it substantiates a much milder condition for unbiased moment estimation than non-informative sampling.

## 3. Plug-ins and Model Diagnostics

Concentration inequalities employed for inference in the manuscript ultimately require $R_i = M_i - m_i$ to be known or estimable. It is known that $Y_{(n)} - Y_{(1)}$ (the maximum minus the minimum order statistic) is a consistent estimator when statistical dependencies are local or restricted, i.e., when some version of weak mixing applies, or when outcome variables become approximately independent w.r.t. some conception of distance [1]. We do not suppose this here. Hence, we require a different justification for using extreme order statistics.

The example analysis of the paper made use of stochastic linear regression w.r.t. the model $\mathbf{Y} = \mathbf{X}\beta + \epsilon$. Like in the paper, say $(\mathbf{X}^\top\mathbf{X})^{-1}\mathbf{X}^\top = \mathbf{W}$ and $\mathbf{B} = \mathbf{WY}$ is a an estimator of $\beta$ s.t. $E(\mathbf{W}\epsilon) \to \mathbf{0}$ as $n$ grows arbitrarily large. Also consider $\{W_{s,i}\hat{e}_i\}_{i \in I}$, where $I = \{1, \ldots, n\}$, $\hat{e}_i = Y_i - \mathbf{X}_i\mathbf{B}$, and $W_{s,i}$ is cell $(s, i)$ of $\mathbf{W}$. Ultimately, we are interested in the properties of $\{W_{s,i}\epsilon_i\}_{i \in I}$. When an arbitrary $\mu_{s,n} = o(n)$, $\mathbf{B} \xrightarrow{p} \beta$ and hence $\hat{e}_i \xrightarrow{p} \epsilon_i$ as $n \to \infty$. When this is the case, $W_{s,i}\hat{e}_i = W_{s,i}\epsilon_i + o_p(1)$ s.t. $n \cdot o_p(1) \to 0$ as $n \to \infty$. Since we are assuming strict stationarity, we can let $s$ be arbitrary and say $Z_i = W_{s,i}\epsilon_i$ to reason more generally.

The next proposition offers an asymptotic justification of using $Z_{(n)}$ WLOG.

**Proposition 12.** *Consider $\{Z_i\}_{i \in I}$ s.t. each $Z_i \sim F_i(z)$ is a continuous random variable on support $\mathcal{S}_i$ and $\max_{i \in I}(\mathcal{S}_i) = M$. Moreover, say $Z_{(n)} = \max(Z_1, \ldots, Z_n)$ has a density supported on $\mathcal{S}_n \subseteq \mathcal{S} = [-M, M]$ s.t. $\max(\mathcal{S}_n) = M$ for $\forall n \in \mathbb{N}$. Now, let $Z_{\mathcal{A}_i} \sim F_i(z | Z_{i-1} \leq z, Z_{i-2} \leq z, \ldots, Z_1 \leq z)$ be well-defined under the convention that $Z_{\mathcal{A}_1} \sim F_1(z)$. Denote the sequence of conditional CDFs $(F_n\{z | Z_{n-1} \leq z, \ldots, Z_1 \leq z\}, F_{n-1}\{z | Z_{n-2} \leq z, \ldots, Z_1 \leq z\}, \ldots, F_2\{z | Z_1 \leq z\}, F_1\{z\})$ as $\mathcal{F} = (F_i\{z | \mathcal{A}_i\})_{i \in I}$. Suppose (1) there exists an absolutely continuous random variable $Q$ on support $\mathcal{S}_Q \supseteq \mathcal{S}$ s.t. $\max(\mathcal{S}_Q) = M$ and, for all increasing functions $g$, $E\{g(Z_{\mathcal{A}_i})\} \leq E\{g(Q)\}$ for at least $k(n) \in \mathbb{N}$ of the conditional distributions in the sequence $\mathcal{F}$, or (2), it is true that $1_{z<M} F_i(z | \mathcal{A}_i) < 1$ for all $z \in \mathcal{S}$ for at least $k(n) \in \mathbb{N}$ of the distribution functions of $\mathcal{F}$. Finally, suppose that $k(n) \to \infty$ as $n \to \infty$. Then $Z_{(n)} \xrightarrow{p} M$ as $n \to \infty$.*

*Proof.* First, note for any random variable $X$ on $\mathcal{S}_X = [m, M]$ s.t. $-M \leq m$, it is true that $\int_{-M}^{M} \{2^{-1} - F(x)\} dx = EX$. Now, note that $F(Z_{(n)} \leq z) = \Pr(Z_1 \leq z, Z_2 \leq z, \ldots, Z_n \leq z) = F_n(z | Z_{n-1} \leq z, \ldots, Z_1 \leq z) \cdot F_{n-1}(z | Z_{n-2} \leq z, \ldots, Z_1 \leq z) \cdots F_2(z | Z_1 \leq z) \cdot F_1(z)$ for an arbitrary $z \in \mathcal{S}$.

We start with (1). Say there exists a random variable $Q \sim F_Q(z)$ on $\mathcal{S}_Q \supseteq \mathcal{S}$ s.t. $\max(\mathcal{S}_Q) = M$. Furthermore, assume that $E\{g(Z_{\mathcal{A}_i})\} \leq E\{g(Q)\}$ for all increasing functions $g$ for at least $k(n) \in \mathbb{N}$ conditional distributions in the sequence $\mathcal{F}$. We collect the indexes of these conditional distributions into a set $W \subseteq I$. It is then implied that $F_w(z | Z_{w-1} \leq z, \ldots, Z_1 \leq z) \leq F_Q(z)$ for all $z \in \mathcal{S}_Q$ and $w \in W$ since $E\{g(Z_{\mathcal{A}_w})\} \leq E\{g(Q)\}$ for all increasing functions $g$ if and only if $Z_{\mathcal{A}_w}$ first-order stochastically dominates $Q$. Hence, for all $z \in \mathcal{S}_Q$:

$$F(Z_{(n)} \leq z) = F_n(z | Z_{n-1} \leq z, \ldots, Z_1 \leq z) \cdot F_{n-1}(z | Z_{n-2} \leq z, \ldots, Z_1 \leq z) \cdots F_2(z | Z_1 \leq z) \cdot F_1(z)$$
$$\leq 1^{n-k(n)} \cdot \{F_Q(z)\}^{k(n)}$$
$$\leq \{F_Q(z)\}^{k(n)}$$

From the above, we know that $EZ_{(n)} = \int_{-M}^{M} \{2^{-1} - F(Z_{(n)} \leq z)\} dz$. Since $F(Z_{(n)} \leq z) \leq \{F_Q(z)\}^{k(n)}$ for all $z \in \mathcal{S}$, it follows that $2^{-1} - \{F_Q(z)\}^{k(n)} \leq 2^{-1} - F(Z_{(n)} \leq z)$ for all $z \in \mathcal{S}$ and:

$$\int_{-M}^{M} \{2^{-1} - \{F_Q(z)\}^{k(n)}\} dz \leq EZ_{(n)} \leq M$$

Since $1_{z<M} \{F_Q(z)\}^{k(n)}$ is discontinuous only at $\{M\}$, which has measure zero, $\int_{-M}^{M} 1_{z<M} \{F_Q(z)\}^{k(n)} dz = \int_{-M}^{M} \{F_Q(z)\}^{k(n)} dz$ by the Riemann-Lebesgue theorem. Moreover, observe that $1_{z<M} \{F_Q(z)\}^{k(n)} \to 0$ uniformly on $\mathcal{S}$ as $n \to \infty$ and hence $k(n) \to \infty$ and it is also true that $\int_{-M}^{M} 1_{z<M} \{F_Q(z)\}^{k(n)} dz < \infty$ for $\forall n$. Note that uniform convergence follows since $1_{z<M} \{F_Q(z)\}^{k(n)} = \{1_{z<M} F_Q(z)\}^{k(n)}$ and $1_{z<M} F_Q(z) < 1$ for all $z$ considered. Thus, we can interchange the order of limits and integration and:

$$M - \int_{-M}^{M} \lim_{n \to \infty} 1_{z<M} \{F_Q(z)\}^{k(n)} dz \leq \lim_{n \to \infty} EZ_{(n)} \leq M$$

Finally, this of course implies the following statement:

$$M - 0 = M \leq \lim_{n \to \infty} EZ_{(n)} \leq M$$

Hence, by the Squeeze theorem, $EZ_{(n)} \to M$ as $n \to \infty$. From here, since we know that $\max(\mathcal{S}_n) = M$ for all $n \in \mathbb{N}$, and $EZ_{(n)} \to M$ as $n \to \infty$, it is then implied that $F(Z_{(n)} \leq z)$ converges to a distribution function s.t. $\Pr(Z_{(n)} < M) = 0$ and $\Pr(Z_{(n)} \geq M) = 1$ on limiting support $\mathcal{S}_n = \{M\}$. Hence, $Z_{(n)} \xrightarrow{d} M$, which implies that $Z_{(n)} \xrightarrow{p} M$ since $M$ is a constant.

We now commence by (2). Again construct the set $W \subseteq I$ s.t. for an arbitrary $w \in W$, $1_{z<M} F_w(z | \mathcal{A}_w) < 1$ for all $z \in \mathcal{S}$. Furthermore, say $F_*(z_* | \mathcal{A}_*) = \max_{w \in W, z \in \mathcal{S}} \{1_{z<M} F_w(z | \mathcal{A}_w)\}$. Then $1_{z<M} \cdot F(Z_{(n)} \leq z) \leq \{F_*(z_* | \mathcal{A}_*)\}^{k(n)}$ and the proof proceeds exactly the same. □

Pertinently, the condition that a dominating expectation exists for all increasing functions $g$ is relatively mild. This is fulfilled, for instance, if there exists some (possibly unknown) random variable $Q$ on the specified support s.t. $Q \leq Z_{\mathcal{A}_i}$ with probability one for a sufficient amount of the conditional functions. Stated informally, this condition,

or the second one, will be fulfilled when the dependencies that are present don't eliminate $M$ from the support of at least $k$ of the conditional distributions and $k$ still goes to infinity with sample size. This is a very mild condition because it can still be true even when every single outcome variable is statistically dependent. Pertinently, since $Z_{(1)} = -\max(-Z_1, \ldots, -Z_n)$, Proposition 12 can automatically be applied to sample minimums as well. This is also important since it establishes that $2^{-1}(Z_{(1)} + Z_{(n)})$ is a consistent estimator of the mid-range. If $Z \in \mathcal{U}$, it is also a consistent estimator of $EZ$, and even possibly when $\bar{Z}$ is not.

For our paper, Proposition 12 can be applied to $\hat{Z}_i = W_{s,i}\hat{e}_i$ under the auspices of its additional premises when it is also true that $W_{s,i}\hat{e}_i$ converges in probability to $W_{s,i}\epsilon_i$. This approach also works for fixed regressions with very little difference. For this case, we simply require that $\mu_n = o(n)$ for $\{\epsilon_i\}_{i \in I}$ and for Proposition 12 to apply to $\{Y_i\}_{i \in I}$ w.r.t. their marginal distribution, and for the model to be validly specified for its first moments.

Lastly, we made use of the empirical CDF for checking the $\mathcal{U}$ status of $W_{s,i}\epsilon_i$. Since the empirical CDF is an arithmetic average of indicator functions, we actually require $\mu_D = o(n)$ for consistency. This is because two indicator variables are uncorrelated if and only if they are independent. Consistency follows from the propositions in the main document. This does not mean that the empirical CDF is useless when $\mu_D = O(n)$. In this case, each estimator is still unbiased. Therefore, although it is not likely to be consistent, it can still be used to gauge the veracity of the assumption. If a $\mathcal{U}$ shape is visually present, this is still adequate evidence that the assumption is at least approximately true. This is because the probability of observing such a shape should be fairly low when the assumption is not met. In a situation where convergence to probability does not exist, there should naturally be more variability present. In short, we would expect to see shapes that deviate from our expectation more often than not.


\fi

\end{document}